\documentclass[11pt]{article}
\usepackage{amsfonts,amssymb,latexsym,amsmath,amsthm}
\usepackage{eucal}
\usepackage{geometry}
\usepackage{hyperref}
\usepackage[utf8]{inputenc}

\usepackage{tikz}
\usepackage{graphicx}
\usepackage{amsfonts}

\usepackage{xcolor}
\usepackage{lipsum}
\usepackage{upgreek}
\usepackage[draft]{todonotes}

\usepackage[english]{babel}
\usepackage{ifthen}
\usepackage{graphicx}
\usepackage{tikz}
\usetikzlibrary{decorations.markings}
\usetikzlibrary{plotmarks}
\usetikzlibrary{patterns}
\usepackage{pgfplots}
\usepgfplotslibrary{fillbetween}
\usepackage{makeidx}
\usetikzlibrary{calc}

\usepackage[style=numeric,
            sorting=anyt,
            backref=true,
            backrefstyle=three,
            isbn=false,
            backend=biber]{biblatex}

\newtheorem{thm}{Theorem}[section]
\newtheorem{theorem}{Theorem}

\newtheorem{prop}{Proposition}[section]
\newtheorem{coro}{Corollary}[section]
\newtheorem{lemma}{Lemma}[section]
\newtheorem{defn}{Definition}[section]
\newtheorem{remark}{Remark}[section]

\def\<{\langle}
\def\>{\rangle}

\newcommand{\R}{\mathbb{R}}

\newcommand\be{\begin{equation}} 
\newcommand\ee{\end{equation}}
\newcommand{\comment}[1]

\newcommand{\s}{\mathbb{S}}

\tikzset { domaine/.style 2 args={domain=#1:#2} }

\newtheorem{theo}{Theorem}[section]
\newtheorem*{theo*}{Theorem}

\newtheorem*{prop*}{Proposition}
\newtheorem*{lem*}{Lemma}

\newtheorem*{cor*}{Corollary}

\def\bea{\begin{eqnarray*} }

\def\eea{\end{eqnarray*} }

\begin{document}

\title{Energy in Fourth Order Gravity}
\author{R. Avalos \footnote{University of Tübingen, Mathematics Department, Tübingen, Germany.}, J. H. Lira\footnote{ Federal University of Ceará, Mathematics Department, Fortaleza, Ceará, Brazil. }, N. Marque\footnote{Université de Lorraine, Institut Elie Cartan de Lorraine, Nancy, France.}}
\maketitle

\begin{abstract}
In this paper we make a detailed analysis of conservation principles in the context of a family of fourth-order gravitational theories generated via a quadratic Lagrangian. In particular, we focus on the associated notion of energy and start a program related to its study. We also exhibit examples of solutions which provide intuitions about this notion of energy which allows us to interpret it, and introduce several study cases where its analysis seems tractable. Finally, positive energy theorems are presented in restricted situations.
\end{abstract}

\tableofcontents

%






\setcounter{MaxMatrixCols}{10}

\section{Introduction}
In the last decades there has been an increasingly rich interaction between problems arising naturally in the context of general relativity (GR) with deep problems in differential geometry and geometric analysis. Many rely on the initial value formulation of this theory. It was through the remarkable work of Y. Choquet-Bruhat that it was shown that, in most situations of interest in physics, the Einstein equations can be formulated as a hyperbolic system for initial data satisfying certain geometric constraint equations \cite{CB0} (see \cite{CB-book} for updated discussions on this topic). This led to plenty of research related both to the Einstein constraint equations (ECEs) as well as to the evolution of initial data. Both these problems are translated into geometric partial differential equation (PDE) problems. In particular, the analysis of the ECEs is intrinsically related to scalar curvature prescription problems and, for instance, their conformal formulation is intrinsically connected to the Yamabe problem in Riemannian geometry (see, for instance, \cite{Isenberg1,CB2,Maxwell1,Maxwell2,Holst1,Holst2,Maxwell3,Mazzeo1,Holst3} and references therein). Furthermore, stability issues related to both generic and special solutions are natural problems arising in physics, which have proven to be connected to stability questions in Riemannian geometry and produced rich results in geometric analysis \cite{FM,FMM,AMM,Bartnik2}. Along the same lines, producing suitable initial data through gluing techniques has proven to be a valuable tool with deep impact for both mathematics and physics \cite{IMP,Corvino1,Corvino2,Chrusciel1,Chrusciel2,Maxwell4,Chrusciel3}.

Another area of extremely fruitful interaction between these fields, and which is more directly linked to this work, is related to the analysis of conserved quantities in GR. Namely, the analysis of the so-called ADM charges, which describe the conserved total energy-momenta of isolated gravitational systems \cite{ADM,RT} (see \cite{Chrusciel4,Carlotto1,Lee} for modern reviews on this topic). Here, isolated is supposed to mean that there is a good control of the asymptotic behaviour of the fields at space-infinity, which is mathematically modelled by imposing that the ends $E_i$ of the manifold $M^n\hookrightarrow V\doteq M^n\times \mathbb{R}$ have a specific model structure, and where we have denoted by $V$ our (globally hyperbolic) $(n+1)$-dimensional space-time and by $M\cong M\times \{t=0\}$ a fixed $t=cte$ hypersurface. The simplest case, and maybe the easiest to motivate, is when the ends $E_i$ are asymptotically Euclidean (AE), which means that $E_i\cong \mathbb{R}^n\backslash \overline{B}$, where $\overline{B}$ denotes some closed ball in $\mathbb{R}^n$, and the fields are supposed to decay at infinity at specific rates. A natural problem in this context, which turned out to be highly non-trivial, is the non-negativity of the associated ADM mass of such systems. Most notably, the resolution of this problem turned out to be a fundamental result for the resolution of the Yamabe problem in the early 80's, which was open in the so called Yamabe-positive case in dimensions three, four and five, and also in the locally conformally flat case, until the remarkable work of R. Schoen in \cite{Schoen1}, where the author noted that the resolution of the positive mass conjecture in GR would imply the final resolution of the Yamabe problem in these open cases (see also \cite{Lee-Parker} for a review in this topic). 

In the above context, R. Schoen and S. T. Yau proved the positive mass theorem in the contexts needed for the resolution of the Yamabe problem in \cite{PM1,PM2}. These results led to plenty of interesting mathematical developments on their own right. In particular, E. Witten provided a proof of the same result for spin manifolds in \cite{Witten}, which itself led to plenty of mathematical research (see, for instance, \cite{Bartnik}). Furthermore, the Schoen-Yau proof, which did not demand any additional topological assumption proved to be difficult to extend to arbitrary dimensions because of its relation to minimal surfaces. This was generalised in \cite{Eichmair1} to cover the cases of dimensions $n\leq 7$ (see also references in \cite{Eichmair1}). In the last few years, this problem has continued to develop plenty of interest and at least two different proofs of the general result have been announced by Schoen-Yau \cite{Schoen2} and J. Lokhamp \cite{Lokhamp1,Lokhamp2}. Other related results concerning the analysis of the ADM energy can be found in \cite{Carlotto1,Lee}.

Let us now stress that not only the ADM energy/mass have proven to be very interesting objects from the analytic view-point, but, for instance, also the ADM center of mass (COM) has been shown to have very subtle and interesting properties. Just to name a few, it was noted by Huisken and Yau that it is related to \textit{geometric foliations of infinity} \cite{HY}, at least in special cases, and this led to several generalisations such as \cite{Metzger1,Nerz1,Nerz2,Carla1}. It should be noted that some of these works are related to physical questions on the interpretation of the COM, its dynamical properties and appropriate hypotheses that allow it to be well-defined, just to name a few.

Taking into consideration all the above history of rich connections between GR and deep mathematical problems, let  us now draw our attention to certain modifications of GR which have also been proposed in the last several decades. Namely, let us focus on higher-order gravitational theories. These are models which modify the Einstein equations by adding some higher-order modifications, which typically arise by modifying the Einstein-Hilbert action and incorporating, for instance, quadratic terms in the curvature tensor. Such modifications have a long and rich history in physics, having been explored extensively within the physics literature, and remain as objects of intensive study. In particular, in classical four-dimensional GR, addition of certain quadratic terms to GR has proven to produce theories better suited to standard approaches of quantization \cite{Hooft,Stelle1}. Furthermore, the addition of a quadratic term in the scalar curvature in the Einstein-Hilbert action is related to the so-called Starobisky inflationary model \cite{Inflation}. Also, several other motivations for the analysis of these theories can be found within effective field theory approaches to GR \cite{EF1,EF2,EF3}; low energy limits of string theory \cite{String1}; the so-called conformal gravity proposal \cite{Confgrav}; low-dimensional gravity (such as massive gravity) \cite{Massivegrav1,Massivegrav2} as well as other approaches to both classical and quantum gravity \cite{Maldacena,Kaku}. 

Since our aim is not to judge  whether or not any of the above models are successful descriptions of the associated phenomena they are meant to describe, we refer the reader to the above references and references therein for more information related to such questions. On the other hand, our intention is to explore the links that such higher derivative theories have with natural higher-order problems in geometric analysis. In particular, higher-order problems have also received plenty of attention in geometric analysis (see, for instance, \cite{Branson1,Branson2,Graham,Paneitz,Chang}) as can be seen notably in $Q$-curvature analysis (see \cite{Djadli1,Gursky1,Gursky2,Esposito,Ndiaye,Malchiodi4d,Humbert,HangYang1,HangYang2,HangYang3} and references therein). Nevertheless, to the best of our knowledge no clear connection seems to have been made between these two areas of research. In particular, based on the fruitful relation between GR and classical problems in geometry that was briefly described above, we do not think it would be surprising to find such connections. Thus, our aim is to take this paper as the starting point of a project related to the translation of certain classical problems in mathematical GR to these higher-order theories, with special emphasis in the potential connections with existing problems within geometric analysis. In particular, we will devote this first step to a detailed analysis of conservation principles in this context, which relates to very well-known literature within physics. By the end we will make contact with appropriate energy notions arising in this frame which relate to well-known $Q$-curvature positive mass theorems, which have proven essential in $Q$-curvature analysis (see \cite{Humbert,Gursky2,HangYang2,HangYang3}). This link is explored in detail in a related paper (see, \cite{avalos2021positive}).  
  
With all the above in mind, let us fix our attention to the following type of gravitational theories. Consider a globally hyperbolic space-time $(V\doteq M\times\mathbb{R},\bar{g})$, and a  functional of the form
\begin{align}\label{action}
S(\bar{g})=\int_{V}\left(\alpha R^2_{\bar{g}} + \beta\< \mathrm{Ric}_{\bar{g}},\mathrm{Ric}_{\bar{g}} \> \right)dV_{\bar{g}},
\end{align}
where $\alpha$ and $\beta$ are free parameters of the problem. In order to make sense of the above functional, let us assume that the class of metrics here considered are such that $R^{2}_{\bar{g}}$ and $\< \mathrm{Ric}_{\bar{g}},\mathrm{Ric}_{\bar{g}} \>$ are integrable. Then, the functional $\bar{g}\mapsto S(\bar{g})$ is well-defined and we have an $L^2$-gradient for this functional, given by $A_{\bar{g}}\in \Gamma(T^0_2V)$, which is explicitly given by 
\begin{align}\label{A-tensor}
\begin{split}
A_{\bar{g}}&=\beta\Box_{\bar{g}}{\mathrm{Ric}_{\bar{g}}} + (\frac{1}{2} \beta  + 2\alpha)\Box_{\bar{g}}R\: {\bar{g}} - (2\alpha +  \beta)\bar{\nabla}^2R_{\bar{g}}  + 2\beta{\mathrm{Ric}_{\bar{g}}}_{\cdot}\mathrm{Riem}_{\bar{g}}    \\
&  + 2\alpha R_{\bar{g}}\mathrm{Ric}_{\bar{g}} -\frac{1}{2 }\alpha R^2_{\bar{g}} {\bar{g}} -\frac{1}{2 } \beta\langle\mathrm{Ric}_{\bar{g}},\mathrm{Ric}_{\bar{g}}\rangle_{\bar{g}} {\bar{g}},
\end{split}
\end{align}
where the contraction ${\mathrm{Ric}_{\bar{g}}}_{\cdot}\mathrm{Riem}_{\bar{g}}$ is on the first and third indexes (see the appendix for our curvature convention). As was described in detail above, the analysis of (\ref{action}) is well-motivated within contemporary theoretical physics. We have explicitly omitted the second order term arising from the Einstein-Hilbert action, since, as was explained above, our intention is to make contact with fourth order geometric problems. Nevertheless, it is worth highlighting that models such as four-dimensional conformal gravity are contained in this analysis.

On the geometric side, the critical points of (\ref{action}) solve a set of fourth order geometric partial differential equations (PDEs), as can be seen from (\ref{A-tensor}). These equations are, in principle at least, amenable to the same geometric PDE treatment as the Einstein equations, which is to say the analysis of the Cauchy problem associated to $A_{\bar{g}}=0$. In particular, the main arguments on how the fourth order system arising from $A_{\bar{g}}=0$ can be rewritten as a larger fully coupled second order system of non-linear wave equations with constraints on the initial data have been laid out in \cite{Noakes}. Along these lines, problems such as optimum regularity and geometric uniqueness remain open, and, most importantly, the analysis of the associated constraint system remains completely open to the best of our knowledge. While we intend to address some of these problems in upcoming work, in this paper we will concentrate in the analysis of conservation principles associated to the space-time equations $A_{\bar{g}}=0$.

In view of the analysis related to conserved quantities in GR, we should stress that there are different ways of finding these  under the presence of asymptotic symmetries (see, for instance, \cite{RT,AD,Stelle2,Stelle3,ADT,Teiltelboim,Ashtekar,Wald1,Wald2} and \cite{Hollands} for a review on this topic). Nevertheless, not all these approaches translate equally well to other Lagrangian theories. One method that does, is the one described in \cite{AD,ADT} (see also \cite{CB-PM} for a more mathematically-oriented presentation). This approach relies on the analysis of small perturbations of solutions with special symmetries, and using such symmetries to produce conserved quantities of the perturbed solutions. All this depends crucially on the contracted Bianchi identities, which give rise to a linearised version of these local conservation laws. From the diffeomorphism invariance of geometric Lagrangian theories, we know that they obey a version of these local conservation principles which will allow us to follow the same path towards a good notion of energy for these theories. This kind of analysis has been exploited by several authors to study conserved quantities associated to higher-order gravitational theories (see, for instance, \cite{Massivegrav1,Massivegrav2,4thenergy1,4thenergy2,4thenergy3,4thenergy4}, and \cite{4thenergy5} for a review on this topic). 

Along the lines of the above paragraph, after some preliminary preparations, in Section 3, we will apply this construction to the space-time equations (\ref{A-tensor}). This analysis will be done taking into consideration several analytic details that, although important for our purposes, have not been fully addressed in current literature to the best of our knowledge. In particular, we will consider space-times with AE ends, but weaken traditional definitions so as to, in principle, admit more flexible asymptotic conditions in our analysis, which can be better suited to this problem. Then, given an $A$-flat metric $\hat{\bar{g}}$ possessing a Killing field $\xi$, we will see that for any $A$-flat perturbed metric $\bar{g}=\hat{\bar{g}}+h$, there is a locally conserved $1$-form $\mathcal{P}_{\hat{\overline{g}}}(\xi , h)$, such that if $\mathcal{P}$ obeys certain $L^1$-integrability conditions, then
\begin{align}
\mathcal{E}_{\alpha,\beta}(\bar{g})=-\int_M\langle \mathcal{P}_{\hat{\overline{g}}}(\xi , h),\hat{n} \rangle_{\hat{\bar{g}}}dV_{\hat{g}}
\end{align}
is conserved through evolution, where $\hat{g}$ stands for the induced Riemannian metric on $M$ by $\hat{\bar{g}}$ and $\hat{n}$ stands for the $\hat{\bar{g}}$-future pointing unit normal to $M$ (see Proposition \ref{Energyconservation}). Clearly, when $\xi$ is time-like, this becomes natural notion of energy to be attached to $\bar{g}$. In physics literature, such conserved quantities are typically expressed through some charge computed as a boundary integral at space-like infinity. Because of the local conservation law that $\mathcal{P}$ obeys, this can be seen to be the case in orientable manifolds (see Proposition \ref{ADMenergy.1}), implying the existence of a $2$-form $\mathcal{Q}_{\hat{\bar{g}}}(\xi,h)$ satisfying
\begin{align*}
\mathcal{P}_{\hat{\bar{g}}}(\xi,h)=\delta_{\hat{\bar{g}}}\mathcal{Q}_{\hat{\bar{g}}}(\xi,h),
\end{align*}
where $\delta_{\hat{\bar{g}}}=d^{*}$ stands for the co-differential operator acting on differential forms. The existence of such \textit{superpotential} is known from the work of Deser and Tekin when $(V,\hat{\bar{g}})$ is taken to be a maximally symmetric space \cite{ADT}. For our purposes this is not strong enough, and therefore through Propositions \ref{Energycharge.Eins} and \ref{pqGR} we will prove that this holds without any addition symmetry assumptions around arbitrary Einstein solutions $\hat{\bar{g}}$ which possess a Killing field $\xi$ obeying appropriate asymptotic conditions, and provide the corresponding formulae. To end this section, we will provide an explicit expression for the leading order of the energy density when computed on a perturbation $\bar{g}$ of a Ricci-flat asymptotically Minkowskian (in an appropriate sense, see Definition \ref{AMmanifolds}) solution $\hat{\bar{g}}$, which is given in rectangular coordinates near infinity by
\begin{align}\label{energy-desity}
\begin{split}
-\mathcal{Q}(\hat{n},\hat{\nu})|_{t=0}&=\left( \frac{3}{2}\beta  + 2\alpha\right) \left( \partial_{j}\partial_{i}\partial_{i}g_{aa} - \partial_{j}\partial_{u}\partial_{i}g_{u i}\right)\hat{\nu}^{j} + \frac{\beta}{2}\left( \partial_{i}\ddot{g}_{ji} - \partial_{j}\ddot{g}_{ii} \right)\hat{\nu}^{j}  \\
&+ \frac{\beta}{2}\left( \partial_{i}\partial_{j}\dot{X}_{i} - \partial_{i}\partial_{i}\dot{X}_{j} \right)\hat{\nu}^{j} + ( \beta  + 2\alpha)\partial_{j}\partial_{i}\partial_{i}N^2 \hat{\nu}^{j}- ( \beta  + 2\alpha)\partial_{j}\ddot{g}_{ii}\hat{\nu}^{j} \\
&+ 2( \beta  + 2\alpha)\partial_{j}\partial_{i}\dot{X}_{i}\hat{\nu}^{j}+ O_1(r^{-(\hat{\tau}+\tau) -3}),
\end{split}
\end{align}
where $N$ and $X$ stand for the initial values of the lapse and shift functions associated to the space-time decomposition of $\bar{g}$; $\dot{X}=\partial_tX|_{t=0}$; $\ddot{g}=\partial^2_t\bar{g}|_{t=0}$; $\hat{\nu}$ stands for outward-pointing $\hat{g}$-unit normal vector field to a hypersurface $S\hookrightarrow M$ sufficiently  far away in the ends of $M$ and $\hat{\tau}$ describes the order of decay of $\hat{\bar{g}}$ at infinity while $\tau$ controls the behaviour of $\bar{g}$ at infinity.  In the above, we simplified notations for the $2$-form $\mathcal{Q}_{\hat{\bar{g}}}(\xi, h)$ applied to the vectors $\hat{n}$ and $\hat{\nu}$ by dropping the dependancy on the model metric, the Killing field and the perturbation. It should be highlighted that this explicit ADM-type expression is very useful for our analysis, since it makes contact with suitable decaying assumptions and, more importantly, with geometric objects directly linked to these expression (for instance, see Section 6 for a clear link with $Q$-curvature). The above expressions motivates our definition of energy given in Definition \ref{Energydefn}. This definition clearly suggests several issues to be analysed, two of which are the existence of examples that provide good intuitions and the rigidity which should be associated to the vanishing of the energy (see Remark \ref{Energy-rigidty}). These two issues are explored in the following sections.

Concerning the first of the above two mentioned problems, borrowing ideas from GR, these type of intuitions are typically achieved by looking at specially simple solutions of our theory where we can interpret the results straightforwardly. In GR, this can be done by testing the appropriate notion of energy on highly symmetrical solutions modelling isolated systems, which can be computed explicitly, such as Schwarzschild, Kerr, Schwarzschild de Sitter (SdS) and anti-de Sitter (SAdS) solutions. Let us notice that in dimension four, it has been observed for instance in \cite{ADT} that under certain decay assumptions of the solutions, higher-order notions of energy do not produce any new contributions (notice that this could be deduced explicitly from (\ref{energy-desity})). Nevertheless, this depends strongly on the asymptotic behaviour of the solutions, and because of the higher-order nature of the theory, one could easily expect that in highly symmetrical cases we may find new solutions with decays naturally suited to produce contributions to these higher-order notions of energy, for instance through weaker decaying solutions which accompany new integration constants. In other words, the natural decay assumptions which are known from GR may not be the appropriate ones in this context. All this will be analysed in detail in Section 4, where we will present classifications of 4-dimensional exterior static spherically symmetric $A$-flat solutions in two complementary cases. 

The first of the above cases is for arbitrary $\alpha$ and $\beta$, but for exterior solutions \textit{in Schwarzschild form}, while the second case is for the special case of $3\alpha+\beta=0$ (which corresponds to the conformally invariant case) and without the Schwarzschild form restriction. These kinds of classifications have repeatedly appeared in physics literature (specially for the conformal case), where similar results seem to have been rediscovered several times \cite{Fiedler-Schimming,SSsols1,SSsols2,SSsols3,SSsols4}. We would like to draw the reader's attention specially to \cite{Fiedler-Schimming} which seems to contain most of the subsequent results, and is, to the best of our knowledge, the original reference. These comments in particular apply to the so-called Mannheim-Kazanas solution \cite{SSsols1}, to which we will refer to as the FSMK-solution, which has been extensively analysed in the context of conformal gravity. Nevertheless, it seems to be that the resurgences of these results happened without reference to previous results which were closely related. Thus, we will use this opportunity to compile the existing results available to our knowledge and, to the benefit of the reader, provide a self-contained independent proof which, in order to save time and space associated to long computations, will be computer assisted. In particular, we will make an analysis which is well-suited for our purposes, exploring some global aspects of these classifications which, up to the best of our knowledge have not been previously analysed explicitly. The final result of this analysis can be compiled as follows (see Proposition \ref{solscharzschgener} and Theorem \ref{Classificationthm}):
\begin{prop*}
Assume that $(V,\bar{g})$ is a 4-dimensional {$A$-flat} exterior static spherically symmetric space in Schwarzschild form. Assume further that $3\alpha + \beta \neq 0$. Then, 1) If $\beta=0$, then $\bar{g}$ is either a Schwarzschild-de Sitter  (or SAdS) metric or a Reissner-Nordstr\"om metric. 2) If $\beta \neq0$ then $\bar{g}$ is a Schwarzschild-de Sitter  (or SAdS) metric.
\end{prop*}
In order to address the $3\alpha+\beta=0$, let us first draw the reader's attention to the FSMK family of solutions, which is given by metrics of the form\footnote{We are parametrising these solutions in a convenient form for our purposes.}
\begin{equation}\label{FSMKsol}
\begin{aligned}
\hat{\bar{g}}(m,\Lambda,\mu)&=-N^2(r)dt^2+\frac{1}{N(r)^2}dr^2+r^2g_{\s^2},\\
N^2(r)&=1-3m\mu - \frac{m}{r} + \mu(3m\mu-2)r-\frac{\Lambda}{3} r^2,
\end{aligned}
\end{equation}
where $m,\mu$ and $\Lambda$ are constants parametrising the family. Let us highlight that all the static spherically symmetric Bach-flat space-time metrics in Schwarzschild form belong to the family of metrics given by (\ref{FSMKsol}). In particular, depending on the values of $m,\mu$ and $\Lambda$ these solution are defined either for $r\in (r_{-},r_{+})$, with $0<r_{-}<r_{+}<\infty$, or for all $r>r_{*}$, with $r_{*}$ depending on $m,\mu$ and $\Lambda$. The precise combination for each of these cases are given in Proposition \ref{Classificationlemma}. In particular, we can find combinations with $\Lambda=0$ which allow for exterior solution defined for all $r>r_{*}$. We will refer to these exterior solutions as $(\overset{\circ}{V},\overset{\circ}{\bar{g}})$ and write generically $\overset{\circ}{N}{}^2=c_1-\frac{m}{r}+c_2r$, for constants $c_1,c_2$ and $m\geq 0$.

Let us highlight that the constant which accompanies the linear term in $\overset{\circ}{N}$ possesses some interpretations within the context of conformal gravity, being related to a flattening effect in the rotation curves of galaxies, which, in that context, has been proposed as an alternative to conventional dark matter explanations. Interestingly enough, and aligned with previous comments concerning the asymptotics of interesting solutions to these higher-order theories, we can prove the following (see theorem \ref{FSMKlemma}).
\begin{theorem}
\label{lemmahighlighted}
The fourth order energy of the solution $(\overset{\circ}{V},\overset{\circ}{\bar{g}})$ is well-defined and given by $\mathcal{E}_{\alpha,-3\alpha}(\overset{\circ}{\bar{g}})=8\pi\alpha c_2$.
\end{theorem}

This theorem is part of a series of results presented in this paper which show that the energies $\mathcal{E}_{\alpha,\beta}$ can be nicely interpreted in several cases and positivity as well as rigidity statements seem to be attainable in different limits (see Corollary \ref{ADTcoro}, Theorem \ref{PEthm} and the discussion in Section 6).

We will close the $3\alpha+\beta=0$ classification by presenting the following results, which compiles results presented in several papers in current literature (see \cite{Fiedler-Schimming,SSsols1,SSsols2,SSsols3,SSsols4}).
\begin{theo*}
Any 4-dimensional exterior static spherically symmetric Bach-flat space-time $(V,\bar{g})$ is almost conformally Einstein. More specifically, any static spherically symmetric Bach-flat space-time is almost conformal to a subset of a Schwarzschild-de Sitter (or SAdS)  space-time or to $\bar{g}_0 = - \cos^2 x dT^2 + dx^2 + g_{\mathbb{S}^{2}}$.
\end{theo*}
In the above theorem, by \textit{almost conformal} we mean that in $(V,\bar{g})$ there may be topological spheres which separate connected regions which are globally conformal to one of the above model spaces. Let us also notice that the famous FSMK-solution is contained in the SdS/SAdS conformal families (see Proposition \ref{Classificationlemma} for details of the domains of definition depending on the values of the parameters.) 

The above two results (specially with the assistance of Proposition \ref{Classificationlemma}), show us that appealing to highly symmetrical solutions to the fourth order equations to provide good intuition for $\mathcal{E}_{\alpha,\beta}$, although useful, is quite limited. Thus, in Section 5 we will drop these symmetry assumptions and build implicit Einstein 4-dimensional metrics via the evolution of initial data which explicitly break time reversal symmetry, and do not impose any a priori spatial symmetry for the solutions. Furthermore, the kind of initial data that we will deal with is AE in a weaker sense which we refer to as $\Lambda$AE, where $\Lambda>0$ stands for some fixed cosmological constant (see Definition \ref{LamdAE}). In particular, in an appropriate sense, these initial data sets are asymptotically umbilical, which makes their decay weaker.\footnote{{Notice that the Einstein constraint equations imply that $\Lambda$-vacuum initial data sets, with $\Lambda>0$, cannot be AE according to standard definitions.}} {These kinds of initial data sets appear to us to be well motivated by physical arguments laid out in detail in Section \ref{PEthemSection}. In particular, the breaking of time-symmetry produced by the presence of a positive cosmological constant seems to be aligned with cosmological models, and the same holds true for our asymptotic umbilicity hypothesis. Thus, we regard such $\Lambda$AE initial data sets as appropriate models for isolated systems in an expanding (or contracting) cosmological background. Such solutions may deserve further analysis on their right in the context of GR and we refer the reader to the beginning of Section \ref{PEthemSection} for further details.} Let us draw the attention of the interested reader to \cite{Avalos-Lira}, where initial data construction for these types of initial data sets have been done, and, also, it has been explained how they can be used to incorporate the presence of a positive cosmological constant in accurate models such as Schwarzschild's solution. Furthermore, the analysis of the energy $\mathcal{E}_{\alpha\beta}$ over these solutions is also well-motivated by Corollary \ref{ADTcoro}. In this context, and for these specific type of solutions, we will address two questions which were posed above. Namely, what does the fourth order energy measure and the rigidity properties associated to the leading order of (\ref{energy-desity}) (see Theorem \ref{PEthm}). {Associated to these questions, below we present the main analytic results of this paper.}
\begin{theorem}\label{MainThm1}
Let $(V^4,\bar{g})$ be an Einstein space-time generated by $\Lambda$AE initial data $\mathcal{I}$ of order $\tau$ with $R_g\in L^{1}(M^3,dV_g)$. Then, the following statements follow:
\begin{enumerate}
\item If $g$ is asymptotically Schwarzschild, then the fourth-order energy (\ref{energy.2}) is well defined for general values of $\alpha$ and $\beta$. Furthermore if $\Lambda>0$; $\alpha<0$ and $\beta\geq -\frac{3}{2}\alpha$, then $\mathcal{E}_{\alpha,\beta}(\bar{g})\geq 0$. Additionally, if $R_g\geq 0$ and $\beta > - \frac{3}{2} \alpha$, then $\mathcal{E}_{\alpha,\beta}(\bar{g})=0$ iff $(M,g)\cong (\mathbb{R}^3,\cdot)$.
\item In the special case $2\alpha+\beta=0$, the fourth order energy is well-defined for $\tau>\frac{1}{2}$. If, additionally, $R_{g}\geq 0$; $\Lambda>0$ and $\alpha<0$, then $\mathcal{E}_{\alpha,-2\alpha}(\bar{g})\geq 0$ with equality holding iff $(M,g)\cong (\mathbb{R}^3,\cdot)$.
\end{enumerate}
\end{theorem}
The above theorem proves that looking for positive energy theorems and good interpretations of the fourth order energy for appropriate classes of $A$-flat spaces is a sensible program, even in dimension four. Although it is not the main theme of this paper, the potential physical implications of positive energy theorems of the above type could be an interesting topic for higher-order models of gravitational phenomena. 


{It is worth explicitly pointing out that, besides the physical motivations commented above, considering these fourth order energies on second order type solutions has been also motivated by comparison with the Willmore problem in Riemmanian geometry. Indeed, in that case, the study of the \emph{trivial} second order solutions in the fourth order context has revealed the existence of bubbling phenomena specific to the  fourth order problem \cite{bibnm3}. With this paradigm in mind, it is natural to study Einstein metrics as $A$-flat space-times to understand their behaviours with respect to these new degrees of liberty. This can explicitly be seen to be the case within Corollary \ref{BachFlatCoro}, where we analyse conformally Einstein solutions, which are known to be Bach-flat, and thus are fourth order solutions relevant for the case $3\alpha+\beta=0$. In particular, one can then see how the analysis of Theorem \ref{MainThm1} and Theorem \ref{MainThm2} provide us with enough tools to analyse the energy of these new fourth order solutions without symmetry assumptions.}

{As is usual when dealing with asymptotic charges related to geometric invariants, one would like to know that such quantities are in some appropriate sense independent of the asymptotic coordinate systems. We address this issue in the specific cases treated in the above theorem and prove an analogous invariance property to that known for the ADM energy, for instance from \cite{Bartnik}.\footnote{{Such invariance property has been analysed in other important limiting cases in \cite{avalos2021positive}}.} Let us notice that in our case there is one further subtlety associated to the explicit dependence of $\mathcal{E}_{\alpha,\beta}$ on the space-time \emph{observers}, which manifests itself in (\ref{energy-desity}) via the dependence on the $(N,X)$ initial data. Such a dependence is also known to hold for the ADM energy in GR, although in that case it might be less explicit (see, for instance, \cite[Chapter 1, Section 1.1.3]{ChruscielEnergyNotes}). In that context, the above theorem actually refers to the energy measured by a set of \emph{canonical observers} whose flow lines are orthogonal to the initial Cauchy surface. We shall therefore analyse the invariance of the energy studied in Theorem \ref{MainThm1} both for observers asymptotic to the canonical ones as well as with respect to the asymptotic coordinate systems. This leads us to the following result.}

{\begin{theorem}\label{MainThm2}
Under the same hypotheses as in Theorem \ref{MainThm1}, given a $\Lambda$AE initial data set and two asymptotic observers $\mathcal{O}_1$ and $\mathcal{O}_2$ of orders $\rho> \frac{1}{2}$, if we denote the energies associated to them by $\mathcal{E}^{(\mathcal{O}_i)}_{\alpha,\beta}$, $i=1,2$, then
\begin{align*}
\mathcal{E}^{(\mathcal{O}_1)}_{\alpha,\beta}(\bar{g})=\mathcal{E}^{(\mathcal{O}_2)}_{\alpha,\beta}(\bar{g}).
\end{align*}
Furthermore, if $\phi_x,\phi_y:M\backslash K\mapsto \mathbb{R}^3\backslash \overline{B_1(0)}$ are two asymptotic charts where $g$ is of order $\tau_x,\tau_y>\frac{n-2}{2}$ respectively, and where the general hypotheses of Theorem \ref{MainThm1} are satisfied, then the value of  $\mathcal{E}^{(\mathcal{O}_i)}_{\alpha,\beta}$ is the same for both coordinate systems.
\end{theorem}
}

{Let us comment that the proof of the above theorem requires one to analyse certain gauge conditions for the Cauchy problem in GR in a general manner (in particular, without imposing zero initial data for the shift vector). We refer the interested reader to Lemma \ref{GeneralGaugeConditions} for further useful details.}



{Finally, our last main conclusion in this paper will be the introduction of a distinguished limiting case, which deserves its own treatment.} This is the case of fourth order solutions which are \textit{stationary}. In these cases, and for the (recurrently) special choice of $2\alpha+\beta=0$, the energy associated to these solutions is given by
\begin{align*}
\begin{split}
\mathcal{E}_{\alpha,-2\alpha}(\bar{g})&=- \alpha\lim_{r\rightarrow\infty}   \int_{S^{n-1}_r} \left( \partial_{j}\partial_{i}\partial_{i}g_{aa} - \partial_{j}\partial_{u}\partial_{i}g_{u i}\right)\hat{\nu}^{j}d\omega_{r}   \end{split}.
\end{align*}
{Let us notice that the appearance and recognition of the above limiting case is a direct product of accumulated work in this paper:  it only becomes  natural  after (\ref{energy-desity}) has been found and others  have been studied. In particular, the existence of positivity/rigidity statements such as Theorems \ref{MainThm1} and \ref{MainThm2} are key to provide some evidence that the fourth order energy $\mathcal{E}_{\alpha,\beta}$ is well-behaved in these natural cases.} In this Riemannian setting, interestingly enough, the energy density is explicitly related to $\Delta_gR_g$ which is the leading order term of the $Q$-curvature associated to $g$. In a related paper,  the first two authors and P. Laurain proved a positive energy theorem for this case under natural geometric assumptions (and for $\alpha<0$), and showed how this positive energy theorem relates to $Q$-curvature analysis in very much a parallel way in which the Schoen-Yau positive mass theorems relate to scalar curvature analysis (see \cite{avalos2021positive}).

\bigskip
{\bf Acknowledgments: } The authors would like to thank the CAPES-COFECUB and CAPES-PNPD for their financial support, and Paul Laurain for insightful discussions on this topic. Most of the work on this article was done when the first author was employed by the University of Cear\'a, and the third author by the University of Potsdam. 

\bigskip
{\bf Remarks: } {Since the publication of the original preprint version, the first and third authors have worked on exploiting this fourth order energy in the static case to obtain a fourth order rigidity result linked to $Q$-curvature \cite{avalosmarquelaurainlirab}. The energy we here study has thus been shown to provide insights on the role of fourth order curvatures on a Riemannian manifold.}

\section{Preliminaries}

In this section we will collect some necessary definitions and results which will be useful in the core of this paper. 

\subsection{Analytical preliminaries}

In order to introduce conserved quantities for solutions to the equations $A_{\bar{g}}=0$, we will consider solutions which can be treated as \textit{perturbations} of some fixed solution $\hat{\bar{g}}$, where the latter possess some continuous symmetry, which induces a conservation principle. In order to obtain such conservation principles, we will appeal to the analysis of the linearised operator $DA_{\hat{\bar{g}}}:\mathcal{B}_1\mapsto \mathcal{B}_2$, associated to the tensor field $A$ seen as a map  $\bar{g}\mapsto A_{\bar{g}}\in \mathcal{B}_2$ for Lorentzian metrics $\bar{g}\in \mathcal{B}_1$, where $\mathcal{B}_1,\mathcal{B}_2$ are appropriately chosen functional spaces.

Although quite standard, the above procedure is a little bit more delicate in our situation where we are considering Lorentzian manifolds $(V^{n+1}=M^n\times\mathbb{R},\bar{g})$ with non-compact Cauchy slices $M$, and, furthermore, where we intend to admit perturbations of $\bar{g}$ with asymptotics as flexible as possible. On the one hand, computing directional derivatives $\frac{d}{d\lambda}A(\bar{g}+\lambda h)|_{\lambda=0}$ for arbitrary $h\in \mathcal{B}_1$ can become tricky for geometric objects, such as $A$, since the curve $\bar{g}_{\lambda}\doteq \bar{g}+\lambda h$ can degenerate and make them ill-defined. On the other hand, choosing natural function spaces in this case where $\bar{g}$ is indefinite can be a subtle issue.

Concerning the last of the above problems, under mild assumptions on $M$ we can naturally introduce useful norms on $V=M\times [0,T]$. Consider that $M$ is complete with a smooth Riemannian metric $e$ of bounded geometry. Then, consider the Riemannian metric $\hat{e}\doteq dt\otimes dt + e$ on $V$. Now, let $T\in \Gamma(T^i_jV)$ be a space-time tensor field. We can then decompose $T$ into a set of time-dependent space-tensors of ranks $i+j,i+j-1,\cdots,0$ which are maps $[0,T]\mapsto \Gamma(T^k_lM)$ for the appropriate $l,k$. Explicitly, let $T\in \Gamma(T_2^0V)$, decompose it as $T^{(0)}\in C^{r}([0,T];\Gamma(M\times\mathbb{R}))$, $T^{(1)}\in C^r([0,T];\Gamma(T^{*}M))$ and $T^{(2)}\in C^r([0,T];\Gamma(T^{0}_{2}M))$ given by
\begin{align*}
T^{(0)}&\doteq T(\partial_t,\partial_t),\\
T^{(1)}(X)&\doteq T (\partial_t,X) \;\; \forall \;\; X\in \Gamma(TM),\\
T^{(2)}(X,Y)&\doteq T(X,Y) \;\; \forall \;\; X,Y\in \Gamma(TM)
\end{align*}
Having a preferred norm on $(M,e)$ which controls fields at space-infinity, say $\mathcal{B}$, we can then topologise these spaces via 
\begin{align*}
||u(\cdot,t)||_{\mathcal{B}(T)}\doteq \sup_{t\in [0,T]}\sum_{k=0}^s||\partial^k_tu(\cdot,t)||_{\mathcal{B}}.
\end{align*}
This construction is classical when $\mathcal{B}$ stands for some $L^2$-Sobolev space (standard, uniformly local or maybe weighted), which are well-suited to prove well-posedness of (non-linear) wave equations (\cite{Espaces1,CB-book}).
\begin{remark}\label{remarkdecompoADM}
The above decomposition, when applied to the space-time Lorentzian metric $\bar{g}$ gives us $\bar{g}^{(1)}=X$, $\bar{g}^{(2)}=g$  and $\bar{g}^{(0)}= -\left[N^{2} - \left|X\right|_g^2 \right]$, where $g$ stands for the induced Riemannian metric on $M$ and $N$ and $X$ stand for the lapse function and shift vector fields associated to the isometric embedding $(M,g)\hookrightarrow (V,\bar{g})$. Let us recall that this lapse-shit decomposition allows us to write
{\begin{align*}
\bar{g}=-N^{2}dt^2 + g_t,
\end{align*}
where $g_t$ restricts to the induced Riemannian metric on each $M_t\cong M\times \{t\}$,
that is $g_t = \bar{g}_{ij}(dx^i+X^idt)\otimes (dx^j+X^jdt)$, and thus, for any tangent vector $v$ to $M_t$ $g(v,v)\doteq g_t(v,v)$ stands for the induced Riemannian metric.}
\end{remark}
For our purposes, we will have in mind functional spaces such as those described above, where in particular, for a smooth Lorentzian metric $\bar{g}$, we can impose controls of $N,X,g$ at space-infinity, but it is not necessary to specify them a priori. On the contrary, we will impose specific decays at space-like infinity which can be accommodated appealing to different kinds either uniform or weighted spaces. The idea is to take advantage of this fact so as to impose decays suited for our specific problems. 

Concerning the problem of having well-defined directional derivatives of geometric objects such as $A_{\bar{g}}$ for arbitrary directions $h\in \mathcal{B}_1$, let us first notice that whenever $\bar{g}_{\lambda}$ does not degenerate, appealing to certain local considerations, the directional derivatives of sufficiently regular metrics $\bar{g}$ and perturbations $h$ exist in a pointwise sense (see, for instance, Theorem III.1 in \cite{Catalanes}, and also Lemma 3.1 in \cite{Moncrief1}). Since around any point we can always find an interval such that $\bar{g}_{\lambda}$ is non-degenerate, then we conclude that $DA_{\bar{g}}\cdot h$ is well-defined around any point in $V$. Since our analysis will only rely on such pointwise considerations, this is enough for our purposes. That is, we know that given any perturbation of a smooth globally hyperbolic space-time $(V=M\times\mathbb{R},\bar{g})$, around any point $p\in V$, the directional derivative $DA_{\bar{g}}\cdot h$ is well-defined for any $h$ in some (weighted) $C^4$-space such that for any compact $\Omega\subset V$, $h\in C^4(\Omega)$. 
\begin{remark}
By topologising the space of $(0,2)$ symmetric tensor fields with some of the above topologies chosen to be strong enough to provide global $C^0$-control of $\bar{g},h\in \mathcal{B}_1(T)$ (for instance using as $\mathcal{B}_1$ some - potentially weighted - $C^k$-norm with $k$ sufficiently large), it is not difficult to show that the directional derivative in an arbitrary direction $h\in \mathcal{B}_1$ is well-defined. If, furthermore, the \textit{Gateaux}-differential $DA_{\bar{g}}\cdot h\doteq \frac{d}{d\lambda}A(\bar{g}+\lambda h)|_{\lambda=0}$ is a bounded linear map $DA_{\bar{g}}:\mathcal{B}_1\mapsto \mathcal{B}_2$, then $A$ would actually be $C^1$-Frech\'et. Again, choosing spaces with good \textit{multiplication properties} on $M$, such that (weighted) $C^k$ or Sobolev spaces can be used to make $DA_{\bar{g}}$ a bounded linear map.
\end{remark}
 
Let us now impose an specific structure of infinity for our manifolds $M$.

\begin{defn}[Manifolds Euclidean at infinity]\label{ManifoldsEAI}
A complete $n$-dimensional smooth Riemannian manifold $(M,e)$ is called Euclidean at infinity if there is a compact set $K$ such that $M\backslash K$ is the disjoint union of a finite number of open sets $U_i$, such that each $U_i$ is diffeomorphic to the exterior of an open ball in Euclidean space.   
\end{defn}

We will typically drop dependences on the metric $e$ assuming it fixed once and for all. 

\begin{defn}[AE manifolds]\label{AEmanifolds}
Let $(M,e)$ be a manifold Euclidean at infinity and $g$ be a Riemannian metric on $M$. We will say that $(M,g)$ is asymptotically Euclidean (AE) of order $\tau>0$ with respect to some end coordinate system $\Phi:E_i\mapsto \mathbb{R}^n\backslash\bar{B}$, if, in such coordinates, $g_{ij}-\delta_{ij}=O(|x|^{-\tau})$.
\end{defn}

Typically, we will also demand derivatives of the metric to decay at certain rates, but since such rates can depend on the specific problem at hand, we will look at these requirements explicitly whenever necessary, without introducing them in the definition of AE manifold. Along these lines, let us introduce the following distinguished family of AE manifolds.
\begin{defn}[AS manifolds]\label{ASchmanifolds}
We will say that the AE manifold $(M^n,g)$ is asymptotically Schwarzschild with respect to some end coordinate system $\Phi:E_i\mapsto \mathbb{R}^n\backslash\bar{B}$, if, in such coordinates there exists $\varepsilon>0$ such that
\begin{align*}
g_{ij}=\left(1+\frac{m}{2|x|^{n-2}} \right)^{\frac{4}{n-2}}\delta_{ij} + O_4(|x|^{-(n-2+\varepsilon)}).
\end{align*}
\end{defn}
Finally, let us introduce the following definition.
\begin{defn}[AM space-times]\label{AMmanifolds}
Let $(M\times [0,T],\bar{g})$ be a regularly sliced globally hyperbolic manifold, where $M$ is Euclidean at infinity. Thus, let us write as in remark \ref{remarkdecompoADM} 
{$\bar{g}=-N^2 dt^2+g_t$} where $N$ stands for the associated lapse function and $g$ for a time-dependent Riemannian metric on $M$. We will say that $\bar{g}$ is \textbf{asymptotically-Minkowskian} of order $\tau>0$ with respect to some end coordinate system $\Phi:E_i\mapsto \mathbb{R}^n\backslash\bar{B}$, if, in such coordinates, $g_{ij}(\cdot,t)-\delta_{ij}=O(|x|^{-\tau})$; $N(\cdot,t)-1=O(|x|^{-\tau})$ and $X_i=O(|x|^{-\tau})$, where $X$ denotes the shift vector associated to the orthogonal space-time splitting.
\end{defn}

In what follows, we will analyse conservation principles for solutions to $A_{\bar{g}}=0$. We will restrict $\bar{g}$ to be perturbation of some fixed (smooth) AM solution $\hat{\bar{g}}$ which possesses some (time-like) Killing field $\xi$. In this context, by a perturbation we mean a tensor field $h=\bar{g}-\hat{\bar{g}}\in \mathcal{B}_1$ with a behaviour of $\bar{g}$ at infinity controlled by that of $\hat{\bar{g}}$ and with $\mathcal{B}_1$ chosen so that all such perturbations are at least $C^4$. 
 
Finally, before going to main sections of this paper, we will distinguish a couple of choices of $\alpha$ and $\beta$ which are of special interest. The first one is guided by the interest in conformal gravity within theoretical physics, while the second one appears to be particularly amenable to a nice analytic treatment. This last point will become more transparent by the end of the paper.

\subsection{Remarkable Cases}

\subsubsection{Conformal Gravity}\label{SectionConfGrav}

In this section we will show that some values of $\alpha$ and $\beta$  induce conformally invariant field equations. These values are the ones which are consistent with the conformal gravity proposal of \cite{Confgrav}. 
These considerations can also be found in section II of \cite{SSsols3}. 

\begin{prop}
Let $(V, \bar{g})$ be a Lorentz manifold of dimension $4$. For any $(\alpha, \beta)$ satisfying $3\alpha+\beta= 0$, $A_{\bar{g}}=0$ is conformally invariant.
\end{prop}
\begin{proof}
We first assume the manifold is compact, and recall the Chern-Gauss-Bonnet formula in dimension $4$ (we here apply (1) with $p=3$ and $k=2$ from \cite{GBC1} to apply the formula in a Semi-Riemannian context, see also \cite{MR14760} and  \cite{GBC2} for the original articles):
\begin{equation} \label{CGB} \chi(V)= -\frac{1}{32 \pi^2} \int_V \left( \left| \mathrm{Riem}_{\bar{g}} \right|^2 -4 \left| \mathrm{Ric}_{\bar{g}} \right|^2 +R^2 _{\bar{g}}\right) d{V}_{\bar{g}}, \end{equation}
with $\chi(V)$ a topological invariant.

Let us recall the definition of the Weyl curvature which is the tracefree part of the Riemannian curvature. In local coordinates, one has:
$$W^{\rho}_{\; \; \mu \lambda  \nu } := \mathrm{Riem}^\rho_{ \; \; \mu \lambda \nu   }  - \frac{1}{n-1} \left( \mathrm{Ric}^\rho_\lambda \bar{g}_{\mu \nu}- \mathrm{Ric}^\rho_\nu \bar{g}_{\mu \lambda} + \mathrm{Ric}_{\mu \nu} \bar{g}^\rho_\lambda - \mathrm{Ric}_{\mu \lambda} \bar{g}^\rho_\nu \right)-\frac{R}{n(n-1)} \left( \bar{g}^\rho_\nu \bar{g}_{\mu \lambda} - \bar{g}^{\rho}_\lambda \bar{g}_{\mu \nu} \right).$$
By design,  $W^\rho_{ \; \; \mu \rho \nu } = 0$. One can then compute
\begin{equation}
\label{eqweylriemriccscal}
\left| \mathrm{Riem}_{\bar{g}} \right|^2 = \left| W_{\bar{g}}  \right|^2 + \frac{4}{n-1} \left| \mathrm{Ric}_{\bar{g}}  \right|^2 - \frac{2}{n(n-1)} R_{\bar{g}} ^2.
\end{equation}
 Injecting \eqref{eqweylriemriccscal} with $n=3$ into \eqref{CGB} then yields:
\begin{equation}
\label{therightproportions}
\int_V \left( 3 \left| \mathrm{Ric}_{\bar{g}}  \right|^2 - R_{\bar{g}} ^2 \right) dV_{\bar{g}}  = \frac{3}{2} \int_V \left| W_{\bar{g}}  \right|^2 dV_{\bar{g}} + 48 \pi^2 \chi(V).
\end{equation}

However, since the Weyl tensor $W^{\rho}_{\; \; \mu \lambda \nu}$ is a conformal invariant, if $\tilde g = \phi^2 {\bar{g}} $, one has:  $\tilde W^{\rho}_{\; \; \mu \lambda \nu } =  W^{\rho}_{\; \; \mu \lambda \nu }$. In addition:
$$\begin{aligned}
\tilde W_{\rho}^{\; \; \mu \lambda \nu} &= \tilde g_{\rho r} \tilde g^{\mu m } \tilde g^{\lambda l } \tilde g^{\nu n }  \tilde W^{r}_{\; \; mln} \\
&= \phi^{-4}  W_{\rho}^{ \; \; \mu  \lambda  \nu }.
\end{aligned}$$

Thus, in a Lorentz manifold of space dimension $n$, $\big|\tilde W \big|^2 dV_{\tilde g}= \phi^{n-3}  \left| W_{\bar{g}}  \right|^2 dV_{\bar{g}} $. The Bach energy $$ \mathrm{Bach}= \int_V \left| W_{\bar{g}}  \right|^2 d\mathrm{vol}_{\bar{g}} $$ is thus a conformal invariant if and only if  $n=3$ (spacetime dimension: $4$).

In the case of a compact Lorentz manifold, since the right-hand side of \eqref{therightproportions} is  the sum of a conformally invariant and a topologically invariant term, $E_{-\lambda, 3 \lambda}$ is a conformal invariant. Its Euler-Lagrange tensor $A$ is thus proportional to the Euler-Lagrange tensor of the Bach energy i.e. the Bach tensor:
$$
B_{\mu \nu} = \left[ \nabla^{\rho \lambda} + \frac{1}{2} \mathrm{Ric}^{\rho \lambda} \right] W_{\rho \mu \lambda \nu },
$$
and the relation $A=0$ is necessarily conformally invariant. Here, and in what follows, $\nabla^{\kappa \lambda}$ is a notation to express in a concise manner $\nabla^\kappa \nabla^\lambda$. In fact, tensorial computations show that if $\alpha =-1$ and $\beta= 3$, $A=6B$. 

This identity remains true in the non-compact case, as it is a pointwise tensorial equality.

\qed

\end{proof}

\begin{remark}
This situation is reminiscent of the Willmore case, where the Willmore energy differs from the true conformal invariant by a topological term (see for instance \cite{bibwill} for an introduction to the Willmore problem). We refer the reader  to \cite{MR3415767} for a striking parallel between Willmore surfaces  and Bach flat spaces, notably in their conformal properties.
\end{remark}

\subsubsection{Einstein tensor formulation}
In any globally hyperbolic space-time $(V= M \times \R, \bar{g})$ we wish to express $A_{\bar{g}}$ as a function of the Einstein tensor, $\mathrm{Ric}_{\bar{g}} = G_{\bar{g}}  + \frac{R_{\bar{g}} }{2} \bar{g}$. Thus \eqref{A-tensor} yields

\begin{equation} \label{031220200944} \begin{aligned} A_{\bar{g}}&=\beta\Box_{\bar{g}}{G_{\bar{g}} } + ( \beta  + 2\alpha)\Box_{\bar{g}}R_{\bar{g}}\: {\bar{g}} - (2\alpha +  \beta)\bar{\nabla}^2R_{\bar{g}}  +2\beta{G_{\bar{g}}}_{\cdot}\mathrm{Riem}_{\bar{g}} \\&+ \beta R_{\bar{g}}  \mathrm{Ric}_{\bar{g}}
  + 2\alpha R_{\bar{g}}G_{\bar{g}} + \alpha R^2_{\bar{g}} \bar{g} -\frac{1}{2 }\alpha R^2_{\bar{g}} {\bar{g}} -\frac{1}{2 } \beta\langle\mathrm{Ric}_{\bar{g}},\mathrm{Ric}_{\bar{g}}\rangle_{\bar{g}} {\bar{g}} \\
&=\beta \left[ \Box_{\bar{g}}{G_{\bar{g}} }+ 2{G_{\bar{g}}}_{\cdot}\mathrm{Riem}_{\bar{g}} + \frac{R^2_{\bar{g}}}{4} \bar{g}- \frac{1}{2}\langle\mathrm{Ric}_{\bar{g}},\mathrm{Ric}_{\bar{g}}\rangle_{\bar{g}} {\bar{g}} \right]   \\& - \left( 2 \alpha + \beta \right) \left[ \bar{\nabla}^2R_{\bar{g}} - \Box_{\bar{g}}R_{\bar{g}}\: {\bar{g}} -  R_{\bar{g}}G_{\bar{g}} - \frac{R^2_{\bar{g}}}{4} \bar{g} \right].
\end{aligned}\end{equation}
Further, we can compute: $$\begin{aligned}
 \left( \langle \mathrm{Ric}_{\bar{g}}, \mathrm{Ric}_{\bar{g}} \rangle_{\bar{g}}- \frac{R^2_{\bar{g}}}{2} \right) \bar{g}_{\mu \nu}= {\mathrm{Ric}_{\bar{g}}}^\lambda_\tau {\mathrm{G}_{\bar{g}}}^\tau_\lambda \bar{g}_{\mu \nu} =  {\mathrm{G}_{\bar{g}}}.{\mathrm{Ric}_{\bar{g}}} \bar{g}_{\mu \nu}.
\end{aligned}$$
Injecting this into \eqref{031220200944} (with the same notation) then allows us to reformulate:
\begin{equation} \label{031220201238}\begin{aligned}
 A_{\bar{g}} &=\beta \left[ \Box_{\bar{g}}{G_{\bar{g}} }+  2{G_{\bar{g}}}_{\cdot} \left( \mathrm{Riem}_{\bar{g}} - \frac{\mathrm{Ric}_{\bar{g}}}{4} \bar{g} \right)  \right]   \\& - \left( 2 \alpha + \beta \right) \left[ \bar{\nabla}^2R_{\bar{g}} - \Box_{\bar{g}}R_{\bar{g}}\: {\bar{g}} -  R_{\bar{g}}G_{\bar{g}} - \frac{R^2_{\bar{g}}}{4} \bar{g} \right].
\end{aligned}
\end{equation}
Thus, when $2\alpha + \beta =0$, $A$ is an operator whose leading term is $\Box_{\bar{g}}{G_{\bar{g}} }$:
$$ \Box_{\bar{g}}{G_{\bar{g}} }+  2{G_{\bar{g}}}_{\cdot} \left( \mathrm{Riem}_{\bar{g}} - \frac{\mathrm{Ric}_{\bar{g}}}{4} \bar{g} \right)=0.$$
It is interesting to notice that the the term ${G_{\bar{g}}}_{\cdot} \left( \mathrm{Riem}_{\bar{g}} - \frac{\mathrm{Ric}_{\bar{g}}}{4} \bar{g} \right)$ is \emph{tracefree} in spacetime dimension $4$:
$$\begin{aligned}
\bar{g}^{\mu \nu }\left( {{G_{\bar{g}}}}^\lambda_\tau \left( {\mathrm{Riem}_{\bar{g}}}^\tau_{\, \mu  \lambda \nu} - \frac{{\mathrm{Ric}_{\bar{g}}}^\lambda_\tau}{4} \bar{g}_{\mu \nu} \right) \right) &={{G_{\bar{g}}}}^\lambda_\tau \left( {{\mathrm{Riem}_{\bar{g}}}^{\tau \mu}_{\quad \lambda \mu} }- \frac{{\mathrm{Ric}_{\bar{g}}}^\lambda_\tau}{4} \bar{g}^{\mu}_{ \mu} \right)   \\
&={{G_{\bar{g}}}}^\lambda_\tau \left(  {\mathrm{Ric}_{\bar{g}}}^\lambda_\tau - {\mathrm{Ric}_{\bar{g}}}^\lambda_\tau \right) = 0.
\end{aligned}$$ 
Similarly the $0$th order term  in $R$ is tracefree.
Thus, tracing  \eqref{031220201238} yields: $\Box_{\bar{g}} R_{\bar{g}}=0$ in dimension $4$. If  in addition $2 \alpha + \beta=0$, the Euler-Lagrange equation $A_{\bar{g}}=0$ is:
$$
\left\{
\begin{aligned}
&\Box_{\bar{g}}R_{\bar{g}} = 0 \\
&\Box_{\bar{g}}{ G_{\bar{g}}}_{\mu \nu} +2 \left( {\mathrm{Riem}_{\bar{g}}}^\tau_{\: \mu \lambda \nu } - \frac{{\mathrm{Ric}_{\bar{g}}}^\tau_\lambda}{4} \bar{g}_{\mu \nu} \right) {G_{\bar{g}}}^\lambda_\tau = 0.
\end{aligned}
\right.
$$

\section{Conservation principles and fourth order energy}
{The type of analysis that we here lead is end-local. We will thus, in the following consider manifolds with only one end.}
\subsection{Conservation principles}
The aim of this section is to present a detailed analysis which will lead us to an appropriate notion of energy for AE solution to $A_{\bar{g}}=0$. As we have already stated, we will appeal to symmetry principles and their associated conservation laws as a primary tool. Along these lines, let us start by noticing that by restricting variations of $S(\bar{g})$ to those generated by the flow of smooth vector fields, we have the well-known local conservation principle
\begin{align*}
\mathrm{div}_{\bar{g}}A_{\bar{g}}=0.
\end{align*}
The above local conservation identity is a local property of the tensor field $A_{\bar{g}}$ and therefore remains valid for any smooth Lorentzian metric on $V$ (regardless of whether $M$ is compact or not). This local conservation principle will imply a linearised version of the same principle, which we will exploit in order to define a notion of energy for solutions of the space-time field equations $A_{\bar{g}}=0$.

\begin{lemma}
Let $\bar{g}_{0}$ be a smooth solution of $A_{\bar{g}}=0$ in space-time. Then, for any perturbation $h$, it holds that
\begin{align}\label{Linearised-Aconservation}
\mathrm{div}_{\bar{g}_0}\left(DA_{\bar{g}_0}\cdot h\right)=0.
\end{align}
\end{lemma}
\begin{proof}
Notice that for any smooth $\bar{g}$
\begin{align*}
0={\mathrm{div}_{\bar{g}}A_{\bar{g}}}_{\mu}&=\bar{g}^{\sigma\nu}\bar{\nabla}_{\sigma}A_{\nu\mu}=\bar{g}^{\sigma\nu}\left( \partial_{\sigma}A_{\nu\mu} - \bar{\Gamma}^{\gamma}_{\sigma\nu}A_{\gamma\mu} - \bar{\Gamma}^{\gamma}_{\sigma\mu}A_{\nu\gamma} \right),\\
&=\bar{g}^{\sigma\nu}\left( {}^{0}\!\bar{\nabla}_{\sigma}A_{\nu\mu} - S^{\gamma}_{\sigma\nu}A_{\gamma\mu} - S^{\gamma}_{\sigma\mu}A_{\nu\gamma} \right),
\end{align*}
where $S^{\gamma}_{\sigma\nu}(\bar{g})\doteq \bar{\Gamma}^{\gamma}_{\sigma\nu}(\bar{g})-{}^{0}\!\bar{\Gamma}^{\gamma}_{\sigma\nu}(\bar{g}_0)$; ${}^{0}\!\bar{\Gamma}^{\gamma}_{\sigma\nu}$ stands for the connection coefficients associated to $\bar{g}_0$ and ${}^{0}\!\bar{\nabla}$ denotes the covariant derivative operator associated to $\bar{g}_0$. From the above, and using that $A_{\bar{g}_0}=0$; $\mathrm{div}_{\bar{g}_0}A_{\bar{g}_0}=0$ and $S(\bar{g}_0)=0$, we find that
\begin{align*}
\bar{g}_0^{\sigma\nu} {}^{0}\!\bar{\nabla}_{\sigma}\left(DA_{\bar{g}_0}\cdot h\right)_{\nu\mu} =0,
\end{align*}
which proves the claim.
\qed
\end{proof}

\bigskip

The idea will now be to appeal to the linearised conservation identity presented above in order to analyse perturbations of solutions which present a time-like Killing field.

Let us begin, as above, assuming the existence of a smooth solution $\hat{\bar{g}}$ to the space-time equations $A_{\bar{g}}=0$. Let $\xi$ be a vector field on space-time and denote $\mathcal{P}_{\hat{\bar{g}}}(\xi, h)\doteq (DA_{\hat{\bar{g}}}\cdot h)(\xi,\cdot)$. Then, appealing to (\ref{Linearised-Aconservation}), it holds that
\begin{align*}
\mathrm{div}_{\hat{\bar{g}}}(\mathcal{P}_{\hat{\bar{g}}}(\xi, h))=\frac{1}{2}\langle DA_{\hat{\bar{g}}}\cdot h,\pounds_{\xi}\hat{\bar{g}} \rangle_{\hat{\bar{g}}},
\end{align*}
 We are actually interested in the case where $\xi$ produces some continuous symmetry of $\hat{\bar{g}}$. Therefore, suppose that $\xi$ is a $\hat{\bar{g}}$-Killing field which gives us the following.
\begin{align}\label{Linearised-EMconservation}
\mathrm{div}_{\hat{\bar{g}}}(\mathcal{P}_{\hat{\bar{g}}}(\xi, h))=0.
\end{align}

From now on, we will only consider asymptotically Minkowskian (AM) space-times with AE space-slices. In this context, let us introduce the following definition.

\begin{defn}
Let $(M\times [0,T],\hat{\bar{g}})$ be an AM space-time where $A_{\hat{\bar{g}}}=0$ which admits a time-like Killing field $\xi$. Then, given a compact set $K\subset M$, the energy within the compact sets $K_t=K\times\{t\}\subset M_t$ associated to a perturbation $h$ of $\hat{\bar{g}}$ is defined by
\begin{align}\label{energy.1}
\mathcal{E}_{\hat{\bar{g}}}(t;K,h)\doteq - \int_{K_t}\langle \mathcal{P}_{\hat{\bar{g}}}(\xi,h),n \rangle_{\hat{\bar{g}}}dV_{\hat{g}},
\end{align}
where $\hat{g}$ stands for the induced Riemannian metric on $M_t$ by $\hat{\bar{g}}$ and $n$ for the $\hat{g}$-future-pointing unit normal to $M_t$.
\end{defn}

Following the conventions adopted in the previous section, let us split the 1-form $\mathcal{P}_{\hat{\bar{g}}}(\xi,h)$ into a function $\mathcal{P}^{(0)}$ and 1-form $\mathcal{P}^{(1)}\in \Gamma(T^{*}M)$ defined by 
\begin{align}
\begin{split}
\mathcal{P}^{(0)}&=\<\mathcal{P}_{\hat{\bar{g}}}(\xi, h),\partial_t \>,\\
\mathcal{P}^{(1)}(Z)&=\<\mathcal{P}_{\hat{\bar{g}}}(\xi, h),Z \> \text{ for all } Z\in \Gamma(T^{*}M).
\end{split}
\end{align}
 Here, and in the following, we simplify notations and denote a spacetime form and its restriction using the same symbol.
\begin{prop}\label{Energyconservation}
Let $(M\times [0,T],\hat{\bar{g}})$ be an AM space-time where $A_{\hat{\bar{g}}}=0$ which admits a time-like Killing field $\xi$. Then, if there exists $\epsilon >0$ such that $\mathcal{P}^{(0)},\mathcal{P}^{(1)}= O (r^{-(n+ \epsilon)})$, the energy $(\ref{energy.1})$ over all of $M$ is conserved through evolution.
\end{prop}
\begin{proof}
Consider a compact set $K\subset M$ and a cylinder $C$ over $K$, that is, $C=K\times [0,t]$, and then let us integrate (\ref{Linearised-EMconservation}) over $C$ so as to get, 
using the orientation conventions detailed in the  appendix \ref{subsectionconvention} (more precisely see \eqref{Flux.1}):
\begin{align}\label{fluxformula}
\begin{split}
-\int_{K_t}\langle \mathcal{P}_{\hat{\bar{g}}}(\xi, h),n \rangle_{\hat{\bar{g}}}dV_{\hat{g}} + \int_{K_0}\langle \mathcal{P}_{\hat{\bar{g}}}(\xi, h),n \rangle_{\hat{\bar{g}}}dV_{\hat{g}} &= - \int_L \langle \mathcal{P}_{\hat{\bar{g}}}(\xi, h),\nu \rangle_{\hat{\bar{g}}}dL,\\
&= \int_0^t\int_{\partial K_t} \hat{N}\langle \mathcal{P}_{\hat{\bar{g}}}(\xi,h),\nu \rangle_{\hat{\bar{g}}}dtd(\partial K),
\end{split}
\end{align}
where $n$ stands for the future-pointing unit normal to $M_t$ and $K_t = K \times \{ t \} $ denotes the compact on the $t$ slice; $\nu$ for the outward-pointing unit normal to the lateral boundary $L=\partial K\times [0,t]$; $\hat{N}$ for the lapse function associated with $\hat{\bar{g}}$ and $d\partial K$ for the induced volume form on $\partial K$. 

We want to present the above formula for all of $M$, not just compact subsets. In particular, we want to prove that the right-hand side goes to zero as we move towards infinity in $M$. Let us denote by $\hat{g}_t$ and $\mathcal{P}_{\hat{\bar{g}}}(\xi,h)$ the Riemannian metric and $1$-form induced on $M_t$ by $\hat{\bar{g}}$ and $\mathcal{P}_{\hat{\bar{g}}}(\xi,h)$ respectively. Then, since $\nu$ is tangent to $M$ we get that
\begin{align*}
|\langle \mathcal{P}_{\hat{\bar{g}}}(\xi, h),\nu\rangle_{\hat{g}}|&=|\langle \mathcal{P}^{(1)},\nu\rangle_{\hat{g}}|\leq|\mathcal{P}^{(1)}|_{\hat{g}} = O (r^{-(n + \varepsilon)} )\in L^1(M,dV_{\hat{g}}).
\end{align*}
The above implies that 
near infinity, $|\langle \mathcal{P}_{\hat{\bar{g}}}(\xi, h ),\nu\rangle_{\hat{g}}|=o(r^{-(n-1)})$.

Let us then consider space compacts  $K_r$  sufficiently large so that $\partial K_r$ is contained in the ends of $M$, and then let us chose such $K_r$ so that $\partial K_r$ is orientable. Let us build such  $K_r$  such that $\partial K_r =  \cup_j S_{j,r}$ where the $S_{j,r}$ are Euclidean spheres of radii $r$ sufficiently close to infinity in each end $\{ E_j\}_j^{\mathcal{N}=\#\text{ends}}$. Integrating in the cylinder $\mathcal{C}_r = K_r \times [0, T] = \cup_{t \in [0,T]} K_{t,r}$ then yields \eqref{fluxformula}, where the proximity of the cylindric boundary term to space infinity is quantified by the parameter $r$; and can thus decay in a controlled manner.   Indeed, denoting by $dV_{\hat{g}}$ the volume form associated to $\hat{g}$ in these ends, we see that
\begin{align*}
\langle \mathcal{P}^{(1)},\nu \rangle_{\hat{g}}d(\partial K_r)=\langle \mathcal{P}^{(1)},\nu \rangle_{\hat{g}}\nu\lrcorner dV_{\hat{g}}=\langle \mathcal{P}^{(1)},\nu \rangle_{\hat{g}}\sqrt{\frac{\mathrm{det}(\hat{g})}{\mathrm{det}(e)}}\nu\lrcorner dV_{e}=f\langle \mathcal{P}^{(1)},\nu \rangle_{\hat{g}}\:\nu\lrcorner dV_{e},
\end{align*}
where $f>0$ is a continuous and bounded function on $M$. Also,  sufficiently near infinity, in the natural Cartesian end coordinates $dV_e=dx^1 \wedge \cdots \wedge dx^n$. 
Then, along each $S_{j,r}$, we get that $\nu\lrcorner dV_e= \nu\lrcorner (r^{n-1}dr\wedge d\omega_{n-1})=\nu_rr^{n-1}d\omega_{n-1}$, where $d\omega$ stands for the canonical volume form on the unit sphere. Then, 
\begin{align*}
\langle\mathcal{P}^{(1)},\nu \rangle_{\hat{g}}d(\partial K_{t,r})=f \langle \mathcal{P}^{(1)},\nu\rangle_{g} \nu_rr^{n-1} d\omega_{n-1}.
\end{align*}
Therefore, since $|\langle \mathcal{P}_{\hat{\bar{g}}}(\xi, h),\nu\rangle_{\hat{g}}|=o(r^{-(n-1)})$ and $N=1+O(r^{-\tau})$ with $\tau>0$, it follows that
\begin{align*}
\Big\vert\int_{\partial K_{t,r}}\hat{N}\langle \mathcal{P}^{(1)},\nu \rangle_{\hat{g}}d(\partial K_{t,r})\Big\vert\lesssim\sum_{j=1}^{\mathcal{N}}\int_{S_{j,r}} r^{n-1}|\mathcal{P}|_e d\omega_{n-1}\xrightarrow[r\rightarrow\infty]{}0.
\end{align*}

Similarly, consider $\langle \mathcal{P},\hat{n} \rangle_{\hat{\bar{g}}}=\hat{N}^{-1}\langle \mathcal{P},\partial_t - \hat{X} \rangle_{\hat{g}}=\hat{N}^{-1}\left(\mathcal{P}^{(0)} + \langle \mathcal{P}^{(1)},\hat{X} \rangle_{\hat{g}}\right)\in L^1(M,dV_{\hat{g}})$ since $N=1+O(r^{-\tau}); |\hat{X}|=O(r^{-\tau})$ and $\mathcal{P}^{(0)},\mathcal{P}^{(1)}\in L^1(M,dV_{\hat{g}})$. Therefore, we can pass to the limit in (\ref{fluxformula}), in order to get
\begin{align*}
\int_{M_t}\langle \mathcal{P}_{\hat{\bar{g}}}(\xi, h),n \rangle_{\hat{\hat{g}}}dV_{\bar{g}} = \int_{M_0}\langle \mathcal{P}_{\hat{\bar{g}}}(\xi, h),n \rangle_{\hat{\bar{g}}}dV_{\hat{g}},
\end{align*}

\qed
\end{proof}

\bigskip

The above proposition provides us with a natural notion for the energy of space-time solutions to $A_{\bar{g}}=0$ which arise as perturbations of solutions with time-translational symmetries. Next, we will show how $\mathcal{E}_{\hat{\bar{g}}}(h,M)$ can be computed as a boundary term at space-like-infinity. 
All the conventions for the operators in the Lorentzian setting are defined in the appendix (see subsection \ref{subsectionconvention}).

\begin{prop}\label{ADMenergy.1}
Let $(V=M\times \mathbb{R},\hat{\bar{g}})$ be an orientable AM space-time which satisfies the hypotheses of Proposition \ref{Energyconservation}. Then, there is a $2$-form $\mathcal{Q}_{\hat{\bar{g}}}(\xi, h)$ such that $\mathcal{P}_{\hat{\bar{g}}}(\xi, h)=\delta_{\hat{\bar{g}}}\mathcal{Q}_{\hat{\bar{g}}}(\xi, h)$ and such that the energy is given by
\begin{align}\label{Energychage}
\mathcal{E}_{\hat{g}}(M,h)=\int_{S_{\infty}}*_{\hat{\bar{g}}}\mathcal{Q}_{\hat{\bar{g}}}(\xi, h),
\end{align}
where $S_{\infty}$ denotes \textit{the sphere at infinity}.
\end{prop}
\begin{proof} 
In this proof, we drop the dependency of the corresponding forms on $\hat{\bar{g}}$, $\xi$ and $h$ for simplification. 
First we notice that $\mathrm{div}_{\hat{\bar{g}}}\mathcal{P}=0 \iff \delta_{\hat{\bar{g}}}\mathcal{P}=0  
\iff d(*_{\hat{\bar{g}}}\mathcal{P})=0$. But then, since $M^n\times\{0\}$ is a deformation retraction of $V=M^n\times \mathbb{R}$, their cohomologies agree, and, furthermore, since $M^n$ is non-compact, its top cohomology vanishes (see, for instance, Theorem 11 in Chapter 9 of \cite{Spivak1} and page 279 in the same reference for more details.). All this implies that $*_{\hat{\bar{g}}}\mathcal{P}$ is exact and thus there is an $(n-1)$-form, denoted by $*_{\hat{\bar{g}}}\mathcal{Q}$, such that $*_{\hat{\bar{g}}}\mathcal{P}=d(*_{\hat{\bar{g}}}\mathcal{Q})$. 

Now, consider a neighbourhood $\Omega$ of any point $p\in M$ and an oriented $\hat{\bar{g}}$-orthonormal frame of the form $\{n,e_1,\cdots,e_n\}$, where $n$ is normal to $M$ and $\{e_i\}_{i=1}^n$ is a frame in $\Omega\cap M$. Let $\{\theta^{\alpha}\}$ be its dual frame and write $\mathcal{P}=\mathcal{P}_{\alpha}\theta^{\alpha}$\footnote{Notations and conventions are detailed in the appendix section \ref{subsectionconvention}}. We then get that
\begin{align*}
J^{*}(*\mathcal{P})=-\mathcal{P}_0\circ J\theta^1\wedge\cdots\wedge\theta^n,
\end{align*}
where $J:M \hookrightarrow V$ denotes the inclusion map. In such orthonormal frames $\mathcal{\bar{P}}_0=\epsilon\mathcal{P}^0$, where $\epsilon=\langle n,n \rangle_{\hat{\bar{g}}}$ 
 we get that:
\begin{align*}
J^{*}(*_{\hat{\bar{g}}}\mathcal{P})=-\langle \mathcal{P},n \rangle_{\hat{\bar{g}}} dV_{\hat{g}}.
\end{align*}
Therefore we get:
\begin{align*}
\mathcal{E}_{\hat{\bar{g}}}(K_{r},h)&=\int_{{K_r}}J^{*}(*\mathcal{P})=\int_{{K_r}}J^{*}(d*\mathcal{Q})=\int_{\partial {K_r}}J^{*}_{\partial {K_r},{K_r}}(*\mathcal{Q})
\end{align*}
where $J_{\partial {K_r},{K_r}}:\partial {K_r} \hookrightarrow {K_r}$ denotes the inclusion. From Proposition \ref{Energyconservation} we know that the limit
\begin{align}\label{energy.2}
\mathcal{E}_{\hat{\bar{g}}}(M,h)&=\lim_{r\rightarrow\infty}\mathcal{E}_{\hat{\bar{g}}}(K_{r},h)=\lim_{r\rightarrow\infty}\int_{\partial K_{r}}J^{*}_{\partial {K_r},{K_r}}(*\mathcal{Q})<\infty
\end{align}
exists and is finite. Since $M$ is Euclidean at infinity, we can take advantage that the ends are diffeomorphic to the exterior of a ball in $\mathbb{R}^n$ and integrate over sequences of spheres the boundary terms, thus using some abuse in notation, we get that
\begin{align}\label{Charge.1}
\mathcal{E}_{\hat{\bar{g}}}(M,h)=\int_{S_{\infty}}*\mathcal{Q}.
\end{align}
\qed
\end{proof}

\bigskip
\begin{remark}
\begin{itemize}
\item In the following we will drop the dependency in $\hat{\bar{g}}$, $\xi$ and $h$ when recalling it is not necessary, in order to lighten notations.
\item Notice that, following the above arguments, we see that $\mathcal{Q}$ ($*_{\hat{g}}\mathcal{Q}$) is defined up to a co-exact (exact) form. Nevertheless, notice that this \textit{indeterminacy} concerning the $(n-2)$-form $*\mathcal{Q}$ does not affect the conserved quantities (\ref{Charge.1}), since the addition of an exact form to $*\mathcal{Q}$ (after pull-back to the boundary) will not produce any contribution when integrated over the compact boundary $\partial K_r$.
\end{itemize}
\end{remark}

\begin{coro}\label{ADMenergy.2}
Under the same assumptions as in Proposition \ref{ADMenergy.1}, we can rewrite (\ref{Energychage}) as
\begin{align*}
\mathcal{E}_{\hat{\bar{g}}}(M,h)=-\lim_{r\rightarrow\infty}\int_{\partial K_{r}}\mathcal{Q}(n,\nu)d(\partial K_{r}).
\end{align*}
\end{coro}
\begin{proof}
Picking some oriented coordinate system, notice that\footnote{The completely antisymmetric symbols ${\mu_{\hat{\bar{g}}}}_{\alpha_0\cdots\alpha_{n}}$ are defined by the relation $dV_{\hat{\bar{g}}}={\mu_{\hat{\bar{g}}}}_{\alpha_0\cdots\alpha_{n}}\vartheta^{\alpha_0}\wedge\cdots\wedge\vartheta^{\alpha_n}$ for some positively oriented co-frame $\{\vartheta^{\alpha_0},\cdots,\vartheta^{\alpha_n} \}$.}
\begin{align*}
*_{\hat{\bar{g}}}\mathcal{Q}=\sum_{\nu<\beta}\mathcal{Q}_{\nu\beta}*_{\hat{\bar{g}}}dx^{\nu}\wedge dx^{\beta}=\sum_{ \nu<\beta} \sum_{  i_1< \cdots <i_{n-1}}\mathcal{Q}^{\nu\beta}{\mu_{\hat{g}}}_{\nu\beta i_1\cdots i_{n-1}}dx^{i_1}\wedge\cdots\wedge dx^{i_{n-1}},
\end{align*} 
Let $(x^0,x^1,\cdots,x^{n})$ be a local normal positively oriented space-time coordinate system around some point $p\in \partial K_r$ where $\partial_0|_p=n_p$ and $\partial_1|_p=\nu_p$, with $n$ being the future point normal to $M$ at each time, while $\nu$ stands to the outward point normal vector field of $\partial K$ as a hypersurface of $M_t$. Also, let $J_1:M_t\hookrightarrow V$ and $J_2:\partial K_r\hookrightarrow M_t$ be the natural inclusions. Then, at $p$, we get
\begin{align*}
J_1^{*}(*\mathcal{Q})_p&=\mathcal{Q}^{0 j}_p{\mu_{\hat{g}}}_{0 j i_1\cdots i_{n-1}}dx^{i_1}\wedge\cdots\wedge dx^{i_{n-1}}|_p,\\
J_2^{*}J_1^{*}(*\mathcal{Q})_p&=\mathcal{Q}_p^{01}{\mu_{\hat{g}}}_{0 1 i_1\cdots i_{n-1}}dx^{i_1}\wedge\cdots\wedge dx^{i_{n-1}}|_{p},\\
\end{align*}
Notice also that $(\nu\lrcorner(n\lrcorner dV_{\hat{\bar{g}}}))_p={\mu_{\hat{g}}}_{0 1 i_1\cdots i_{n-1}}dx^{i_1}\wedge\cdots\wedge dx^{i_{n-1}}|_{p}$ and  since our coordinates are orthonormal at $p$, we have $\hat{\bar{g}}_{\mu\nu}(p)=\eta_{\mu\nu}=\mathrm{diag}(-1,1,\cdots,1)$, and thus it follows that $\mathcal{Q}^{01}(p)=-\mathcal{Q}_{01}(p)=-\mathcal{Q}(\partial_0|_p,\partial_1|_p)=-\mathcal{Q}_p(n,\nu)$. Thus,
\begin{align*}
J^{*}_2J_1^{*}(*\mathcal{Q})_p&=-\mathcal{Q}_p(n,\nu)(\nu\lrcorner(n\lrcorner dV_{\hat{\bar{g}}}))_p.
\end{align*}
Furthermore, since the last expression is coordinate independent, it holds for all points in $\partial K$. Finally, recalling that $(\nu\lrcorner(n\lrcorner dV_{\hat{\bar{g}}}))=d(\partial K)$, is the induced volume form on $\partial K$ 
 (once more the notations are clarified for the Lorentz setting in the appendix section \ref{subsectionconvention}), we get that
\begin{align}\label{ADMconservedchages.1}
\mathcal{E}_{\hat{g}}(M,h)=-\lim_{r\rightarrow\infty}\int_{\partial K_{r}}\mathcal{Q}(n,\nu)d(\partial K_{r}).
\end{align}
\qed
\end{proof}

\bigskip
\begin{remark}
Notice that the condition that guarantees that (\ref{ADMconservedchages.1}) is well-defined is that $|\mathcal{Q}(n,\nu)|=O(|x|^{-(n-2)})$. This condition is weaker than the one we used in Proposition \ref{Energyconservation}.
\end{remark}

\begin{prop}\label{Energycharge.Eins}
Let $(M\times [0,T],\hat{\bar{g}})$ be an AM space-time where $A_{\hat{\bar{g}}}=0$ which admits a time-like Killing field $\xi$. Then, if ${\hat{\bar{g}}}$ is Einstein there is a $2$-form $\mathcal{Q}_{\hat{\bar{g}}} (\xi,h)\in \Omega^2(V)$ such that for any perturbation $h$ it holds that
\begin{align*}
\mathcal{P}_{\hat{\bar{g}}}(\xi,h)&=\delta_{\hat{\bar{g}}}\mathcal{Q}_{\hat{\bar{g}}} (\xi,h).
\end{align*}
\end{prop}
\begin{proof}
We will use  simplified notations to lighten the formulas and denote $\hat{g}$ the background solution around which we linearize. We will use expression \eqref{031220201238} for convenience (the expression as a function of the Einstein curvature will be more practical given our hypothesis, and will simplify  the expression as a function of quantities known in GR):
$$\begin{aligned}
 {A_{{g}} }_{\mu \nu} &=\beta \left[ \Box_{{g}}{ G_{{g}}}_{\mu \nu} +2 \left( {\mathrm{Riem}_{{g}}}^\tau_{\: \mu \lambda \nu} - \frac{{\mathrm{Ric}_{{g}}}^\tau_\lambda}{4} {g}_{\mu \nu} \right) G^\lambda_\tau  \right]   \\& - \left( 2 \alpha + \beta \right) \left[ {\nabla}_{\mu \nu}R_{{g}} - \Box_{{g}}R_{{g}}\: {{g}_{\mu \nu}} -  R_{{g}}{G_{{g}}}_{\mu \nu} - \frac{R^2_{{g}}}{4} {g}_{\mu \nu} \right] \\
&= \beta {A_g}_{\mu\nu}^{(1)} - \left( 2 \alpha + \beta \right)  {A_g}_{\mu\nu}^{(2)}.
\end{aligned}
$$

Let us then assume that $\hat{g}$ is Einstein. 
For such a metric to be a solution of $A_g=0$, one needs $\Lambda = 0$ or $n=3$. Thus: 
\begin{equation}
\label{031220201243}
\begin{aligned}
&G_{\hat{g}}= -\Lambda \hat{g} \\
&\mathrm{Ric}_{\hat{g}} = \frac{2 \Lambda}{n-1} \hat{g} =  \Lambda \hat{g}\\
& R_{\hat{g}} =  \frac{2(n+1) \Lambda}{n-1} =4 \Lambda.
\end{aligned}
\end{equation}

From \eqref{031220201243} we deduce:
\begin{equation}
\label{031220201311}
\hat{\nabla} \mathrm{Ric}_{\hat{g}}= \hat{ \nabla} G_{\hat{g}} = \hat{\nabla} R_{\hat{g}}=0.
\end{equation}
 We consider $\xi^\nu$ a Killing field, and a perturbation of $\hat{g}$ of the form:
$g(\varepsilon) = \hat{g} + \varepsilon h + o(\varepsilon)$. Denoting $D=\left. \frac{d}{d\varepsilon} \right|_{\varepsilon=0}$ (which we will abbreviate to $R'$ and $\mathrm{Ricc}'$ in those two cases, for conciseness in the final formula),  one has: 
\begin{equation}
\label{031220201334}
\begin{aligned}
D \left( \Gamma^\tau_{\mu \nu} (\varepsilon) \right)(h)&=  \frac{1}{2} \left( \hat{\nabla}_\mu h^\kappa_\nu +\hat{\nabla}_\nu h^\kappa_\mu - \hat{\nabla}^\kappa h_{\mu \nu} \right) \\
\end{aligned}
\end{equation}
Then, from \eqref{031220201311} we find:
\begin{equation}
\label{031220201308}
\begin{aligned}
D \left( {\nabla}_{\mu \nu } R_g \right).h &= \hat{\nabla}_{\mu \nu} \left( D R_g.h\right) - D \left( \Gamma^\kappa_{\mu \nu} \right).h \hat{\nabla}_\kappa R_{\hat{g}}  \\
&= \hat{\nabla}_{\mu \nu} \left( DR_g.h \right):= \hat{\nabla}_{\mu \nu} \left( R_{\hat{g}}'.h\right).
\end{aligned}
\end{equation}
Which also implies:
$$
D \left( \Box_{g} R_g g_{\mu \nu} \right).h = \Box_{\hat{g}} \left( R_{\hat{g}}'.h \right)\hat{g}_{\mu \nu}.
$$

Similarly: 
$$
\begin{aligned}
D \left( R_g {G_g}_{\mu \nu}  + \frac{R_g^2}{4}g_{\mu \nu }  \right).h &=- R_{\hat{g}}'.h \Lambda \hat{g}_{\mu \nu} + 2 \Lambda R_{\hat{g}}'.h  \hat{g}_{\mu \nu} + 4 \Lambda^2 h_{\mu \nu} + 4 \Lambda  G_{\hat{g}}'.h_{\mu \nu} \\
&=  \Lambda R_{\hat{g}}'.h \hat{g}_{\mu \nu}   + 4 \Lambda D \left( {G_{{g}}}_{\mu \nu} +  \Lambda g_{\mu \nu} \right).h.
\end{aligned}
$$
Then:
$$ 
\begin{aligned}
D\left( {A_{\hat{g}}}_{\mu\nu}^{(2)} \right).h \xi^\nu&=D \left( \nabla_{\mu \nu } R_g - \Box_{g} R_g g_{\mu \nu} -R_g{G_g}_{\mu \nu} - \frac{R_g^2}{4} g_{\mu \nu} \right).h \xi^\nu \\ &=\hat{\nabla}_{\nu \mu } \left(  R_{\hat{g}}'.h \right)\xi^\nu - \Box_{\hat{g}} \left( R_{\hat{g}}'.h \right) \xi_\mu-  \Lambda ( R_{\hat{g}}'.h ) \xi_\mu \\&-4 \Lambda D \left( {G_g}_{\mu \nu} + \Lambda g_{\mu \nu} \right).h \xi^\nu \\
&= \hat{\nabla}_{\tau} \left( \hat{\nabla}_{\mu} \left(  R_{\hat{g}}'.h \right) \xi^\tau - \hat{\nabla}^\tau \left( R_{\hat{g}}'.h \right) \xi_\mu + \left(  R_{\hat{g}}'.h \right) \hat{\nabla}^\tau \xi_\mu \right) -  ( R_{\hat{g}}'.h) \Box_{\hat{g}} \xi_\mu- \Lambda ( R_{\hat{g}}'.h ) \xi_\mu \\& -4 \Lambda D \left({ G_g}_{\mu \nu} +  \Lambda g_{\mu \nu} \right).h \xi^\nu.
\end{aligned}
$$

Since $\xi$ is Killing, one has:
$$
\Box_{\hat{g}} \xi_\mu = - \hat{\nabla}_{\tau \mu} \xi^\tau = -{\mathrm{Riem}_{\hat{g}}}^\tau_{\quad \lambda \tau  \mu  } \xi^\lambda = - {\mathrm{Ric}_{\hat{g}}}_{\mu \tau} \xi^\tau =-\Lambda \xi_\mu.
$$

Then, thanks to classical GR theory (which we recall in proposition \ref{pqGR} below) 
we deduce that: \begin{equation} \label{031220201711} \Lambda D \left( {G_g}_{\mu \nu} +  \Lambda g_{\mu \nu} \right).h \xi^\nu=\Lambda D \left( {G_g}_{\mu \nu} - (-  \Lambda) g_{\mu \nu} \right).h \xi^\nu = \Lambda \mathcal{P}_{\hat{g}}^{GR} (\xi, h)_\mu = -\Lambda \hat{\nabla}^\tau \left( {\mathcal{Q}^{GR}_{\hat{g}}}_{  \tau \mu}  ( \xi,h) \right).\end{equation}
From this, we deduce:

\begin{equation}
\label{031220201427}
\begin{aligned}
D\left({ A_{{g}}}_{\mu\nu}^{(2)} \right).h \xi^\nu&= \hat{\nabla}^{\tau} \left( \hat{\nabla}_{\mu} \left(  R_{\hat{g}}'.h \right) \xi_\tau -  \hat{\nabla}_\tau \left(  R_{\hat{g}}'.h \right) \xi_\mu + \left(  R_{\hat{g}}'.h \right) \hat{\nabla}_\tau \xi_\mu +4 \Lambda{\mathcal{Q}^{GR}_{\hat{g}}}_{  \tau \mu}( \xi,h)  \right).
\end{aligned}
\end{equation}
Let us now consider $A_{\mu\nu}^{(1)}$. First, with \eqref{031220201311} and \eqref{031220201334}:
$$\begin{aligned}
D \left( \Box_g {G_g}_{\mu \nu} \right).h &= D \left( g^{\tau \kappa} \nabla_{\tau \kappa} {G_g}_{\mu \nu} \right).h \\
&= -h^{\tau \kappa} \hat{\nabla}_{\tau \kappa} {G_{\hat{g}}}_{\mu \nu} + \hat{g}^{\tau \kappa} \left[  \hat{\nabla}_\tau \left( D \left( \nabla_\kappa {G_g}_{\mu \nu} \right).h \right) - D\left( \Gamma^\rho_{\tau \kappa} (\varepsilon) \right).h \hat{\nabla}_\rho {G_{\hat{g}}}_{\mu \nu} \right. \\ & \left. -  D\left( \Gamma^\rho_{\tau \mu} (\varepsilon)\right).h \hat{\nabla}_\kappa {G_{\hat{g}}}_{\rho \nu} -  D\left( \Gamma^\rho_{\tau \nu} \right).h \hat{\nabla}_\rho {G_{\hat{g}}}_{\rho\nu} \right] \\
&= \hat{g}^{\tau \kappa}\hat{\nabla}_\tau \left( D \left( \nabla_\kappa {G_g}_{\mu \nu} \right).h \right) \\
&= \hat{g}^{\tau \kappa} \hat{\nabla}_\tau \left[  \hat{\nabla}_\kappa \left({G_{\hat{g}}}'.h_{\mu \nu} \right) -  D \left(\Gamma^\rho_{\kappa \mu}(\varepsilon)  \right).h{G_{\hat{g}}}_{\rho \nu}- D \left( \Gamma^\rho_{\kappa \nu}(\varepsilon) \right).h {G_{\hat{g}}}_{ \mu \rho} \right] \\
&=\Box_{\hat{g}} \  \left( {G_{\hat{g}}}'.h_{\mu \nu} \right) + \Lambda \nabla^\tau \left[ D \left(\Gamma^\rho_{\tau \mu}  \right).h \hat{g}_{\rho \nu} + D \left( \Gamma^\rho_{\tau \nu} \right).h \hat{g}_{\mu \rho} \right] \\
&= \Box_{\hat{g}}  \left( {G_{\hat{g}}}'.h_{\mu \nu} \right) + \Lambda \Box_{\hat{g}} h_{\mu \nu} \\
&= \Box_{\hat{g}} \left( D\left( {G_g}_{\mu \nu} + \Lambda g_{\mu \nu} \right).h \right).
\end{aligned}$$
Then:
\begin{equation}
\label{071220201627}
\begin{aligned}
D \left( \Box_g {G_g}_{\mu \nu} \right).h \xi^\nu &= \hat{\nabla}^\tau \left[ \hat{\nabla}_\tau \left(  D\left( {G_g}_{\mu \nu} + \Lambda g_{\mu \nu} \right).h \right) \xi^\nu -   \ D\left( {G_g}_{\mu \nu} + \Lambda g_{\mu \nu} \right).h \hat{\nabla}_\tau \xi^\nu \right]  \\&+  D\left( {G_g}_{\mu \nu} + \Lambda g_{\mu \nu} \right).h \Box_{\hat{g}} \xi^\nu  \\
&=\hat{\nabla}^\tau \left[ \hat{\nabla}_\tau \left(  D\left( {G_g}_{\mu \nu} + \Lambda g_{\mu \nu} \right).h \right) \xi^\nu -   \ D\left( {G_g}_{\mu \nu} + \Lambda g_{\mu \nu} \right).h \hat{\nabla}_\tau \xi^\nu \right]  \\&- \Lambda D\left( {G_g}_{\mu \nu} + \Lambda {g}_{\mu \nu} \right).h \xi^\nu \\
&= \hat{\nabla}^\tau \left[  \hat{\nabla}_\tau \left(  D\left( G_{\mu \nu} + \Lambda g_{\mu \nu} \right).h \right) \xi^\nu -   \ D\left( G_{\mu \nu} + \Lambda g_{\mu \nu} \right).h \hat{\nabla}_\tau \xi^\nu + \Lambda {\mathcal{Q}_{\hat{g}}^{GR}}_{\tau \mu}(\xi, h) \right].
\end{aligned}
\end{equation}

Further:
$$\begin{aligned}
D \left[ \left({\mathrm{Riem}_g}^\tau_{\; \mu \lambda \nu } - \frac{{\mathrm{Ric}_g}^\tau_\lambda}{4} {g}_{\mu \nu} \right) {G_g}^\lambda_\tau \right].h \xi^\nu& = \left[ -\Lambda   \left({\mathrm{Ric}_{\hat{g}}}'.h_{\mu \nu}\right) + \Lambda \frac{ R_{\hat{g}}'.h}{4} \hat{g}_{\mu \nu} + \Lambda \frac{R_{\hat{g}}}{4} h_{\mu \nu}+ {\mathrm{Riem}_{\hat{g}}}^\tau_{\; \mu \lambda \nu }  D({G_g}^\lambda_\tau).h \right. \\ &\left. - \frac{2 \Lambda}{4*2} \hat{g}_{\mu \nu}D \left( -\frac{2}{2}R_g \right).h \right] \xi^\nu \\
&= \left[ - \Lambda   G_{\hat{g}}'.h_{\mu \nu}  - \Lambda \frac{ R_{\hat{g}}'.h}{4} \hat{g}_{\mu \nu} - \Lambda \frac{R_{\hat{g}}}{4} h_{\mu \nu}+ {\mathrm{Riem}_{\hat{g}}}^\tau_{\; \mu \lambda \nu }D({G_g}^\lambda_\tau).h\right. \\ &\left. +\Lambda \frac{ R_{\hat{g}}'.h }{4} \hat{g}_{\mu \nu}  \right] \xi^\nu  \\
&= -\Lambda D\left[  {G_g}_{\mu \nu}  + \Lambda g_{\mu \nu} \right].h \xi^\nu + {\mathrm{Riem}_{\hat{g}}}^\tau_{\; \mu \lambda \nu } D({G_g}^\lambda_\tau).h \xi^\nu.
\end{aligned}$$
Thanks to \eqref{031220201711}, we know how to deal with the first term. Let us then consider:
$$ 
\begin{aligned}
{\mathrm{Riem}_{\hat{g}}}^\tau_{\; \mu  \lambda \nu}D({G_g}^\lambda_\tau).h \xi^\nu &= D({G_g}^\lambda_\tau).h \left( { \hat{\nabla}^\tau}_{ \, \, \mu} \xi_\lambda - \hat{\nabla}_{\mu }^{ \, \,  \, \, \, \tau} \xi_\lambda \right) \\
&=D({G_g}^\lambda_\tau).h \hat{\nabla}^\tau_{\, \,\mu} \xi_\lambda \\
&= \hat{\nabla}^\tau \left[ D({G_g}^\lambda_\tau).h \hat{\nabla}_\mu \xi_\lambda \right] - \hat{\nabla}^\tau \left( D({G_g}^\lambda_\tau).h \right) \hat{\nabla}_\mu \xi_\lambda.
\end{aligned}$$

In addition, since: ${\nabla}_\tau {G_{g}}^\tau_\lambda= 0$, one has: $D \left( \nabla_\tau {G_g}^\tau_\lambda \right).h = \hat{\nabla}_\tau \left( D\left({ G_g}^\tau_\lambda \right).h \right) + D \left( \Gamma^\tau_{\, \tau \rho}(\varepsilon) \right).h {G_{\hat{g}}}^\rho_\lambda - D \left( \Gamma^\rho_{\, \tau \lambda}(\varepsilon) \right).h {G_{\hat{g}}}^\tau_\rho=0$, which means 
$$\begin{aligned}
\hat{\nabla}_\tau \left( D\left( {G_g}^\tau_\lambda \right).h \right)& = D \left( \Gamma^\rho_{\, \tau \lambda}(\varepsilon) \right).h {G_{\hat{g}}}^\tau_\rho-D \left( \Gamma^\tau_{\, \tau \rho}(\varepsilon) \right).h {G_{\hat{g}}}^\rho_\lambda \\
&=-\Lambda \left(  D \left( \Gamma^\tau_{\, \tau \lambda} (\varepsilon)\right).h +D \left( \Gamma^\tau_{\, \tau \lambda} (\varepsilon)\right).h\right) = 0.
\end{aligned}$$

Thus:
\begin{equation}
\label{031220201932}
{\mathrm{Riem}_{\hat{g}}}^\tau_{\; \mu \lambda \nu } D \left( {G_g}^\lambda_\tau \right).h \xi^\nu = \hat{\nabla}_\tau \left( D\left( {G_g}^\tau_\lambda \right).h \nabla_\mu \xi^\lambda \right)= \hat{\nabla}^\tau \left( D \left[ {G_g}_{\tau \nu} + \Lambda g_{\tau \nu }\right].h \hat{\nabla}_\mu \xi^\nu \right).
\end{equation}
 Injecting \eqref{031220201711} and \eqref{031220201932} and into the previous computation yields:
$$\begin{aligned}
D \left[ \left({\mathrm{Riem}_g}^\tau_{\; \mu \lambda \nu } - \frac{{\mathrm{Ric}_g}^\tau_\lambda}{4} g_{\mu \nu} \right) {G_g}^\lambda_\tau \right].h \xi^\nu& = \hat{\nabla}^\tau \left( \Lambda {\mathcal{Q}^{GR}_{\hat{g}}}_{ \tau \mu}( \xi, h) +D\left[  {G_g}_{\tau \nu} + \Lambda {G_g}_{\tau \nu }\right].h\nabla_\mu \xi^\nu \right).
\end{aligned}$$

From \eqref{071220201627} and \eqref{031220201932}, we  deduce that:
\begin{equation}
\label{031220201732ais}
\begin{aligned}
D\left( {A_g}^{(1)}_{\mu \nu} \right).h \xi^\nu&=
D \left( \Box_g {G_g}_{\mu \nu}  +2 \left( {\mathrm{Riem}_g}^\tau_{\; \mu \lambda \nu } - \frac{{\mathrm{Ric}_g}^\tau_\lambda}{4}\right) {G_g}^\lambda_\tau \right).h \xi^\nu \\&= \hat{\nabla}^\tau \left[  \hat{\nabla}_\tau \left(  D\left( {G_{{g}}}_{\mu \nu} + \Lambda g_{\mu \nu} \right).h \right) \xi^\nu -   \ D\left( {G_g}_{\mu \nu} + \Lambda g_{\mu \nu} \right).h \hat{\nabla}_\tau \xi^\nu + \Lambda {\mathcal{Q}^{GR}_{\hat{g}}}_{ \tau \mu}(\xi, h)   \right. \\&\left.+2 \Lambda {\mathcal{Q}_{\hat{g}}^{GR}}_{ \tau \mu}  +2 D\left(  {G_g}_{\tau \nu} + \Lambda g_{\tau \nu }\right).h\hat{\nabla}_\mu \xi^\nu \right] \\
&= \hat{\nabla}^\tau \left[  \hat{\nabla}_\tau \left(  D\left( {G_g}_{\mu \nu} + \Lambda g_{\mu \nu} \right).h  \xi^\nu \right) -  2 D\left( {G_g}_{\mu \nu} + \Lambda g_{\mu \nu} \right).h \hat{\nabla}_\tau \xi^\nu  +3  \Lambda {\mathcal{Q}^{GR}_{\hat{g}}}_{\tau \mu} \right. \\ & \left. +2 D\left(   {G_g}_{\tau \nu} + \Lambda g_{\tau \nu }\right).h \hat{\nabla}_\mu \xi^\nu    \right] \\
&= \hat{\nabla}^\tau \left[  \hat{\nabla}_{\tau} {\mathcal{P}^{GR}_{\hat{g}}}_\mu (\xi, h)+
2\left[D\left(  G_{\tau \nu} + \Lambda g_{\tau \nu }\right).h \hat{\nabla}_\mu \xi^\nu - D\left( G_{\mu \nu} + \Lambda g_{\mu \nu} \right).h \hat{\nabla}_\tau \xi^\nu \right] \right. \\ &\left.+3\Lambda {\mathcal{Q}^{GR}_{\hat{g}}}_{ \tau \mu}(\xi, h) \right].
\end{aligned}
\end{equation}
Further, one has:

$$\begin{aligned}
\hat{\nabla}^\tau d  {\mathcal{P}_{\hat{g}}}^{GR}_{\tau\mu} (\xi, h) &= \hat{\nabla}^\tau \left( \hat{\nabla}_\tau  {\mathcal{P}_{\hat{g}}}^{GR}_\mu (\xi, h) - \hat{\nabla}_\mu  {\mathcal{P}_{\hat{g}}}^{GR}_\tau (\xi,h) \right)  \\
&= \hat{\nabla}^\tau \left( \hat{\nabla}_\tau  {\mathcal{P}_{\hat{g}}}^{GR}_\mu ( \xi, h)\right) - \hat{\nabla}_\mu \left(  \hat{\nabla}^\tau {\mathcal{P}_{\hat{g}}}^{GR}_\tau ( \xi, h) \right) - \mathrm{Ric}^\kappa_\mu  {\mathcal{P}_{\hat{g}}}^{GR}_\kappa (\xi, h) \\
&= \hat{\nabla}^\tau \left( \hat{\nabla}_\tau  {\mathcal{P}_{\hat{g}}}^{GR}_\mu (\xi, h)\right) - \Lambda  {\mathcal{P}_{\hat{g}}}^{GR}_\mu (\xi, h)  \\
&= \hat{\nabla}^\tau \left( \hat{\nabla}_\tau  {\mathcal{P}_{\hat{g}}}^{GR}_\mu ( \xi, h)\right)  +\Lambda  \hat{\nabla}^\tau \left( {\mathcal{Q}_{\hat{g}}^{GR}}_{ \tau \mu} ( \xi, h) \right),
\end{aligned}$$
using \eqref{Linearised-EMconservation}. 
Injecting the above into \eqref{031220201732ais} yields:

\begin{equation}
\label{031220201732}
\begin{aligned}
D\left( {A_g}^{(1)}_{\mu \nu} \right).h \xi^\nu
&= \hat{\nabla}^\tau \left[   d  {\mathcal{P}_{\hat{g}}}^{GR}_{\tau\mu} ( \xi, h) +
2\left[D\left(  G_{\tau \nu} + \Lambda g_{\tau \nu }\right).h \hat{\nabla}_\mu \xi^\nu - D\left( G_{\mu \nu} + \Lambda g_{\mu \nu} \right).h \hat{\nabla}_\tau \xi^\nu \right] \right. \\ &\left.+ 2\Lambda {\mathcal{Q}^{GR}_{\hat{g}}}_{\tau \mu}(\xi, h) \right].
\end{aligned}
\end{equation}

Combining \eqref{031220201427} and  \eqref{031220201732} yields:
\begin{equation}
\label{071220201715}
\begin{aligned}
D({A_{\hat{g}}}_{\mu \nu } ). h \xi^\nu &= \beta D\left( {A_{\hat{g}}}^{(2)}_{\mu \nu} \right).h \xi^\nu- \left( 2 \alpha + \beta \right) D\left( {A_{\hat{g}}}^{(1)}_{\mu \nu} \right).h \xi^\nu \\
&= \nabla^\tau \left[   \beta \left(
 d  {\mathcal{P}_{\hat{g}}}^{GR}_{\tau\mu} ( \xi, h)+
2\left[D\left(  {G_g}_{\tau \nu} + \Lambda g_{\tau \nu }\right).h \nabla_\mu \xi^\nu - D\left( {G_g}_{\mu \nu} + \Lambda g_{\mu \nu} \right).h \hat{\nabla}_\tau \xi^\nu \right] \right. \right. \\ &\left. \left.+ 2 \Lambda {\mathcal{Q}^{GR}_{\hat{g}}}_{ \tau \mu}( \xi, h) \right) - \left( 2 \alpha + \beta \right) \left( \hat{\nabla}_{\mu} \left(  R_{\hat{g}}'.h \right) \xi_\tau - \hat{\nabla}_\tau \left(  R_{\hat{g}}'.h \right) \xi_\mu + \left(  R_{\hat{g}}'.h \right) \nabla_\tau \xi_\mu \right. \right. \\ & \left. \left.  +4 \Lambda{\mathcal{Q}^{GR}_{\hat{g}}}_{ \tau \mu}(\xi, h)  \right)
\right] \\
&=\nabla^\tau \left[   \beta \left(
 d  {\mathcal{P}_{\hat{g}}}^{GR}_{\tau\mu} (\xi, h)+
2\left[D\left(  {G_g}_{\tau \nu} + \Lambda g_{\tau \nu }\right).h \nabla_\mu \xi^\nu - D\left( {G_g}_{\mu \nu} + \Lambda g_{\mu \nu} \right).h \hat{\nabla}_\tau \xi^\nu \right]  \right) \right.\\&\left. - \left( 2 \alpha + \beta \right) \left( \hat{\nabla}_{\mu} \left(  R_{\hat{g}}'.h \right) \xi_\tau - \hat{\nabla}_\tau \left(  R_{\hat{g}}'.h \right) \xi_\mu + \left(  R_{\hat{g}}'.h \right) \nabla_\tau \xi_\mu \right) \right. \\ & \left.  -2(4\alpha + \beta) \Lambda{\mathcal{Q}^{GR}_{\hat{g}}}_{ \tau \mu}(\xi, h)  
\right].
\end{aligned}
\end{equation}

Which gives us the formula for $\mathcal{Q}_{\nu \mu}(\xi,h)$, expressed as a function of the Einstein and scalar curvatures. One may prefer to express it in terms of Ricci and scalar curvatures: 
\begin{equation}
\label{eqglobcase}
\begin{aligned}
D\left( {A_{\hat{g}}}_{\mu \nu } \right).h \xi^\nu &= \hat{ \nabla}^\tau \left[ \beta  d  {\mathcal{P}_{\hat{g}}}^{GR}_{\tau\mu} (\xi, h)+ 2 \beta \left( \mathrm{Ric}_{\hat{g}}'.h_{\tau\nu} \hat{\nabla}_\mu \xi^\nu   -\mathrm{Ric}_{\hat{g}}'.h_{\mu \nu} \hat{\nabla}_\tau \xi^\nu \right) \right.\\&\left. - (2 \alpha + \beta ) \left( \hat{\nabla}_\mu (R_{\hat{g}}'.h) \xi_\tau -\hat{ \nabla}_\tau( R_{\hat{g}}'.h)\xi_\mu \right) - ( 2 \alpha - \beta )R_{\hat{g}}'.h \nabla_\tau \xi_\mu  \right.\\&\left.- 2(4 \alpha+ \beta) \Lambda{ \mathcal{Q}^{GR}_{\hat{g}}}_{  \tau \mu }(h , \xi)  -2\beta\Lambda \left( h_{\tau \nu} \hat{\nabla}_\mu \xi^\nu - h_{\mu \nu} \hat{\nabla}_\tau \xi^\nu \right)
\right].
\end{aligned}
\end{equation}

The Ricci flat case ($\Lambda=0$) will be of special importance, with the following formula:
\begin{equation}
\label{linearisedfourthordercharge}
\begin{aligned}
D\left( {A_{\hat{g}}}_{\mu \nu } \right).h \xi^\nu &= \hat{ \nabla}^\tau \left[ \beta  d  {\mathcal{P}_{\hat{g}}}^{GR}_{\tau\mu} (\xi, h) + 2 \beta \left( \mathrm{Ric}_{\hat{g}}'.h_{\tau\nu} \hat{\nabla}_\mu \xi^\nu   -\mathrm{Ric}_{\hat{g}}'.h_{\mu \nu} \hat{\nabla}_\tau \xi^\nu \right) \right.\\&\left. + (2 \alpha + \beta ) \left(\hat{ \nabla}_\tau( R_{\hat{g}}'.h)\xi_\mu - \hat{\nabla}_\mu (R_{\hat{g}}'.h) \xi_\tau  \right)+ ( \beta - 2 \alpha )R_{\hat{g}}'.h \nabla_\tau \xi_\mu
\right].
\end{aligned}
\end{equation}

In all three cases, the term inside the divergence is explicitly antisymmetric, and is thus a $2$-form.
\qed
\end{proof}

\begin{remark}
These computations intersect those in \cite{ADT}, done when $\hat{g}$ is the Schwarzschild-de Sitter (or AdS) metric (see (8) in \cite{ADT}) and when the pertubation remains Einstein with the same constant (according to (34)). We obtain, in this more general case, the same formula (which can be checked by comparing \eqref{071220201715} to (31) of \cite{ADT}).
\end{remark}

In the following proposition we will recall $\mathcal{P}^{GR}$ explicitly (see for instance \cite{CB-book,ADT}), which will be useful in what follows.

\begin{prop}\label{pqGR}
Let $(V,\hat{\bar{g}})$ be a smooth Einstein space-time admitting a Killing field $\xi$. Then, for any perturbation $h$, the conserved 1-form $\mathcal{P}^{GR}(\xi, h)_{\mu}\doteq { D\left(G+ \Lambda g \right)_{\hat{g}}\cdot h}_{\nu\mu}\xi^{\nu}$ is co-exact and thus can be written as
\begin{align}\label{ADM-energymomentum.01}
\mathcal{P}^{GR}(h,\xi)=\delta_{\hat{g}}\mathcal{Q}^{GR}(\xi, h),
\end{align}
where the 2-form $\mathcal{Q}^{GR}(\xi, h)$ is given by
\begin{align}\label{ADM-energymomentum.1}
\mathcal{Q}^{GR}_{\beta\mu}=\hat{\nabla}^{\nu}\mathcal{K}_{\beta\mu\alpha\nu}\xi^{\alpha} - \mathcal{K}_{\beta\nu\alpha\mu}\hat{\nabla}^{\nu}\xi^{\alpha} ,
\end{align}
where, if we write $H_{\mu \nu } \doteq h_{\mu \nu} - \frac{1}{2}h^\sigma_\sigma \hat{g}_{\mu \nu}$:
\begin{align*}
\mathcal{K}_{\beta\nu\alpha\mu}&\doteq \frac{1}{2}\left(\hat{g}_{\mu\beta}H_{\nu\alpha} + \hat{g}_{\nu\alpha}H_{\mu\beta} - \hat{g}_{\alpha\beta}H_{\mu\nu} - \hat{g}_{\mu\nu}H_{\alpha\beta} \right).
\end{align*}
\end{prop}

Let us now analyse a particular case of interest, which has been discussed for instance in \cite{ADT}. This is the case where the perturbed metric $\bar{g}$ is Einstein with cosmological constant $\Lambda$. For instance, in \cite{ADT} it is stated that several terms associated to $\mathcal{E}_{\alpha,\beta}$ will not contribute to these kinds of solutions in a general way, which seems to omit some implicit assumptions. Let us now make these assumptions explicit in our case. 

\begin{coro}\label{ADTcoro}
Consider the same setting as in Proposition \ref{Energycharge.Eins}, assume $n=3$, and let $\bar{g}$ belong to a smooth curve of Einstein metrics $\bar{g}_{\lambda}$ with fixed cosmological constant $\Lambda$ such that $\bar{g}_0=\hat{\bar{g}}$. Then, it holds that
\begin{align*}
\mathcal{Q}_{\hat{\bar{g}}}(\xi, h)=2(4 \alpha+ \beta) \Lambda{ \mathcal{Q}^{GR}_{\hat{g}}}(\xi, h),
\end{align*}
where $\mathcal{Q}^{GR}_{\hat{\bar{g}}}(\xi, h)$ is given by (\ref{ADM-energymomentum.1}).
\end{coro}
\begin{proof}
In this case, since we have that $\mathrm{Ric}_{\bar{g}_{\lambda}}=\Lambda \bar{g}_{\lambda}$, we see that $D\mathrm{Ric}_{\hat{\bar{g}}}\cdot h=\Lambda h$, and similarly $DG_{\hat{\bar{g}}}.h=-\Lambda h$, which implies that $\mathcal{P}^{GR}_{\hat{\bar{g}}}\equiv 0$. Also, since $R_{\bar{g}_{\lambda}}=4\Lambda$ for all $\lambda$, then $DR_{\hat{\bar{g}}}\cdot h\equiv 0$. All this already simplifies (\ref{eqglobcase}) to
\begin{align*}
D\left( {A_{\hat{g}}}_{\mu \nu } \right).h \xi^\nu &= \hat{ \nabla}^\tau \left[ -2(4 \alpha+ \beta) \Lambda{ \mathcal{Q}^{GR}_{\hat{g}}}_{ \tau \mu }( \xi, h) +  2 \beta \left( \mathrm{Ric}_{\hat{g}}'.h_{\tau\nu} \hat{\nabla}_\mu \xi^\nu   -\mathrm{Ric}_{\hat{g}}'.h_{\mu \nu} \hat{\nabla}_\tau \xi^\nu \right) \right.\\
&\left.  - 2\beta\Lambda \left( h_{\tau \nu} \hat{\nabla}_\mu \xi^\nu - h_{\mu \nu} \hat{\nabla}_\tau \xi^\nu \right)
\right],\\
&=-\hat{ \nabla}^\tau \left( 2(4 \alpha+ \beta) \Lambda{ \mathcal{Q}^{GR}_{\hat{g}}}_{  \tau \mu }(\xi, h)
\right),
\end{align*}
implying that $\mathcal{Q}_{\hat{\bar{g}}}(\xi, h)=2(4 \alpha+ \beta) \Lambda{ \mathcal{Q}^{GR}_{\hat{g}}}_{  \tau \mu}( \xi, h)$, where $\mathcal{Q}^{GR}_{\hat{\bar{g}}}(\xi, h)$ is given by (\ref{ADM-energymomentum.1}).
\qed
\end{proof}

\bigskip
The above corollary implies that for these special families of Einstein metrics their fourth-order conserved quantities are given in practice by their second order ones as solutions of the Einstein equations. Nevertheless, notice that the above simple corollary depends crucially on the fact that the cosmological constant is kept fixed through the whole family. In this context, families of Einstein metrics parametrized by their cosmological constant have been constructed in \cite{Avalos-Lira} for instance, and the above corollary does not give us information about the situation for them. Let us also put this discussion in perspective of (and as an extra motivation for) the analysis presented in Section 5.

\subsection{ADM formulation in asymptotically Minkowski spaces}
\label{subsection32}
In the present subsection we wish to develop an ADM-like formulation for the conserved energy $\mathcal{E}$. We will thus consider  a globally hyperbolic Asymptotically Minkowski manifold. Let us first recall definition \ref{AMmanifolds}:

\begin{defn}[AM space-times]\label{AMmanifoldsb}
Let $(M\times [0,T],\bar{g})$ be a regularly sliced globally hyperbolic manifold, where $M$ is Euclidean at infinity. Thus, let us write 
{$\hat{g}=-\hat{N}^2dt^2+ g_t$} where $\hat{N}$ stands for the associated lapse function,  $\hat{X}$ denotes the shift vector associated to the orthogonal space-time splitting and $g_t$, {restrict to} a time-dependent Riemannian metric on $M$ {when applied to vectors tangent to $M$}. We will say that $\hat{g}$ is \textbf{asymptotically-Minkowskian} of order $\tau>0$ with respect to some end coordinate system $\Phi:E_i\mapsto \mathbb{R}^n\backslash\bar{B}$, if, in such coordinates:
\begin{align}\label{Falloff-AM.1}
\begin{split}
\hat{g}_{ti}&=\hat{X}_i(.,t)=O_{4}(r^{-\hat{\tau}}),\\
\hat{g}_{ij}&=g_{ij}(.,t)=\delta_{ij}+O_{4}(r^{-\hat{\tau}}),\\
\hat{g}_{tt}&=- \left[\hat{N}(.,t)^2 - \left| \hat{X} \right|^2_{\hat{g}} \right]= - \hat{N}^2+ O_4 \left(r^{-2\hat{\tau} } \right) = -1+O_{4}(r^{-\hat{\tau}}).
\end{split}
\end{align}
\end{defn} 

Here by $O_a \left(r^{b} \right)$ we mean that the estimate is differentiable $a$ times by taking one power each time. For instance if $f=O_4\left( r^{-\hat{\tau} } \right)$, $\nabla f = O_3\left( r^{-\hat{\tau} -1 } \right)$,$\nabla^2 f=O_2\left( r^{-\hat{\tau}-2 } \right)$, and so on.

On such an AM manifold we consider a compact exhaustion $(K_r, \partial K_r)$ of $M$. We will denote $n$ the future pointing timelike unit normal vector to $M$ in $V$ and $\nu$ the outward pointing normal to $\partial K_r$ in $M$. Along the line of the introductions of the ADM mass we wish to find an explicit formula for $\mathcal{E}$ as in integral over the sphere at infinity (as in \eqref{ADMconservedchages.1}).  We will thus need to consider a Killing vector field $\xi$ of $\hat{g}$. We will  assume that, at least outside a compact set, $\xi$ agrees with the time coordinate vector $\partial_t$.  For $r$ large enough we will thus assume that $\xi= \partial_t$. Comparing \eqref{linearisedfourthordercharge} and \eqref{eqglobcase} shows that the leading terms in both cases are the same, and we will thus restrict ourselves to the Ricci flat case.


Let us  then consider such an AM Ricci-flat solution decomposed as explained above.  We need to compute
\begin{align*}
\mathcal{Q}_{\hat{g}}(\xi, h)_{\lambda\mu}&=-\Big\{\beta d\mathcal{P}^{GR}_{\hat{g}}(\xi, h)_{\lambda\mu}  +  2\beta({\mathrm{Ric}'_{\hat{g}}}\cdot h_{\lambda\nu}\hat{\nabla}_{\mu}\xi^{\nu} - \mathrm{Ric}'_{\hat{g}}\cdot h_{\mu\nu}\hat{\nabla}_{\lambda}\xi^{\nu}) \\
&+ (\beta - 2\alpha)\hat{\nabla}_{\lambda}\xi_{\mu}(R'_{\hat{g}}\cdot h) + ( \beta  + 2\alpha)\left(\hat{\nabla}_{\lambda}(R'_{\hat{g}}\cdot h) \xi_{\mu}  - \hat{\nabla}_{\mu}(R'_{\hat{g}}\cdot h)\xi_{\lambda}\right)\Big\},
\end{align*}
where
\begin{align*}
\mathcal{P}^{GR}_{\mu}&=-\hat{g}^{\nu\beta}\hat{\nabla}_{\nu}\mathcal{Q}^{GR}_{\beta\mu},\\
\mathcal{Q}^{GR}_{\beta\mu}&=\hat{\nabla}^{\nu}\mathcal{K}_{\beta\mu\alpha\nu}\xi^{\alpha} - \mathcal{K}_{\beta\alpha\nu\mu}\hat{\nabla}^{\alpha}\xi^{\nu}.
\end{align*}
Let us now assume that the perturbation $g$ of $\hat{g}$ has the fall-off behaviour
\begin{align}\label{Falloff-AM.2}
\begin{split}
{g}_{ti}&={X}_i(.,t)=O_{4}(r^{-{\tau}}),\\
{g}_{ij}&=g_{ij}(.,t)=\delta_{ij}+O_{4}(r^{-{\tau}}),\\
{g}_{tt}&=- \left[{N}(.,t)^2 - \left| X \right|^2_g \right]= - {N}^2+ O_4 \left(r^{-2{\tau} } \right) = -1+O_{4}(r^{-{\tau}}).
\end{split}
\end{align}
where $ \tau <\hat{\tau}$ and thus a perturbation of the form $g-\hat{g}=O_{4}(r^{-\tau})$. 
Then,
\begin{align*}
\mathcal{Q}^{GR}_{\nu\beta}&=\hat{\nabla}^{j}\mathcal{K}_{\nu\beta tj} - \mathcal{K}_{\nu \mu\alpha\beta}\hat{\nabla}^{\mu}\xi^{\alpha}
\end{align*}
thus
\begin{align*}
\mathcal{P}^{GR}_{\beta}&=-\hat{\nabla}^{\nu}\left( \hat{\nabla}^{j}\mathcal{K}_{\nu\beta tj} - \mathcal{K}_{\nu \mu\alpha\beta}\hat{\nabla}^{\mu}\xi^{\alpha}\right)= -\hat{\nabla}^{\nu}\hat{\nabla}^{j}\mathcal{K}_{\nu\beta tj} + \hat{\nabla}^{\nu}\left(\mathcal{K}_{\nu \mu\alpha\beta}\hat{\nabla}^{\mu}\xi^{\alpha}\right),\\
&= - \hat{\nabla}^{t}\hat{\nabla}^{j}\mathcal{K}_{t\beta tj} - \hat{\nabla}^{i}\hat{\nabla}^{j}\mathcal{K}_{i\beta tj} + \hat{\nabla}^{\nu}\left(\mathcal{K}_{\nu \mu\alpha\beta}\hat{\nabla}^{\mu}\xi^{\alpha}\right).
\end{align*}
Then, it follows that
\begin{align*}
d\mathcal{P}^{GR}_{\lambda\beta}&= - \hat{\nabla}_{\lambda}\hat{\nabla}^{t}\hat{\nabla}^{j}\mathcal{K}_{t\beta tj} - \hat{\nabla}_{\lambda}\hat{\nabla}^{i}\hat{\nabla}^{j}\mathcal{K}_{i\beta tj} + \hat{\nabla}_{\lambda}\left(\hat{\nabla}^{\nu}\left(\mathcal{K}_{\nu \mu\alpha\beta}\hat{\nabla}^{\mu}\xi^{\alpha}\right)\right) \\
&+ \hat{\nabla}_{\beta}\left( \hat{\nabla}^{t}\hat{\nabla}^{j}\mathcal{K}_{t\lambda tj} + \hat{\nabla}^{i}\hat{\nabla}^{j}\mathcal{K}_{i\lambda tj} - \hat{\nabla}^{\nu}\left(\mathcal{K}_{\nu \mu\alpha\lambda}\hat{\nabla}^{\mu}\xi^{\alpha}\right) \right)
\end{align*}
where 
\begin{align*}
\mathcal{K}_{\beta\nu\alpha\mu}\doteq \frac{1}{2}\left(\hat{g}_{\mu\beta}H_{\nu\alpha} + \hat{g}_{\nu\alpha}H_{\mu\beta} - \hat{g}_{\alpha\beta}H_{\mu\nu} - \hat{g}_{\mu\nu}H_{\alpha\beta} \right).
\end{align*}

Recall that we are interested in the object
\begin{align*}
-\mathcal{Q}(\hat{n},\hat{\nu})&=-\mathcal{Q}\left(\frac{1}{\hat{N}}(\partial_t-\hat{X}),\hat{\nu} \right)=-\frac{1}{\hat{N}}\mathcal{Q}(\partial_t,\hat{\nu}) + \frac{1}{\hat{N}} \mathcal{Q}(\hat{X},\hat{\nu}),\\
&=-\frac{1}{\hat{N}}\mathcal{Q}_{tj}\hat{\nu}^{j} + \frac{1}{\hat{N}} \mathcal{Q}_{ij}\hat{X}^i\hat{\nu}^j.
\end{align*}
We are therefore mainly interested in the components
\begin{align*}
\mathcal{Q}_{tj}&=-\Big\{\beta d\mathcal{P}^{GR}_{\hat{g}}(\xi, h)_{tj}  +  2\beta({\mathrm{Ric}'_{\hat{g}}}\cdot h_{t\nu}\hat{\nabla}_{j}\xi^{\nu} - \mathrm{Ric}'_{\hat{g}}\cdot h_{j\nu}\hat{\nabla}_{t}\xi^{\nu}) \\
&+ (\beta - 2\alpha)\hat{\nabla}_{t}\xi_{j}(R'_{\hat{g}}\cdot h) + ( \beta  + 2\alpha)\left(\hat{\nabla}_{t}(R'_{\hat{g}}\cdot h) \xi_{j}  - \hat{\nabla}_{j}(R'_{\hat{g}}\cdot h)\xi_{t}\right)\Big\},\\
&=-\beta d\mathcal{P}^{GR}_{\hat{g}}(\xi, h)_{tj} + ( \beta  + 2\alpha) \hat{\nabla}_{j}(R'_{\hat{g}}\cdot h) -  2\beta({\mathrm{Ric}'_{\hat{g}}}\cdot h_{t\nu}\hat{\nabla}_{j}\xi^{\nu} - \mathrm{Ric}'_{\hat{g}}\cdot h_{j\nu}\hat{\nabla}_{t}\xi^{\nu}) \\
&- (\beta - 2\alpha)\hat{\nabla}_{t}\xi_{j}(R'_{\hat{g}}\cdot h)
\end{align*}
Therefore, we compute the following expressions.
\begin{align}\label{ADM.0}
\begin{split}
-d\mathcal{P}^{GR}_{\hat{g}}(\xi, h)_{tj}&=\hat{\nabla}_{t}\hat{\nabla}^{t}\hat{\nabla}^{k}\mathcal{K}_{tj tk} + \hat{\nabla}_{t}\hat{\nabla}^{i}\hat{\nabla}^{k}\mathcal{K}_{ij tk} - \hat{\nabla}_{t}\left(\hat{\nabla}^{\nu}\mathcal{K}_{\nu \mu\alpha j}\hat{\nabla}^{\mu}\xi^{\alpha}\right) \\
&-\hat{\nabla}_{j}\left( \hat{\nabla}^{t}\hat{\nabla}^{k}\mathcal{K}_{tt tk} + \hat{\nabla}^{i}\hat{\nabla}^{k}\mathcal{K}_{it tk} - \hat{\nabla}^{\nu}\mathcal{K}_{\nu \mu\alpha t}\hat{\nabla}^{\mu}\xi^{\alpha} \right),\\
&=\hat{\nabla}_{t}\hat{\nabla}^{t}\hat{\nabla}^{k}\mathcal{K}_{tj tk} + \hat{\nabla}_{t}\hat{\nabla}^{i}\hat{\nabla}^{k}\mathcal{K}_{ij tk} - \hat{\nabla}_{j}\hat{\nabla}^{i}\hat{\nabla}^{k}\mathcal{K}_{it tk} \\
& + \hat{\nabla}_{j}\left(\hat{\nabla}^{\nu}\left(\mathcal{K}_{\nu \mu\alpha t}\hat{\nabla}^{\mu}\xi^{\alpha}\right)\right)  - \hat{\nabla}_{t}\left(\hat{\nabla}^{\nu}\left(\mathcal{K}_{\nu \mu\alpha j}\hat{\nabla}^{\mu}\xi^{\alpha}\right)\right).
\end{split}
\end{align}
Now, recalling that $H=h-\frac{1}{2}\mathrm{tr}_{\hat{g}}h\:\hat{g}$; considering $h\doteq g-\hat{g}$ as the perturbations and appealing to (\ref{Falloff-AM.1})-(\ref{Falloff-AM.2}), it also holds that
\begin{align*}
\mathcal{K}_{tj tk}&=\frac{1}{2}\left(\underbrace{{\hat{g}_{tk}H_{jt}}+ \hat{g}_{jt}H_{tk}}_{O_{4}(r^{-(\tau + \hat{\tau})})} - \underbrace{\hat{g}_{tt}}_{-1+O_4(r^{-\hat{\tau}})}H_{jk} - \underbrace{\hat{g}_{jk}}_{\delta_{jk}+O_{4}(r^{-\hat{\tau}})}H_{tt} \right),\\
\end{align*}
and noticing  that $\hat{g}_{j \alpha}V^\alpha  = \hat{g}_{jk}V^k  + O_4(r^{-\hat{\tau}}|V|)= V_k + O_4(r^{-\hat{\tau}}|V|) $, we can raise and lower space indexes at the cost of a decaying term.  Thus,
\begin{align}
\begin{split}
\hat{\nabla}_{t}\hat{\nabla}^{t}\hat{\nabla}^{k}\mathcal{K}_{tj tk}&=\frac{1}{2}\left(- \hat{g}_{tt}\hat{\nabla}_{t}\hat{\nabla}^{t}\hat{\nabla}^{k}H_{jk} - \hat{\nabla}_{t}\hat{\nabla}^{t}\hat{\nabla}_{j}H_{tt} + O_{1}(r^{-(\tau + \hat{\tau})-3}) \right),\\
&=\frac{1}{2}\left( \hat{\nabla}_{t}\hat{\nabla}_{t}\hat{\nabla}_{j}H_{tt} - \hat{\nabla}_{t}\hat{\nabla}_{t}\hat{\nabla}^{k}H_{jk} + O_{1}(r^{-(\tau + \hat{\tau})-3}) \right),
\end{split}
\end{align}
where 
\begin{align*}
H_{tt}&=h_{tt}-\frac{1}{2}\left(\underbrace{\hat{g}^{tt}}_{-1+O_{4}(r^{-\hat{\tau}})}h_{tt} + 2\underbrace{\hat{g}^{ti}}_{O_4(r^{-\hat{\tau}})}h_{ti} + \underbrace{\hat{g}^{ab}}_{\delta_{ab}+O_4(r^{-\hat{\tau}})}h_{ab} \right)\underbrace{\hat{g}_{tt}}_{-1+O_4(r^{-\hat{\tau}})},\\
&=h_{tt}+\frac{1}{2}\left(-h_{tt} + h_{aa}\right) + O_{4}(r^{-(\tau + \hat{\tau})}).
\end{align*}
That is, 
\begin{align*}
H_{tt}&=\frac{1}{2}\left(h_{tt} + h_{aa} \right) + O_4(r^{-(\tau + \hat{\tau})}).
\end{align*}
On the other hand, it also holds that
\begin{align*}
H_{jk}&=h_{jk}-\frac{1}{2}\left(-h_{tt} + h_{aa} + O_4(r^{-(\tau + \hat{\tau})})\right)\hat{g}_{jk},
\end{align*}
implying that
\begin{align*}
\hat{\nabla}_{t}\hat{\nabla}_{t}\hat{\nabla}_{j}H_{tt} - \hat{\nabla}_{t}\hat{\nabla}_{t}\hat{\nabla}^{k}H_{jk}&=\frac{1}{2}\hat{\nabla}_{t}\hat{\nabla}_{t}\hat{\nabla}_{j}\left(h_{tt} + h_{aa} \right) - \hat{\nabla}_{t}\hat{\nabla}_{t}\hat{\nabla}^{k}\left(h_{jk}-\frac{1}{2}\left(-h_{tt} + h_{aa} \right)\hat{g}_{kj}\right) \\
&+ O_1(r^{-(\tau + \hat{\tau})-3}),\\
&=\hat{\nabla}_{t}\hat{\nabla}_{t}\hat{\nabla}_{j}h_{ii} - \hat{\nabla}_{t}\hat{\nabla}_{t}\hat{\nabla}_{i}h_{ji}  + O_1(r^{-(\tau + \hat{\tau}) - 3})
\end{align*}
and hence
\begin{align}\label{ADM.-1}
\begin{split}
\hat{\nabla}_{t}\hat{\nabla}^{t}\hat{\nabla}^{k}\mathcal{K}_{tj tk}&=\frac{1}{2}\left( \hat{\nabla}_{t}\hat{\nabla}_{t}\hat{\nabla}_{j}h_{ii} - \hat{\nabla}_{t}\hat{\nabla}_{t}\hat{\nabla}_{i}h_{ji}\right) + O_{1}(r^{-(\tau + \hat{\tau})-3}).
\end{split}
\end{align}

Since $\hat{\nabla} h =\hat{\nabla} g - \hat{\nabla } \hat{g} =\hat{\nabla} g $ we can change  \eqref{ADM.-1} into:
\begin{equation}
\label{ADM.1}
\begin{aligned}
\hat{\nabla}_{t}\hat{\nabla}^{t}\hat{\nabla}^{k}\mathcal{K}_{tj tk}&=\frac{1}{2}\left( \hat{\nabla}_{t}\hat{\nabla}_{t}\hat{\nabla}_{j}g_{ii} - \hat{\nabla}_{t}\hat{\nabla}_{t}\hat{\nabla}_{i}g_{ji}\right) + O_{1}(r^{-(\tau + \hat{\tau})-3}).
\end{aligned}
\end{equation}
We will favour these expressions and make these substitutions in the following.

We can similarly compute that
\begin{align*}
\mathcal{K}_{it tk}&=\frac{1}{2}\left(\hat{g}_{ik}H_{tt}+\hat{g}_{tt}H_{ik} \underbrace{- \hat{g}_{it}H_{tk} - \hat{g}_{tk}H_{it} }_{O_4(r^{-(\tau + \hat{\tau})})}\right)\\
\hat{\nabla}_{j}\hat{\nabla}^{i}\hat{\nabla}^{k}\mathcal{K}_{it tk}&=\frac{1}{2}\left(\hat{\nabla}_{j}\hat{\nabla}^{i}\hat{\nabla}_{i}H_{tt} - \hat{\nabla}_{j}\hat{\nabla}^{i}\hat{\nabla}^{k}H_{ik}\right)  + {O_1(r^{-(\tau + \hat{\tau})-3})},
\end{align*}
which gives us
\begin{align*}
\hat{\nabla}_{j}\hat{\nabla}^{i}\hat{\nabla}^{k}\mathcal{K}_{it tk}&=\frac{1}{2}\left(\hat{\nabla}_{j}\hat{\nabla}^{i}\hat{\nabla}_{i} g_{aa}  -  \hat{\nabla}_{j}\hat{\nabla}^{i}\hat{\nabla}^{k}g_{ik}\right) + {O_1(r^{-(\tau + \hat{\tau})-3})}.
\end{align*}
Now, in order to compute $\hat{\nabla}_{t}\hat{\nabla}^{i}\hat{\nabla}^{k}\mathcal{K}_{ij tk}$, notice that
\begin{align*}
\mathcal{K}_{ij tk}&=\frac{1}{2}\left(\hat{g}_{ik}H_{jt} \underbrace{+ \hat{g}_{jt}H_{ik} - \hat{g}_{it}H_{jk}}_{O_4(r^{-(\tau + \hat{\tau})})} - \hat{g}_{jk}H_{it} \right),\\
&=\frac{1}{2}\left(\hat{g}_{ik}H_{jt} - \hat{g}_{jk}H_{it} \right) + O_4(r^{-(\tau + \hat{\tau})}).
\end{align*}
Thus, we need to compute 
\begin{align*}
H_{jt}=h_{jt} - \frac{1}{2}(-h_{tt}+h_{aa})\hat{g}_{jt} + O_4(r^{-(\tau + \hat{\tau})})=h_{jt} + O_4(r^{-(\tau + \hat{\tau})}),
\end{align*}
implying 
\begin{align*}
\mathcal{K}_{ij tk}&=\frac{1}{2}\left(\hat{g}_{ik}h_{jt} - \hat{g}_{jk}h_{it} \right) + O_4(r^{-(\tau + \hat{\tau})}),
\end{align*}
and thus
\begin{align}\label{ADM.3}
\hat{\nabla}_{t}\hat{\nabla}^{i}\hat{\nabla}^{k}\mathcal{K}_{ij tk}&=\frac{1}{2}\left(\hat{\nabla}_{t}\hat{\nabla}^{i}\hat{\nabla}_{i}g_{jt} - \hat{\nabla}_{t}\hat{\nabla}^{i}\hat{\nabla}_{j}g_{it} \right) + O_1(r^{-(\tau + \hat{\tau}) - 3}).
\end{align}
Therefore, putting together (\ref{ADM.0})-(\ref{ADM.3}), we find
\begin{align*}
\begin{split}
d\mathcal{P}^{GR}_{\hat{g}}(\xi, h)_{tj}
&=\frac{1}{2}\left(\hat{\nabla}_{j}\hat{\nabla}^{i}\hat{\nabla}_{i} g_{aa}  -  \hat{\nabla}_{j}\hat{\nabla}^{i}\hat{\nabla}^{k}g_{ik}\right) - \frac{1}{2}\left( \hat{\nabla}_{t}\hat{\nabla}_{t}\hat{\nabla}_{j}g_{ii} - \hat{\nabla}_{t}\hat{\nabla}_{t}\hat{\nabla}_{i}g_{ji}\right) \\
&- \frac{1}{2}\left(\hat{\nabla}_{t}\hat{\nabla}^{i}\hat{\nabla}_{i}g_{jt} - \hat{\nabla}_{t}\hat{\nabla}^{i}\hat{\nabla}_{j}g_{it} \right) + \hat{\nabla}_{t}\left(\hat{\nabla}^{\nu}\mathcal{K}_{\nu \mu\alpha j}\hat{\nabla}^{\mu}\xi^{\alpha}\right)\\
&- \hat{\nabla}_{j}\left(\hat{\nabla}^{\nu}\mathcal{K}_{\nu \mu\alpha t}\hat{\nabla}^{\mu}\xi^{\alpha}\right) + O_1(r^{-(\tau + \hat{\tau}) -3})
\end{split}
\end{align*}

Now let us compute that
\begin{align*}
R'_{\hat{g}}\cdot h&=\hat{\nabla}_{\lambda}\hat{\nabla}_{\alpha}h^{\lambda\alpha} - \hat{g}^{\alpha\beta}\hat{\nabla}_{\alpha}\hat{\nabla}_{\beta}h^{\lambda}_{\lambda},\\
&=\hat{\nabla}_{\lambda}\hat{\nabla}_{t}h^{\lambda t} +  \hat{\nabla}_{\lambda}\hat{\nabla}_{i}h^{\lambda i} - \underbrace{\hat{g}^{tt}}_{-1+O_4(r^{-\hat{\tau}})}\hat{\nabla}_{t}\hat{\nabla}_{t}h^{\lambda}_{\lambda} - \hat{\nabla}^{i}\hat{\nabla}_{i}h^{\lambda}_{\lambda} \underbrace{- 2\hat{g}^{ti}\hat{\nabla}_t\hat{\nabla}_ih^{\lambda}_{\lambda}}_{O_2(r^{-(\tau + \hat{\tau})-2})} + O_2(r^{-(\tau + \hat{\tau}) -2} )\\
&=\hat{\nabla}_{t}\hat{\nabla}_{t}h^{tt} + \hat{\nabla}_{i}\hat{\nabla}_{t}h^{it} + \hat{\nabla}_{t}\hat{\nabla}_{i}h^{it} +  \hat{\nabla}_{u}\hat{\nabla}_{i}h^{u i} + \hat{\nabla}_{t}\hat{\nabla}_{t}(-g_{tt} + g_{uu}) \\
&- \hat{\nabla}^{i}\hat{\nabla}_{i}(-g_{tt} + g_{uu}) + O_2(r^{-(\tau + \hat{\tau}) -2} ),
\end{align*}
Notice that 
\begin{align*}
h^{tt}&=\underbrace{\hat{g}^{tt}}_{-1+O_4(r^{-\hat{\tau}})}h^t_t + \overbrace{\hat{g}^{ti}}^{O_4(r^{-\hat{\tau}})}h^t_i=-\hat{g}^{tt}h_{tt} - \hat{g}^{ti}h_{ti} + O_4(r^{-(\tau + \hat{\tau})})=h_{tt} + O_4(r^{-(\tau + \hat{\tau})}),\\
h^{ti}&=-h_{ti}+O_4(r^{-(\tau + \hat{\tau})}).
\end{align*}
Therefore
\begin{align*}
R'_{\hat{g}}\cdot h&=\hat{\nabla}_{t}\hat{\nabla}_{t}g_{tt} - \hat{\nabla}_{i}\hat{\nabla}_{t}g_{it} - \hat{\nabla}_{t}\hat{\nabla}_{i}g_{it} +  \hat{\nabla}_{u}\hat{\nabla}_{i}g_{u i} -\hat{\nabla}_{t}\hat{\nabla}_{t}g_{tt} + \hat{\nabla}_{t}\hat{\nabla}_{t}g_{uu} \\
&+ \hat{\nabla}^{i}\hat{\nabla}_{i}g_{tt} - \hat{\nabla}^{i}\hat{\nabla}_{i}g_{uu} + O_2(r^{-(\tau + \hat{\tau}) -2} ),\\
&= \hat{\nabla}_{u}\hat{\nabla}_{i}g_{u i} + \hat{\nabla}^{i}\hat{\nabla}_{i}g_{tt} - 2\hat{\nabla}_{i}\hat{\nabla}_{t}g_{it} - \hat{\nabla}^{i}\hat{\nabla}_{i}g_{uu} + \hat{\nabla}_{t}\hat{\nabla}_{t}g_{uu} + O_2(r^{-(\tau + \hat{\tau}) -2} ).
\end{align*}
All this gives us that
\begin{align*}
-\mathcal{Q}_{tj}
&= \frac{\beta}{2}\left(\hat{\nabla}_{j}\hat{\nabla}^{i}\hat{\nabla}_{i} g_{aa}  -  \hat{\nabla}_{j}\hat{\nabla}^{i}\hat{\nabla}^{k}g_{ik}\right) - \frac{\beta}{2}\left( \hat{\nabla}_{t}\hat{\nabla}_{t}\hat{\nabla}_{j}g_{ii} - \hat{\nabla}_{t}\hat{\nabla}_{t}\hat{\nabla}_{i}g_{ji}\right) \\
&- \frac{\beta}{2}\left(\hat{\nabla}_{t}\hat{\nabla}^{i}\hat{\nabla}_{i}g_{jt} - \hat{\nabla}_{t}\hat{\nabla}^{i}\hat{\nabla}_{j}g_{it} \right)  \\
&+ ( \beta  + 2\alpha) \left( \hat{\nabla}_{j}\hat{\nabla}^{i}\hat{\nabla}_{i}g_{aa} - \hat{\nabla}_{j}\hat{\nabla}_{u}\hat{\nabla}_{i}g_{u i} - \hat{\nabla}_{j}\hat{\nabla}^{i}\hat{\nabla}_{i}g_{tt} + 2\hat{\nabla}_{j}\hat{\nabla}_{i}\hat{\nabla}_{t}g_{it} - \hat{\nabla}_{j}\hat{\nabla}_{t}\hat{\nabla}_{t}g_{uu}\right)   \\
&+ \beta\hat{\nabla}_{t}\left(\hat{\nabla}^{\nu}\left(\mathcal{K}_{\nu \mu\alpha j}\hat{\nabla}^{\mu}\xi^{\alpha}\right)\right) - \beta\hat{\nabla}_{j}\left(\hat{\nabla}^{\nu}\left(\mathcal{K}_{\nu \mu\alpha t}\hat{\nabla}^{\mu}\xi^{\alpha}\right)\right)\\
&+  2\beta({\mathrm{Ric}'_{\hat{g}}}\cdot h_{t\nu}\hat{\nabla}_{j}\xi^{\nu} - \mathrm{Ric}'_{\hat{g}}\cdot h_{j\nu}\hat{\nabla}_{t}\xi^{\nu}) + (\beta - 2\alpha)\hat{\nabla}_{t}\xi_{j}(R'_{\hat{g}}\cdot h)+ O_1(r^{-(\tau + \hat{\tau}) -3}).
\end{align*}
We can rearrange the above as follows
\begin{align*}
-\mathcal{Q}_{tj}&= ( \frac{3}{2}\beta  + 2\alpha) \left( \hat{\nabla}_{j}\hat{\nabla}^{i}\hat{\nabla}_{i}g_{aa} - \hat{\nabla}_{j}\hat{\nabla}_{u}\hat{\nabla}_{i}g_{u i}\right) - ( \beta  + 2\alpha)\hat{\nabla}_{j}\hat{\nabla}^{i}\hat{\nabla}_{i}g_{tt} \\
&+ \frac{\beta}{2}\left( \hat{\nabla}_{t}\hat{\nabla}_{t}\hat{\nabla}_{i}g_{ji} - \hat{\nabla}_{t}\hat{\nabla}_{t}\hat{\nabla}_{j}g_{ii} \right) - ( \beta  + 2\alpha)\hat{\nabla}_{j}\hat{\nabla}_{t}\hat{\nabla}_{t}g_{uu} \\
& + \frac{\beta}{2}\left( \hat{\nabla}_{t}\hat{\nabla}^{i}\hat{\nabla}_{j}g_{it} - \hat{\nabla}_{t}\hat{\nabla}^{i}\hat{\nabla}_{i}g_{jt} \right) + 2( \beta  + 2\alpha)\hat{\nabla}_{j}\hat{\nabla}_{i}\hat{\nabla}_{t}g_{it}   \\
&+ \beta\hat{\nabla}_{t}\left(\hat{\nabla}^{\nu}\left(\mathcal{K}_{\nu \mu\alpha j}\hat{\nabla}^{\mu}\xi^{\alpha}\right)\right) - \beta\hat{\nabla}_{j}\left(\hat{\nabla}^{\nu}\left(\mathcal{K}_{\nu \mu\alpha t}\hat{\nabla}^{\mu}\xi^{\alpha}\right)\right)\\
&+  2\beta({\mathrm{Ric}'_{\hat{g}}}\cdot h_{t\nu}\hat{\nabla}_{j}\xi^{\nu} - \mathrm{Ric}'_{\hat{g}}\cdot h_{j\nu}\hat{\nabla}_{t}\xi^{\nu}) + (\beta - 2\alpha)\hat{\nabla}_{t}\xi_{j}(R'_{\hat{g}}\cdot h)+ O_1(r^{-(\tau + \hat{\tau}) -3}).
\end{align*}
Now, notice that 
\begin{align*}
\hat{\nabla}_{\alpha}g_{\mu\nu}= \partial_\alpha g_{\mu \nu} - \hat{\Gamma}_{\alpha \mu}^\kappa g_{ \kappa \nu}  - \hat{\Gamma}_{\alpha \nu}^\kappa g_{  \mu \kappa } = \underbrace{\partial_{\alpha}g_{\mu\nu}}_{O_{4}(r^{-\tau-1})} + O_{3}(r^{-\hat{\tau}-1}) + O_3(r^{-\hat{\tau}- \tau -1}).
\end{align*}
Therefore, keeping track of top order terms
\begin{align*}
-\mathcal{Q}_{tj}&= \left( \frac{3}{2}\beta  + 2\alpha\right) \left( \partial_{j}\partial_{i}\partial_{i}g_{aa} - \partial_{j}\partial_{u}\partial_{i}g_{u i}\right) - ( \beta  + 2\alpha)\partial_{j}\partial_{i}\partial_{i}g_{tt} \\
&+ \frac{\beta}{2}\left( \partial_{t}\partial_{t}\partial_{i}g_{ji} - \partial_{t}\partial_{t}\partial_{j}g_{ii} \right) - ( \beta  + 2\alpha)\partial_{j}\partial_{t}\partial_{t}g_{uu} \\
& + \frac{\beta}{2}\left( \partial_{t}\partial_{i}\partial_{j}g_{it} - \partial_{t}\partial_{i}\partial_{i}g_{jt} \right) + 2( \beta  + 2\alpha)\partial_{j}\partial_{i}\partial_{t}g_{it}   \\
&+ \beta\hat{\nabla}_{t}\left(\hat{\nabla}^{\nu}\left(\mathcal{K}_{\nu \mu\alpha j}\hat{\nabla}^{\mu}\xi^{\alpha}\right)\right) - \beta\hat{\nabla}_{j}\left(\hat{\nabla}^{\nu}\left(\mathcal{K}_{\nu \mu\alpha t}\hat{\nabla}^{\mu}\xi^{\alpha}\right)\right)\\
&+  2\beta({\mathrm{Ric}'_{\hat{g}}}\cdot h_{t\nu}\hat{\nabla}_{j}\xi^{\nu} - \mathrm{Ric}'_{\hat{g}}\cdot h_{j\nu}\hat{\nabla}_{t}\xi^{\nu}) + (\beta - 2\alpha)\hat{\nabla}_{t}\xi_{j}(R'_{\hat{g}}\cdot h)+ O_1(r^{-\hat{\tau} -3})+O_1(r^{-(\hat{\tau}+\tau)-3}).
\end{align*}
Finally, let us assume that the $\hat{g}$-Killing vector $\xi$ obeys a fall-off behaviour of the form $\hat{\nabla}^\beta \xi^{\alpha}=O_3(r^{-\hat{\tau}-1})$. Since $\mathrm{Ric}'_{\hat{g}}\cdot h$ is a function of the $\hat{\nabla} \hat{\nabla} h$, one can compute as was done above and conclude that  $\mathrm{Ric}'_{\hat{g}}\cdot h=O_{2}(r^{-\tau-2})$. Similarly, straightforward computations show that $R'_{\hat{g}}\cdot h =\hat{\nabla}_{\lambda \alpha }h^{\lambda\alpha} - \Box_{\hat{g}}h^{\lambda}_{\lambda}-{\mathrm{Ric}_{\hat{g}}}^{\alpha \beta} h_{\alpha \beta}$. Since $\mathrm{Ric}_{\hat{g}} = O_2(r^{- \hat{\tau} -2})$ and $h= O_4(r^{- \tau})$, $R'_{\hat{g}}\cdot h=O_{2}(r^{-\tau-2})+O_2(r^{- ( \tau + \hat{\tau})-2} )=O_{2}(r^{-\tau-2})$ since $\hat{\tau}>0$.  We can work similarly on the $\mathcal{K}_{\nu \mu\alpha j}\hat{\nabla}^{\mu}\xi^{\alpha}$: since $h=O_4(r^{-\tau})$, $\mathcal{K}_{\nu \mu \alpha j}= O_4(r^{-\tau})$. Combined with the estimate on $\hat{\nabla}^{\mu}\xi^{\alpha}$,  we can assure that $\hat{\nabla}_{t}\left(\hat{\nabla}^{\nu}\left(\mathcal{K}_{\nu \mu\alpha j}\hat{\nabla}^{\mu}\xi^{\alpha}\right)\right) =O_1(r^{-(\tau+\hat{\tau})-3})$ .  Similarly, $\hat{\nabla}_{j}\left(\hat{\nabla}^{\nu}\left(\mathcal{K}_{\nu \mu\alpha t}\hat{\nabla}^{\mu}\xi^{\alpha}\right)\right)=O_1(r^{-(\tau+\hat{\tau})-3})$. Combining these estimates, on the initial hypersurface $M_0$ defined by the condition $t=0$, we finally find that
\begin{align}\label{ADM.5}
\begin{split}
-\mathcal{Q}_{tj}|_{t=0}&= \left( \frac{3}{2}\beta  + 2\alpha\right) \left( \partial_{j}\partial_{i}\partial_{i}g_{aa} - \partial_{j}\partial_{u}\partial_{i}g_{u i}\right) + \frac{\beta}{2}\left( \partial_{i}\ddot{g}_{ji} - \partial_{j}\ddot{g}_{ii} \right)   + \frac{\beta}{2} \left( \partial_{i}\partial_{j}\dot{X}_{i} - \partial_{i}\partial_{i}\dot{X}_{j} \right)  \\
&+ ( \beta  + 2\alpha)\partial_{j}\partial_{i}\partial_{i} \left[N^2 - |X|^2 \right] - ( \beta  + 2\alpha)\partial_{j}\ddot{g}_{ii}+ 2( \beta  + 2\alpha)\partial_{j}\partial_{i}\dot{X}_{i}+ O_1(r^{-\hat{\tau} -3})+O_1(r^{-(\hat{\tau}+\tau)-3}),
\end{split}
\end{align}
where we have denoted $\partial_t|_{t=0}$ by placing a dot over the corresponding quantities. We have also appealed to the shift-lapse decomposition associated to the asymptotic coordinates $(t,x^i)$ for the perturbed metric $g$: we have denoted $N$ the lapse function and $X$ the shift vector, thus {$g_{tt}=-N^2+|X|^2_{g}$} and $g_{ti}=X_i$. Therefore, we finally see that
\begin{align*}
-\mathcal{Q}(\hat{n},\hat{\nu})&=-\overbrace{\frac{1}{\hat{N}}}^{1+O_4(r^{-\hat{\tau}})}\mathcal{Q}_{tj}\hat{\nu}^{j} +\overbrace{\frac{1}{\hat{N}}}^{1+O_4(r^{-\hat{\tau}})} \mathcal{Q}_{ij}\overbrace{\hat{X}^i}^{O_4(r^{-\hat{\tau}})}\hat{\nu}^j
\end{align*}
can be rewritten as
\begin{align*}
\begin{split}
-\mathcal{Q}(\hat{n},\hat{\nu})|_{t=0}&=\left( \frac{3}{2}\beta  + 2\alpha\right) \left( \partial_{j}\partial_{i}\partial_{i}g_{aa} - \partial_{j}\partial_{u}\partial_{i}g_{u i}\right)\hat{\nu}^{j} + \frac{\beta}{2}\left( \partial_{i}\ddot{g}_{ji} - \partial_{j}\ddot{g}_{ii} \right)\hat{\nu}^{j}  + \frac{\beta}{2} \left( \partial_{i}\partial_{j}\dot{X}_{i} - \partial_{i}\partial_{i}\dot{X}_{j} \right)\hat{\nu}^{j} \\
&+ ( \beta  + 2\alpha)\partial_{j}\partial_{i}\partial_{i}\left[N^2-  |X|^2\right] \hat{\nu}^{j}- ( \beta  + 2\alpha)\partial_{j}\ddot{g}_{ii}\hat{\nu}^{j}+ 2( \beta  + 2\alpha)\partial_{j}\partial_{i}\dot{X}_{i}\hat{\nu}^{j} \\&+ O_1(r^{-\hat{\tau} -3}) +O_1(r^{-(\hat{\tau}+\tau)-3}).
\end{split}
\end{align*}
Here the expression can be simplified somewhat. Indeed since $X= O_4(r^{-\tau})$, a priori as soon as the shift decreases, $|X|^2$ is of higher order. Since it will be the case in all the following (even in section \ref{sectionconfflat} where the space metric and the lapse are allowed to grow, see remark \ref{remarkonthegrowth} below), we will then take the {following as the} working definition of the energy:

\begin{defn}\label{Energydefn}
Let $(V,\bar{g})$ be an AM solution to $A_{\bar{g}}=0$ obeying the decay conditions of the type imposed in (\ref{Falloff-AM.2}). Then, we define its energy as
\begin{align}\label{ADM-Energy.1}
\begin{split}
\mathcal{E}_{\alpha,\beta}(\bar{g})&=\lim_{r\rightarrow\infty}\Big\{\left( \frac{3}{2}\beta  + 2\alpha\right)\int_{S^{n-1}_r} \left( \partial_{j}\partial_{i}\partial_{i}g_{aa} - \partial_{j}\partial_{u}\partial_{i}g_{u i}\right)\hat{\nu}^{j}d\omega_{r}   \\
&+ \frac{\beta}{2}\int_{S^{n-1}_r}\left( \partial_{i}\ddot{g}_{ji} - \partial_{j}\ddot{g}_{ii} \right)\hat{\nu}^{j}d\omega_r + \frac{\beta}{2} \int_{S^{n-1}_r}\left( \partial_{i}\partial_{j}\dot{X}_{i} - \partial_{i}\partial_{i}\dot{X}_{j} \right)\hat{\nu}^{j}d\omega_r \\
&+ ( \beta  + 2\alpha)\left(\int_{S^{n-1}_r}\partial_{j}\partial_{i}\partial_{i} N^2 \hat{\nu}^{j}d\omega_r - \int_{S^{n-1}_r}\partial_{j}\ddot{g}_{ii}\hat{\nu}^{j}d\omega_r + 2\int_{S^{n-1}_{r}}\partial_{j}\partial_{i}\dot{X}_{i}\hat{\nu}^{j}d\omega_r\right)  \Big\}
\end{split}
\end{align}
whenever the limit exists.
\end{defn}

\begin{remark}
It is important to stress that (\ref{ADM-Energy.1}) is expressed in terms of \textbf{rectangular end coordinates}. 
\end{remark}

\begin{remark}
\label{remarkonthegrowth}
Of course, a classical way to ensure that the limit $\mathcal{E}_{\alpha,\beta}(\bar{g})$ exists is to consider  both $g$ and $\hat{g}$ AM, that is $0< \tau < \hat{\tau}$. But this is not the only possibility. For instance, with $n=3$ and $\hat{\tau}=1$ (keeping in mind that we can then take $\hat{g}$ the Schwarzschild metric) and $\tau <0$ (to allow for a possible growth of the metric $g$) the remaining term becomes: $ O_1(r^{-\hat{\tau} -3}) +O_1(r^{-(\hat{\tau}+\tau)-3})=O_1(r^{-{\tau}-4})$.  After integration on a $2$-sphere, one has the following convergence condition for this remainder: $\tau >2$. In particular, in dimension $3$  one can take $g$ with \emph{linear growth}.
\end{remark}
\begin{remark}\label{Energy-rigidty}
Notice that from our previous analysis (\ref{ADM-Energy.1}) seems to be a reasonable notion of energy attached to solutions of (\ref{action}). This is the case since, around appropriate solutions $\hat{\bar{g}}$ of $A_{\bar{g}}=0$ which possess time-translational symmetries, if we impose appropriate asymptotics for such solutions, then (\ref{ADM.5}) becomes a canonical notion of conserved energy density. Thus, since (\ref{ADM-Energy.1}) stands for what, a priori, is the leading order contribution of (\ref{ADM.5}), it becomes a natural candidate as a notion of energy.
\end{remark}

\section{A look at static spherically symmetric $A$-flat spaces}
\label{sectionconfflat}

The natural next step in this project would be to produce examples of solutions to the fourth order equations where we can actually \textit{test} the fourth order energy presented above, and which are simple enough to provide a good intuition and interpretation for the results. This is the case in GR, where the most basic of those examples is provided by the Schwarzschild solution, which, despite its simplicity, yields  a good interpretation for the ADM energy and serves as a basis for more complicated constructions. Nevertheless, one can quickly see that things are not so simple in our case. In order to explain why, let us start by analysing static spherically symmetric solutions to our field equations $A_{\bar{g}}=0$. 

We will say that a space-time $(V,\bar{g})$ is static and spherically symmetric if $V=I \times \s^{n-1}\times \mathbb{R}$, where $I\subset \mathbb{R}_{+}$ stands for some interval (possibly $I=\mathbb{R}_{+}$); we take $t$ to be the time coordinate along the $\mathbb{R}$ factor and, we assume that the orthogonal group acts by isometries $(r,p,t)\mapsto (r,\mathcal{O}p,t)$. All this constrains our metric to have the form $\bar{g}=-A^2(r)dt^2+B^2(r)dr^2+f^2(r)d\Omega^2$, with $d\Omega^2$ the standard metric on the sphere. Finally, assuming that $f(r)$ is monotonically increasing we can make a admissible change of coordinates so that (after relabelling the radial coordinate) 
\begin{align*}
\bar{g}=-A(r)^2dt^2 + B^2(r)dr^2 + r^2d\Omega^2
\end{align*}
Since this terminology is not completely uniform across standard literature, we will refer the reader to Chapter IV in \cite{CB-book} for more details.  In this section we will consider the $n=3$ case. However, given the complexity of the fourth order problem, even in this restrained configuration, we will first take a simplifying ansatz and assume $A(r)B(r)=1$: the so-called Schwarzschild case. In order to take into account the already known solutions (the Schwarzschild-de Sitter solutions), we will employ a variation of the constant method. To make the tensorial calculations more palatable, we will make a computer assisted proof, using Maple to find simplified equations. We will then broaden our considerations to all the static spherically symmetric solutions in the conformally invariant case: $3\alpha + \beta=0$.

The authors wish to highlight that the considerations and the final classification of spherically symmetric solutions in the conformally invariant case can already be found in \cite{Fiedler-Schimming} (also seen later in  \cite{SSsols2}-\cite{SSsols4}) which approached the problem from  the Bach-flat angle. We write a slightly different proof specific to our approach and developed independently of the problem  for completeness. 
 
\subsection{Schwarzschild solutions, proof of theorem \ref{lemmahighlighted}}
As announced above, we here take the additionnal ansatz $A(r)B(r)=1$ (the so-called Schwarzschild case).  Further, since in the considered dimension $n=3$ we already know a family of solutions: the Schwarzschild-de Sitter  (or SAdS) metrics. We will thus employ  a radial variation of the constant on the mass in the Schwarzschild metric. Concretely we will set 
\begin{equation} \label{gM} g = \begin{pmatrix} -1+\frac{M(r)}{r} & 0&0&0 \\ 0 & \frac{1}{1-\frac{M(r)}{r}} &0&0 \\ 0&0 &r^2 & 0 \\ 0&0&0 & r^2 \sin(\theta)^2 \end{pmatrix},\end{equation}
and see what conditions on $M$ make $g$ solve $A=0$. All Schwarzschild solutions  can be cast in this form, seeking solutions in this form thus entails no loss of generality.   A Maple procedure (see figure \ref{FequationM}) yields that:
\begin{equation}
\label{equationM}
\frac{r^6}{(r-M(r))^2}A_{11} + r^4 A_{22} = (2 \alpha + \beta) r^3 \frac{d^4}{dr^4} M(r) - 4 \left( 3 \alpha + \beta \right)  \left( r \frac{d^2}{dr^2} M(r)- 2 \frac{d}{dr} M(r) \right)=0.
\end{equation}

\begin{figure}
\includegraphics[scale=0.5]{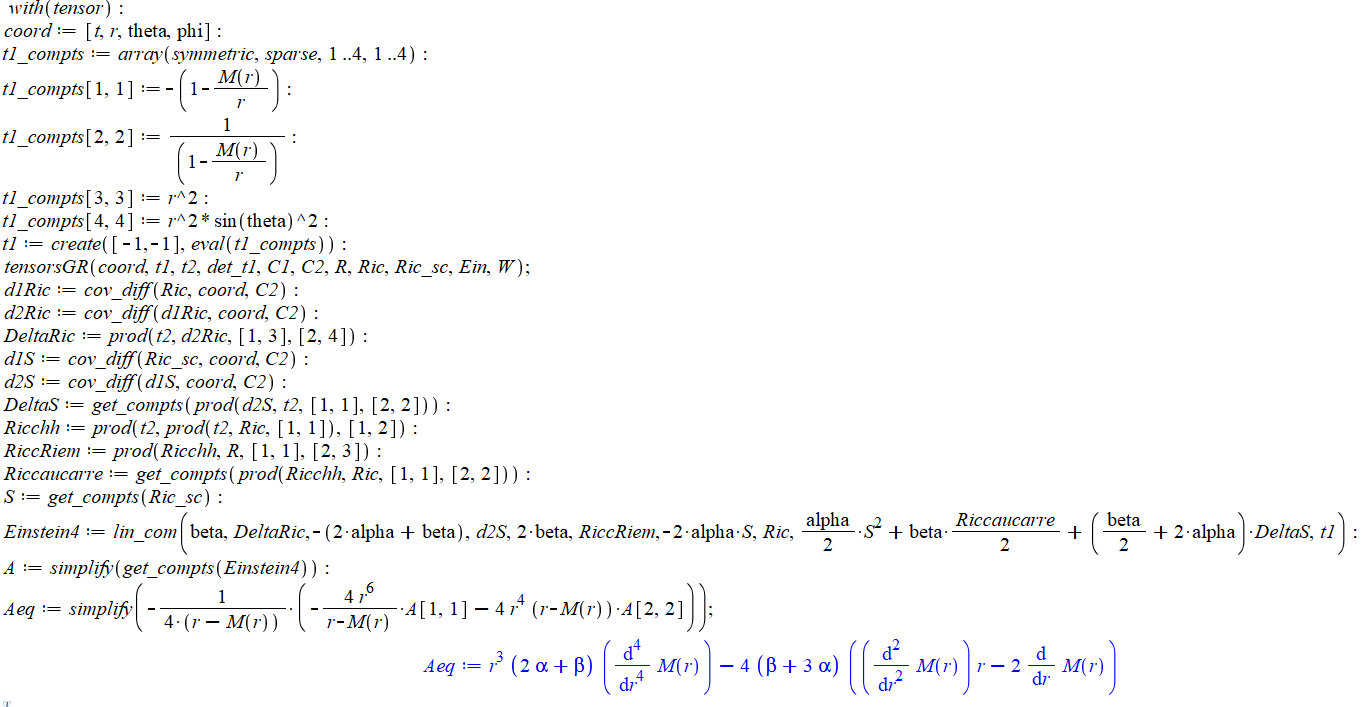}
\caption{An equation on $M$}
\label{FequationM}
\end{figure}

This equation once more highlights  the two special cases: $2 \alpha + \beta= 0$ and $3 \alpha + \beta = 0$. 
 
From this equation, one can deduce the $A$-flat metrics in the Schwarzschild form \eqref{gM}:
\begin{prop}\label{solscharzschgener}
Assume that $(V,g)$ is a static spherically symmetric $A$-flat space in Schwarzschild form. Assume further that $3\alpha + \beta \neq 0$. Then: 
\begin{itemize}
\item
If $\beta=0$, then $g$ is either a Schwarzchild-de Sitter  (or SAdS) metric or a Reissner-Nordstr\"om metric, meaning that:
$$g=-\left(1-\frac{m}{r} -\frac{\Lambda}{3}r^2\right)dt^2+\frac{1}{1-\frac{m}{r} -\frac{\Lambda}{3}r^2}dr^2 +r^2 d\Omega^2,$$
or
$$g=-\left(1-\frac{r_0}{r} -\frac{r_Q}{r^2}\right)dt^2+ \frac{1}{1-\frac{r_0}{r} -\frac{r_Q}{r^2}} dr^2 +r^2 d\Omega^2.$$
\item
If $\beta \neq0$ then $g$ is a Schwarzchild-de Sitter  (or SAdS) metric.
\end{itemize}
\end{prop}
\begin{proof}
Since  \eqref{equationM} is of order $2$ instead of $4$ when $2\alpha+\beta=0$, we deal with this case separately. Assuming $2\alpha+\beta=0$, \eqref{equationM} becomes $r \frac{d^2}{dr^2} M(r)- 2 \frac{d}{dr} M(r)=0$, meaning that only the $M(r)= m+\frac{\Lambda}{3}r^3$ yield potential $A$-flat metrics. These correspond to the Schwarzchild-de Sitter  (or SAdS) metrics, which are Einstein, and thus $A$-flat.

To avoid breaking the flow of the article, and because what remains of the proof is both simple in idea and complicated in execution (with many cases to consider), we will only sketch the proof here, and give the details in the appendix:
\begin{itemize}
\item
Solving \eqref{equationM} in the general case yields that $M= -m -\frac{\Lambda}{3}r^3 +C_1 r^{f(\alpha,\beta)} + C_2 r^{g(\alpha,\beta)}$. 
\item 
With such an $M$, computing the diagonal terms of $A$ shows they can be written as:
$$\begin{aligned}
\frac{\left(2 \alpha + \beta \right)^2 r }{24 \left( \beta + 4 \alpha \right) }A_{22}&= C_1\left( h^+_1 r^{t_1^+(\alpha, \beta) }+h^+_2r^{t_2^+(\alpha, \beta) }+h^+_3 r^{t_3^+(\alpha, \beta) } + h^+_4 r^{t_4^+(\alpha, \beta) } \right) \\&+  C_2\left( h^-_1 r^{t_1^-(\alpha, \beta) }+h^-_2 r^{t_2^-(\alpha, \beta) }+h^-_3 r^{t_3^-(\alpha, \beta) }+ h^-_4 r^{t_4^-(\alpha, \beta) } \right),
\end{aligned}
$$
where the $h^\pm_i$ are constants depending on $\alpha$, $\beta$, $m$, $\Lambda$, $C_1$ and $C_2$. The $4\alpha + \beta=0$ case will be dealt with separately.
\item
Outside of the \emph{finite} number of $[\alpha, \beta]$ configurations for which $t_i^\pm= t_j^\pm$ (and among them the $4\alpha + \beta=0$ case), one finds that $A_{11}=0$ implies that all the $h^\pm_i=0$, or that $C_1=C_2=0$. The former can only occur  when $\beta=0$ or $3\alpha + \beta=0$. The first case yields the  Reissner-Nordstr\"om  metric, while the second is outside the scope of this proposition. In the latter, we fall back on the Schwarzschild-de Sitter  (or SAdS) metric.
\item We treat the finite number of configurations left explicitely and independantly.
\end{itemize}
We once more refer the reader to the appendix for the detailed proof.
\qed
\end{proof}

\begin{remark}
That the Reissner-Nordstr\"om metric was a critical point of the $R^2$ energy was already featured in \cite{blackholesR2gravity}.
In order to explain the a priori singular Reissner-Nordstr\"om solution, one might conjecture that it is part of a family of $A$-flat metrics $g(\alpha, \beta)$  such that only the $g(\alpha, 0)$ are in the Schwarzschild form. 
\end{remark}

\begin{remark}
We do not detail the domains where the Schwarzschild-de Sitter (or SAdS) and the Reissner-Nordstr\"om are properly defined due to the abundance of litterature on these metrics. We will do it  when dealing with solutions to the conformally invariant configuration whose solutions are less commonly encountered.
\end{remark}

In the conformally invariant case, we know that the invariance group will generate more solutions. For instance, we know that any conformally Einstein metric is $A$-flat. We thus expect to find more solutions to the equations:

\begin{prop}\label{Classificationlemma}
The $A$-flat solutions in Schwarzschild form associated to the parameters choice $3\alpha+\beta=0$ can be classified into the following families:
\begin{enumerate}
\item $(V,\hat{g}(m,\Lambda,\mu))$, where
\begin{align}\label{Classification.1}
\begin{split}
\hat{g}(m,\Lambda,\mu) &=-f(r) dt^2 + \frac{1}{f(r)}dr^2 + r^2 g_{\s^2},\\
f(r)&\doteq 1  - 3m \mu - \frac{m}{r} - \mu (3m \mu -2) r - \frac{\Lambda}{3}r^2 ,
\end{split}
\end{align}
and $m$, $\Lambda$ and $\mu$ are integration constants. However in order to have admissible\footnote{by admissible we mean all metric which actually remain static spherically symmetric. In particular having $f\le 0$ means that the roles of $r$ and $t$ are exchanged and the space loses its static attribute.} metrics, these constants must be further constrained. We will detail these constraints when  $m>0$:
\begin{enumerate}
\item The choice $\Lambda\geq 0$ and $3m\mu\geq 2$ is not admissible;
\item For  $\Lambda<0$ and $3m\mu\geq 2$, there is always a solution of the form $V=\mathbb{R}\times (r_{*},\infty)\times \s^2$ for some $r_{*}>0$ depending on $(m,\Lambda,\mu)$.  Under these conditions, solutions of the form $V=\mathbb{R}\times (r_{-},r_{+})\times \s^2$ may be available, where $0<r_{-}<r_{+}<\infty$ depend on $(m,\mu,\Lambda)$;
\item For  $\Lambda>0$ and $0\leq 3m\mu<2$ there is a parameter range where $V=\mathbb{R}\times (r_{-},r_{+})\times \s^2$, and $0<r_{-}<r_{+}<\infty$ depend on $(m,\mu,\Lambda)$
\item For  $\Lambda\leq 0$ and $0\leq 3m\mu< 2$, with $\Lambda$ and $\mu$ not vanishing simultaneously, the only solutions are of the form $V=\mathbb{R}\times (r_{*},\infty)\times \s^2$ for some $r_{*}>0$ depending on $(m,\Lambda,\mu)$; 
\item $\mu=\Lambda=0$ yields the Schwarzschild space;
\item For  $\Lambda> 0$ and $0\leq 3m\mu< 2$, the only possible solutions are as in $(c)$;
\item For  $\mu < 0$ and $\Lambda\geq 0$, the only available solutions are of the form of $(c)$;
\item For  $\mu < 0$ and $\Lambda< 0$, the situation is the same as in $(b)$.
\end{enumerate}
\item When $m=0$, then we fall into the following possible families $(V,\hat{g}_0=\hat{g}(0,\Lambda,\mu))$, with
	\begin{enumerate}
	\item There is a first family of the form
	\begin{align}\label{Classification.3}
	\begin{split}	
	g_0&= -N(r)dt^2 + N^{-1}(r)dr^2+r^2g_{\s^2},\\
	N(r)&=1+2\mu r - \frac{\Lambda}{3}r^2.
	\end{split}
	\end{align}
		\begin{enumerate}
		\item If $\mu\geq 0$ and $\Lambda>0$, then $r\in [0,r_{+})$, with $r_{+}>0$ depending on $\mu,\Lambda$;
		\item If $\mu\geq 0$ and $\Lambda\leq 0$, then $r\in [0,\infty)$;
		\item If $\mu<0$ and $\Lambda\geq 0$, then the solutions are are as in $(2.a.i)$; 
		\item If $\mu<0$ and $\Lambda<0$, there is always a solution of the form $r\in (r_{*},\infty)$ and, depending on $\mu,\Lambda$ there can solution with $r$ in some bounded interval;
		\end{enumerate}
\item There is a second family of the form
\begin{align}\label{Classification.2}
\begin{split}
\hat{g}_0&= - h(r)dt^2 + \frac{1}{h(r)}dr^2+  r^2g_{S^2},\\
h(r)&=-\frac{\Lambda}{3}r^2-\mu r -1,
\end{split}
\end{align}
with the following restrictions for the integration constants $\Lambda$ and $\mu$:
\begin{enumerate}
\item If $\Lambda=0$, then we must have $\mu<0$ and $V=\mathbb{R}\times (\frac{1}{|\mu|},\infty)\times \s^2$.
\item If $\Lambda\neq 0$, then the parameters are restricted to $\mu^2-\frac{4\Lambda}{3}>0$. Furthermore, if $\Lambda>0$, then the choices $\mu\geq 0$ are not admissible, while for $\mu<0$ we find $V=\mathbb{R}\times (r_{-},r_{+})\times \s^2$, for some $0<r_{-}<r_{+}<\infty$ which depend on $\Lambda$ and $\mu$. Finally, if $\Lambda<0$, then there is some $r_{*}>0$ such that $V=\mathbb{R}\times (r_{*},\infty)\times \s^2$.
\end{enumerate}
\end{enumerate}
\end{enumerate}
Furthermore, the metrics (\ref{Classification.1}),((\ref{Classification.3})) and \ref{Classification.2}) are Einstein (wherever they may be defined) iff $\mu=0$. On the other hand, these metrics are almost conformally Einstein globally. More precisely there exists at most a radius $r_s$ such that both for $r<r_s$ and $r>r_s$, these metrics are conformally Einstein. 
\end{prop}

\begin{remark}
We set ourselves in the $m\ge  0$ case to avoid doubling the conditions, and to develop the natural extension of the GR case. 
\end{remark}
\begin{remark}
We limited ourselves to this choice of admissible metrics (those that remain static) for their relevance regarding our study of the fourth order conserved quantity in subsection \ref{subsection32}. We refer the reader to  \cite{SSsols4} for more details on what may happen when the roles of $t$ and $r$ are switched. 
\end{remark}
\begin{proof}
\label{solvariatconst}
In this case, \eqref{equationM} yields a necessary condition $$\frac{d^4}{dr^4} M(r)=0,$$ which is satisfied by any 
$$M(r) = m + \frac{\Lambda}{3} r^3 + C_1r^2 +C_2r.$$ 
With such a $M$ one can check (see figure \ref{condC1C2}) that $A=0$ if and only if:
$$3C_1m - C_2^2+2C_2=0.$$
\begin{figure}[!h]
\includegraphics[scale=0.5]{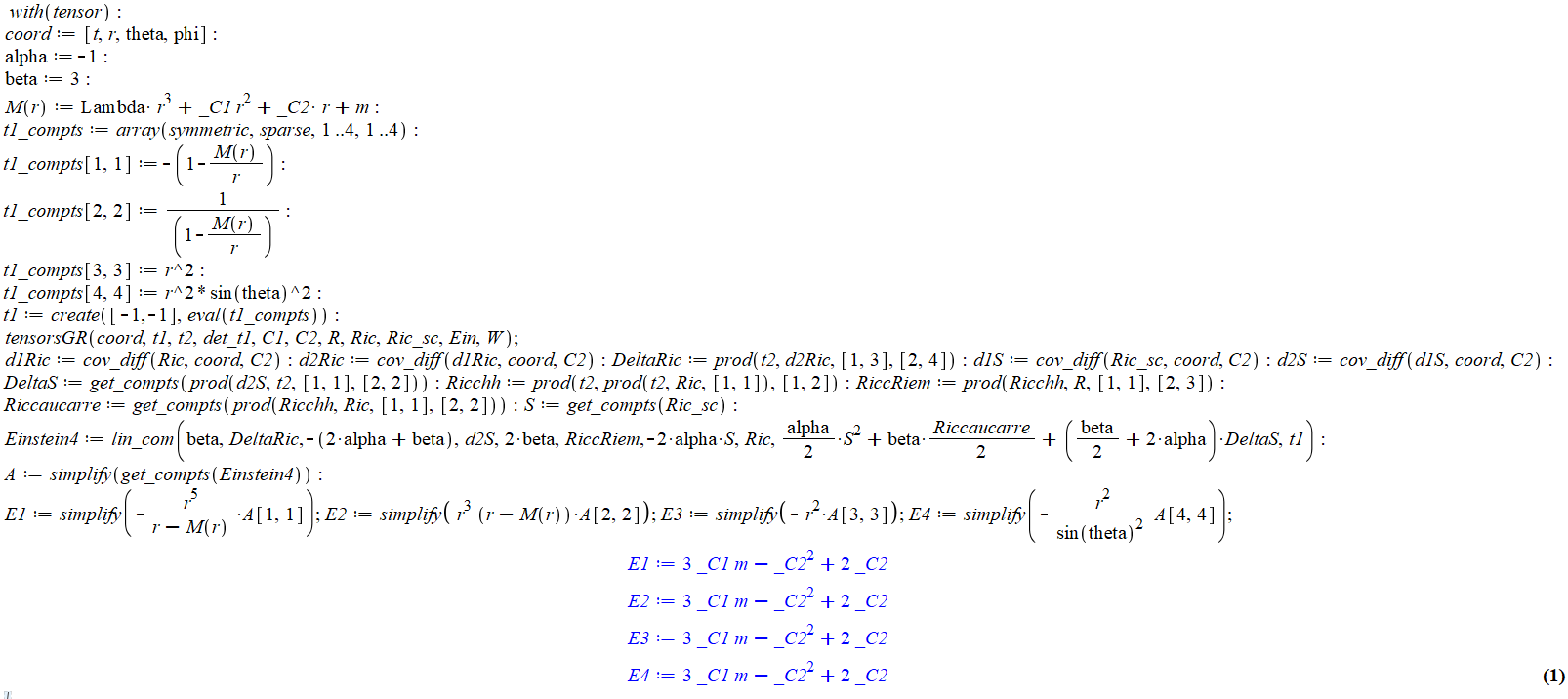}
\caption{The necessary and sufficient condition for $A=0$.}
\label{condC1C2}
\end{figure}
Consequently, we divide the results into the following two large families of cases:
\begin{enumerate}
\item[1)]  $m\neq 0$
\end{enumerate}

In this case we can re-parametrize the solutions by choosing $C_2= 3m \mu$ and $C_1=\mu (3m \mu -2 )$, and any metric
\begin{align}\label{SSmetric.1}
\begin{split}
\hat{g}(m,\Lambda,\mu) =&-f(r) dt^2 + \frac{1}{f(r)}dr^2 + r^2 d\theta^2 + r^2 \sin^2 \theta d\phi^2
\end{split}
\end{align}
with 
\begin{align*}
f(r)\doteq 1  - 3m \mu - \frac{m}{r} - \mu (3m \mu -2) r - \frac{\Lambda}{3}r^2 
\end{align*}
is a critical point of $E_{-\lambda,3\lambda}$. Clearly, the above expression is only admissible if $f(r)>0$. The combination of parameters that fulfil this condition can be analysed straightforwardly and results in the classification provided in the statement of the Lemma. This can be lengthy to do explicitly, but since this analysis does not present any technical or conceptual difficulty, it will be left for the reader. 

\bigskip
\begin{enumerate}
\item[2)] $m=0$.
\end{enumerate}
In this case $C_1$ is a free constant and $C_2=0$ or $2$. The first is already taken into account with $\hat{g}(0, \Lambda, \mu)$, with $C_1 = -2 \mu$. That is, in this case, the solutions are given by
\begin{align*}
\begin{split}
\hat{g}(0,\Lambda,\mu) =&-\left(1 + 2\mu r - \frac{\Lambda}{3}r^2 \right) dt^2 + \frac{1}{1 + 2\mu r - \frac{\Lambda}{3}r^2}dr^2 + r^2 d\theta^2 + r^2 \sin^2 \theta d\phi^2.
\end{split}
\end{align*}
Again, all the admissible choices of $\mu$ and $\Lambda$ as well as the corresponding admissible space-time structures $V$ for these cases follow straightforwardly from the analysis of the condition $N(r)=1 + 2\mu r - \frac{\Lambda}{3}r^2>0$

Let as now consider the second case, that is $C_2=2$, and introduce 
\begin{align}\label{SSsolution.2}
\hat{g}_0(\Lambda,\mu) =&-\left( -1 - \frac{\Lambda}{3}r^2 - \mu r \right) dt^2 + \frac{1}{ -1 - \frac{\Lambda}{3}r^2 - \mu r} dr^2 + r^2 d\theta^2 + r^2 \sin^2 \theta d\phi^2.
\end{align}
The metrics $\hat{g}_0(\Lambda,\mu)$  are Lorentz metrics of signature $(-,+,+,+)$ when $h(r)\doteq -1 - \frac{\Lambda}{3}r^2 - \mu r \ge0$, which imposes the following compatibility conditions for the parameters
\begin{enumerate}
\item[2.1)] If $\Lambda\neq 0$, then 
\begin{equation} \label{relat} 
\mu^2 - \frac{4\Lambda}{3} > 0. 
\end{equation}
Indeed, the above quantity is obviously positive if $\Lambda<0$, and is the discriminant of $h$, and must thus be positive for $h$ to have positive values whenever $\Lambda >0$.
\item[2.2)] If $\Lambda=0$, then $\mu< 0$ since the non-negative cases would produce metrics which are not static. In this case, we arrive at the compatibility condition
\begin{align}
r>\frac{1}{|\mu|}.
\end{align}
\end{enumerate}
Let us now consider the different possibilities when $\Lambda\neq 0$.
\begin{itemize}
\item $\Lambda>0$.
\end{itemize}
Then, the two roots for $h(r)=0$ are given by
\begin{align}\label{roots}
r_{\pm}=\frac{1}{2}\left( -\frac{3\mu}{\Lambda}\pm\sqrt{\frac{9\mu^2}{\Lambda^2} - \frac{12}{\Lambda}} \right).
\end{align}
Because of (\ref{relat}), the above is well-defined and this implies that if $\mu>0$, then  $r_{\pm}<0$ and therefore $h(r)<0$ for all $r>0$. We must conclude that this case is not admissible. We still need to analyse what happens when $\mu<0$. In this last case, both $r_{\pm}>0$ and $f(r)<0$ for $r$ large enough. Thus, we find some interval $(r_{-},r_{+})$ where (\ref{SSsolution.2}) well-defined.

\begin{itemize}
\item $\Lambda<0$.
\end{itemize}
Analysing (\ref{roots}), we can see that for any admissible value of $\mu$ we always get two distinct roots $r_{-}<0$ and $r_{+}>0$. This implies that (\ref{SSsolution.2}) is well-defined for all $r>r_{+}$.

\bigskip


 Let us first show $\hat{g}(m,\Lambda,\mu)$ are not Einstein metrics if $\mu$ is not $0$. Indeed it can be checked that their Einstein tensor satisfies:
$$\begin{aligned}
\left(G_\kappa^\lambda \right) = \Lambda \mathrm{Id} + \begin{pmatrix} 2\mu \frac{3m\mu- 2}{r} + \frac{3m\mu}{r^2} &0&0&0 \\ 0& 2\mu \frac{3m\mu- 2}{r} + \frac{3m\mu}{r^2}&0&0 \\ 0&0& 2\mu \frac{3m\mu- 2}{r} &0 \\ 0&0&0 & 2\mu \frac{3m\mu- 2}{r} \end{pmatrix}.
\end{aligned}$$
Similarly one can compute the Einstein tensor of the $\hat{g}(0, \Lambda, \mu)$ and see:
$$\begin{aligned}
\left(G_\kappa^\lambda \right) = \Lambda \mathrm{Id} + \begin{pmatrix} \frac{2 \mu}{r} + \frac{2}{r^2} &0&0&0 \\ 0& \frac{2 \mu}{r} + \frac{2}{r^2}&0&0 \\ 0&0&\frac{\mu}{r} &0 \\ 0&0&0 &\frac{\mu}{r}  \end{pmatrix}.
\end{aligned}$$

  We will now prove that $\left(\hat{g}(m,\Lambda,\mu), \hat{g}_0 ( \Lambda, \mu ) \right)$ are conformally Einstein. Let us begin with the analysis of $\hat{g}(m,\Lambda,\mu)$ and, first first of all, notice that if $\mu=0$ we fall in one the (B)-type cases labelled above and, in these cases, we can see directly that $\hat{g}(m,\Lambda,0)$ is a S(A)dS metric. Thus, let us now assume $\mu\neq 0$ and consider  a metric $\tilde g = \left(\frac{1}{1+ \mu r} \right)^2 \hat{g}$, which is well-defined in the intersection of the domain where $\hat{g}$ is well-defined with $r\neq -\frac{1}{\mu}$.  Then:
 $$\begin{aligned}
\tilde g &= 
-\frac{ 1- \frac{m}{r} - \frac{\Lambda}{3}r^2 - 3m \mu -\mu (3m \mu -2) r }{ \left( 1+ \mu r\right)^2}dt^2 + \frac{dr^2}{\left( 1+ \mu r\right)^2\left( 1- \frac{m}{r} - \frac{\Lambda}{3}r^2 - 3m \mu -\mu (3m \mu -2) r  \right)}  \\&+ \left( \frac{r}{1+ \mu r } \right)^2 d\theta^2 + \left( \frac{r}{1+\mu r } \right)^2 \sin \theta^2 d\phi^2.
\end{aligned}
$$
Setting $R = \frac{r}{1 + \mu r }$, we find that $dr = \frac{dR}{ \left(1- \mu R \right)^2} = dR \left( 1+ \mu r \right)^2$ and thus
 $$\begin{aligned}
\tilde g &= 
-\frac{ 1- \frac{m}{r} - \frac{\Lambda}{3}r^2 - 3m \mu -\mu (3m \mu -2) r }{ \left( 1+ \mu r\right)^2}dt^2 + \frac{1}{\frac{ 1- \frac{m}{r} - \frac{\Lambda}{3}r^2 - 3m \mu -\mu (3m \mu -2) r }{ \left( 1+ \mu r\right)^2}} dR^2  \\&+ R^2 d\theta^2 + R^2 \sin \theta^2 d\phi^2.
\end{aligned}$$
We can then compute $1+ \mu r = 1+ \frac{\mu R}{1-\mu R} = \frac{1}{1-\mu R}$ and thus
$$
\begin{aligned}
\frac{ 1- \frac{m}{r} - \frac{\Lambda}{3}r^2 - 3m \mu -\mu (3m \mu -2) r }{ \left( 1+ \mu r\right)^2}&= \left( 1- \mu R \right)^2 \left( 1 - \frac{m(1-\mu R )}{R} - \frac{ \Lambda}{3} \frac{R^2}{ \left( 1- \mu R \right)^2} -  3m \mu  \right. \\& \left.- \mu (3m \mu -2 ) \frac{R}{1-\mu R }\right) \\
&= 1- \frac{m}{R} - \frac{ \Lambda - 3 \mu^2 \left( \mu m -1 \right) }{3} R^2.
\end{aligned}$$
This implies that $\tilde g$ is a Schwarzschild-de Sitter metric of mass $m$ and cosmological constant $\tilde \Lambda = \Lambda - 3 \mu^2 \left( \mu m -1 \right)$.

Finally, let us analyse the conformal family of $\hat{g}_0(\Lambda,\mu)$ by considering the metrics $\tilde g_0 = \frac{1}{r^2} \hat{g}_0$.  Then:
$$ \begin{aligned} \tilde {g}_0(\Lambda,\mu) =&-\frac{ -1 - \frac{\Lambda}{3}r^2 - \mu r }{r^2} dt^2 + \frac{1}{r^2 \left( -1 - \frac{\Lambda}{3}r^2 - \mu r \right)} dr^2 + d\theta^2 +\sin^2 \theta d\phi^2. \end{aligned}$$
Setting $dx  = \frac{dr}{r \sqrt{ -1 - \frac{\Lambda}{3}r^2 - \mu r }}$, i.e.$x = \arctan \left(  -\frac{ \mu r +2 }{ 2 \sqrt{ - \frac{\Lambda}{3} r^2 - \mu r - 1 }} \right)$ we find that 
 $$\begin{aligned}
\tilde g_0 &= -\frac{1}{4} \left( \mu^2- \frac{4 \Lambda}{3} \right) \cos^2x dt^2 + dx^2 +d\theta^2 + \sin^2\theta d\phi^2.
\end{aligned}$$ 
Taking $T = \frac{1}{2} \sqrt{ \mu^2- \frac{4 \Lambda}{3} }t$, we simplify 
\begin{align}\label{cylindricalmetric}
\tilde g_0 &= - \cos^2 x dT^2 + dx^2 + d\theta^2 + \sin^2 \theta d\phi^2
\end{align}
which is periodic and well-defined in the above coordinates. Since if we denote $l(r)= -\frac{ \mu r +2 }{ 2 \sqrt{ - \frac{\Lambda}{3} r^2 - \mu r - 1 }} $, we can compute $l'(r)=\frac{\left( \mu^2- \frac{4\Lambda}{3} \right) r}{4 ( - \frac{\Lambda}{3} r^2 - \mu r - 1 )^{\frac{3}{2}}}$ and deduce that the change of variable is well defined whenever $\tilde {g}_0$ is.
Appealing to a warped-product decomposition for the Ricci tensor, straightforwardly we see that $g_0$ is Einstein in these domains.
\qed
\end{proof}

\begin{remark}
As has been mentioned in introduction of this section, this classification (in the conformally invariant case) had already been found (see \cite{Fiedler-Schimming}, or \cite{SSsols2} and \cite{SSsols3} which refound it independantly). In fact, the metrics $\hat{g}$ are known as the FSMK solutions. 
\end{remark}

\begin{remark}
It is interesting to compare these results to \cite{Fiedler-Schimming} (see also \cite{SSsols4}). In it they employ a $2\times 2$ decomposition of the metric and a parametrization by the scalar curvature $R$ of the first metric. When $R\neq \pm2$, they recover the Mannheim-Kazanas solutions, with two outliers when $R=2$ and $R=-2$. The $R=2$ space corresponds to the cylindrical metric \eqref{cylindricalmetric} while the $R=-2$ would be the other cylindrical metric (obtained when  $\mu^2-\frac{4\Lambda}{3}\le 0$) which we did not consider, since it is no longer static.
\end{remark}

Let us now highlight that the above classification presents exterior solutions defined up to infinity in several cases $(1.b),(1.d),(2.a.ii),(2.a.iv),(2.b.i)$ and $(2.b.ii)$, besides from the Schwarzschild solution. All these extra solutions are of the form of the $FSMK$-solutions of conformal gravity. In particular, let us focus on cases $(1.d),(2.a.ii)$ and $(2.b.i)$ which admit such exterior solutions with $\Lambda=0$ for $0<\mu<\frac{2}{3m}$, $\mu\geq 0$ and $\mu<0$ respectively. All these cases can be summarised by considering a solution of the form $(\overset{\circ}{V}=\mathbb{R}\times(r_{*},\infty)\times \s^2,\overset{\circ}{\bar{g}})$ with
\begin{equation}\label{FSMKsol.1}
\begin{aligned}
\overset{\circ}{\bar{g}}&=-N^2(r)dt^2+\frac{1}{N(r)^2}dr^2+r^2g_{\s^2},\\
N^2(r)&=c_1 - \frac{m}{r} + c_2r,
\end{aligned}
\end{equation}
where $c_1,c_2$ are positive constants  and $m\geq 0$. These types of solutions have been extensively analysed within the context of conformal gravity when analysing the rotation curves of galaxies. In particular, in this context, the presence of $c_2$ is used to explain the flattening of the rotation curves, which is deviation from typically expected results and is accounted in standard astrophysics via an appeal to dark-matter. Clearly, if this is to be an interpretation for these solutions, their extrapolation up to infinity is an artefact of an abstraction procedure, which is, nevertheless, useful for our purposes as we shall see shortly.

Notice that the above solution would represent a perturbation of Minkowski that actually grows at infinity. This falls in line with previous comments corning the fact that natural asymptotics for these higher-order problems may fail to follow usual intuition from GR. In particular, for the purposes of analysis of the fourth order energy $\mathcal{E}_{\alpha,-3\alpha}(\bar{g})$, the above solutions prove to be useful, as can be seen from theorem \ref{lemmahighlighted} which we recall here: 

\begin{theo}\label{FSMKlemma}
The fourth order energy of the solution $(\overset{\circ}{V},\overset{\circ}{\bar{g}})$ is well-defined and given by
\begin{align*}
\mathcal{E}_{\alpha,-3\alpha}(\overset{\circ}{\bar{g}})=8\pi\alpha c_2.
\end{align*}
\end{theo}
\begin{proof}
The energy associated to a static solution is given by
\begin{align}\label{FSMK-energy.1}
\mathcal{E}_{\alpha,-3\alpha}(\bar{g})&=-\alpha\lim_{r\rightarrow\infty}\Big\{\frac{5}{2}\int_{\s^2_r}\left( \partial_{jii}g_{aa} - \partial_{jui}g_{ui} \right)\hat{\nu}^jd\omega_r + \int_{\s^2_r}\partial_j\partial_i\partial_iN^2\hat{\nu}^jd\omega_r \Big\}.
\end{align}
In the case that $\bar{g}$ is spherically symmetric, we find that
\begin{align*}
\hat{\nu}^j\partial_j\partial_{i}\partial_iN^2=\frac{\partial}{\partial r}\left(\Delta_{e}N^2 \right)=\frac{d}{dr}\left(\frac{1}{r^2}\frac{d}{dr}\left( r^2\frac{d N^2(r)}{dr} \right)\right),
\end{align*}
where $\Delta_e$ stands for the negative Euclidean Laplacian. Plugging $N^2(r)=c_1-\frac{m}{r}+c_2r$ in the above expression gives $\frac{d}{dr}\left(\Delta_{e}N^2 \right)=-2\frac{c_2}{r^2}$, which implies
\begin{align*}
\int_{\s^2_r}\partial_j\partial_i\partial_iN^2\hat{\nu}^jd\omega_r=-2\omega_2 c_2.
\end{align*}

On the other hand, computing the contributions of the first term in (\ref{FSMK-energy.1}) can be more difficult, since the above expression should be computed in asymptotic Cartesian coordinates. This can be done straightforwardly, nevertheless, we can save some computations proceeding as follows. 

Let $e$ denote the euclidean metric defined near infinity in $\overset{\circ}{M}=(r_{*},\infty)\times \s^2$, so that $e=dr^2+r^2g_{\s^2}$ in the same spherical coordinates as in (\ref{FSMKsol.1}). Let us denote by $\{x^i\}_{i=1}^3$ the Cartesian coordinates associated to these spherical coordinates and use $\{y^j\}_{j=1}^3$ to denote the spherical coordinate system $(r,\theta,\phi)$. Then, let us denote by $\delta_{ij}=e(\partial_{x^i},\partial_{x^j})$, $g_{ij}=g(\partial_{x^i},\partial_{x^j})$, $\tilde{e}_{ij}=e(\partial_{y^i},\partial_{y^j})$ and $\tilde{g}_{ij}=g(\partial_{y^i},\partial_{y^j})$ the matrices associated to $g$ and $e$ is both coordinate systems. Then, clearly, we have that $\partial_{x^k}g_{ij}=\partial_{x^k}(g_{ij}-e_{ij})$. Also, it holds that
\begin{align*}
g_{ij}-e_{ij}=\frac{\partial y^a}{\partial x^i}\frac{\partial y^b}{\partial x^j}\left( \tilde{g}_{ab} - \tilde{e}_{ab} \right)=\frac{\partial r}{\partial x^i}\frac{\partial r}{\partial x^j}\frac{1-N^2}{N^2}=\frac{x^ix^j}{r^2}\frac{1-c_1+\frac{m}{r}-c_2r}{c_1-\frac{m}{r}+c_2r},
\end{align*}
where we have used that $\tilde{g}_{ab} - \tilde{e}_{ab}=\mathrm{diag}\left(\frac{1-N^2}{N^2},0,0\right)$. The above expression clearly implies that $g_{ij}-e_{ij}=O_{\infty}(r^0)$, and therefore $\partial_{jii}g_{aa} - \partial_{jui}g_{ui}=O(r^{-3})$ implying that
\begin{align*}
\int_{\s^2_r}\left( \partial_{jii}g_{aa} - \partial_{jui}g_{ui} \right)\frac{x^j}{r}d\omega_r=O(r^{-1})\xrightarrow[r\rightarrow\infty]{} 0.
\end{align*}
Putting all of the above together implies the desired result.
\qed
\end{proof}

\begin{remark}
As was mentioned in remark \ref{remarkonthegrowth},  applying the previous reasoning requires a strictly below quadratic growth. However, since when considering $\hat{g}$, the $\Lambda$ is not merely quadratic but exactly $r^2$, one can extend the above proposition to the whole family $\hat{g}$ when it is well defined at infinity.
\end{remark}

\subsection{$A$ flat spherically symmetric spaces in the conformally invariant case} 
One can use the conformal invariance to extend the previous discussion to all $A$ flat spherically symmetric spaces. First, let us introduce the following terminology. We will say that a space-time $(V,\bar{g})$ is \textbf{almost conformally Einstein} if  there exists at most a radius $r_s$ such that both for $r<r_s$ and $r>r_s$, these metrics are conformally Einstein.  Similarly, we will say that another Lorentzian metric $\bar{g}^{*}$ is almost conformal to $\bar{g}$ if there is some conformal transformation between them which is defined almost everywhere. In this context, we have the following theorem.

\begin{thm}\label{Classificationthm}
Any exterior static spherically symmetric Bach-flat space-time $(V,\bar{g})$ is almost conformally Einstein. More specifically, any such static spherically symmetric Bach-flat space-time is almost conformal to a subset of a Schwarzschild-de Sitter (or SAdS)  space-time, or to (\ref{cylindricalmetric}).
\end{thm}

\begin{proof}
From our hypotheses, we can assume that $\bar{g}$ takes the form $$\bar{g}= - U(r)^2 dt^2 + V(r)^2dr^2 + r^2 d\theta^2 + r^2 \sin \theta^2 d \phi^2,$$ where $U,V$ are positive functions. Let us now consider 
 $$\bar{g}_1:= \frac{1}{F(r)^2} \bar{g} = - \left( \frac{U(r)}{F(r)} \right)^2 dt^2 + \left( \frac{V(r)}{F(r)} \right)^2 dr^2 + \left( \frac{r}{F(r)} \right)^2  d\theta^2 +  \left( \frac{r}{F(r)} \right)^2 \sin \theta^2 d\phi^2,$$
where ${F^2} \, : \, I\subset \R_+^* \rightarrow \R_-^*$ is such that
\begin{equation} \label{condfactconf} 
\frac{d}{dr} \left( \frac{F(r)}{r} \right)= -\frac{U(r)V(r)}{r^2},
\end{equation}
and satisfies an initial condition of the form $F(r_0)<0$.

Now thanks to \eqref{condfactconf}, if $\bar{g}$ does not degenerate, $\frac{r}{F(r)}$ is a local diffeomorphism. Let us then do a change of variable $R = -\frac{r}{F(r)}$. Then 
 $$g_1(t,R,\theta, \phi):=  -  P(R) dt^2 + Q(R) dR^2 + R^2  d\theta^2 +  R^2 \sin \theta^2 d\phi^2,$$

with 
$$ \begin{aligned}P(R)&=\left(\frac{U(r)}{F(r)} \right)^2\\
Q(R)&=\left(\frac{V}{F} \frac{1}{\frac{d}{dr} \left(- \frac{r}{F(r)} \right)} \right)^2 =  \left(\frac{V(r)}{F(r)} \frac{ \left( \frac{F(r)}{r} \right)^2 }{ \frac{d}{dr} \left( \frac{F(r)}{r} \right)} \right)^2 =  \left(\frac{V(r) F(r)}{r^2 \frac{U(r)V(r)}{r^2} } \right)^2= \left(\frac{F(r)}{U(r)} \right)^2= \frac{1}{P(R)}.
\end{aligned}
$$
Thus, if we force the decomposition $P(R) = 1+ \frac{M(R)}{R}$, the metric $g_1(t,R,\theta, \phi)$ is as studied in Proposition \ref{Classificationlemma}, and is thus conformally Einstein and either confofmal to S(A)dS or to (\ref{cylindricalmetric}).
\qed
\end{proof}

\bigskip

It is interesting to compare theorem \ref{Classificationthm} with the Willmore case. Indeed in the latter, R. Bryant proved (see \cite{bibdualitytheorem}) that all Willmore spheres (surfaces of genus $g=0$ in $\R^3$ which are critical points of $\int H^2$) are conformally minimal: there exists a conformal transformation which imposes $H=0$ on the whole sphere.  The comparison may be skin-deep (the proofs are markedly different) but reveal how the invariance group coupled with the fourth order equations leads to a constrained variety of solutions: under an additionnal assumption (static spherically symmetric when $A$-flat, topologically a sphere when Willmore) the solutions to the fourth order are conformally solutions of a second order equation (Einstein when $A$-flat, minimal when Willmore).

\section{Positive energy theorem for Einstein metrics}\label{PEthemSection}

Because of the above conclusion, we will now move into producing \textit{implicit} $A$-flat examples and relax the previous symmetry assumptions. Implicit constructions of $A$-flat metrics are highly non-trivial in general, contrary to the case of Einstein space-time metrics which can be obtained by evolving initial data satisfying the Einstein constraint equations (ECE). However, for space-time Einstein metrics, the Ricci tensor is proportional to the metric, and thus the $A$ tensor reduces to  the quadratic terms: $A =  2\beta \left[  {\mathrm{Ric}_{\bar{g}}}_{\cdot}\mathrm{Riem}_{\bar{g}}  - \frac{1}{4} \langle\mathrm{Ric}_{\bar{g}},\mathrm{Ric}_{\bar{g}}\rangle_{\bar{g}} {\bar{g}} \right] + 2\alpha \left[ R_{\bar{g}}\mathrm{Ric}_{\bar{g}} - \frac{1}{4}R^2_{\bar{g}} {\bar{g}} \right]$. Inserting $\mathrm{Ric}_{\bar{g}} = \lambda \bar{g}$, and thus $R = \lambda (n+1) $ in the last equality, and since ${\mathrm{Ric}_{\bar{g}}}_{\cdot}\mathrm{Riem}_{\bar{g}} = \lambda \bar{g}_{\cdot}\mathrm{Riem}_{\bar{g}} = \lambda \mathrm{Ric}_{\bar{g}} = \lambda^2 \bar{g}$ since the contraction is on the first and third index,   yields $A = 2 \beta \lambda^2 \left[ 1- \frac{n+1}{4} \right] \bar{g} + 2 \alpha \lambda^2(n+1) \left[  1 - \frac{n+1}{4} \right] \bar{g} $. For $n=3$, we can then deduce that Einstein space-time metrics are $A$-flat.

The Einstein space-time metrics then provide us with a large set of non-trivial examples where we can analyse the behaviour of $\mathcal{E}_{\alpha,\beta}$. Furthermore, for these kinds of second order solutions one may have reasonable expectations on what an appropriate notion of energy should be measuring based on physical interpretations on the different energy sources. Along these lines, one can also appeal to experience within more geometrical quadratic Langrangians, where minimal surfaces appear as second order solutions whose analysis has proven to be relevant for pure higher order results. Let us also highlight that Corollary \ref{ADTcoro} also motivates the analysis of 4-dimensional Einstein solutions and provides some intuitions on what we can expect to obtain. That is, in the cases included in Corollary \ref{ADTcoro}, the fourth order energy turned out to be proportional to the ADM contribution times the cosmological constant. We will see below that this result remains true in a broader scenario.


Having in mind 4-dimensional Einstein solutions, first of all, let us recall that an initial data set for the Einstein equations is a set of the form $\mathcal{I}=(M,g,K,\epsilon,J)$, where $(M,g)$ is a Riemannian manifold, $K$ is a symmetric second rank tensor field on $M$, $\epsilon$ is a function and $J$ a $1$-form on $M$, which is subject to the ECE
\begin{align}\label{ECE}
\begin{split}
R_{g}-|K|^2_g + (\mathrm{tr}_gK)^2=2\epsilon,\\
\mathrm{div}_gK-d\mathrm{tr}_gK=J.
\end{split}
\end{align}
In the above equations $\epsilon$ and $J$ stand for the energy and momentum densities induced on $M\cong M\times\{0\}$ by some energy-momentum tensor field $T$ on space-time. It is a remarkable fact that on many cases of interest, which include vacuum and $\Lambda$-vacuum (which corresponds to $\epsilon=\Lambda=cte$ and $J=0$), the above equations stand as both necessary and sufficient conditions for the initial value problem associated to the space time equations $G_{\bar{g}}=T$ to be well-posed. 
Such initial value problem can be formulated for initial data in uniformly local spaces and therefore it gives great freedom on the global properties of $M$ (see, for instance, \cite{CB-book}). 

Notice that in the case of Einstein spaces where $\epsilon=\Lambda=cte$ the asymptotics of \emph{isolated} systems cannot satisfy usual decaying conditions.\footnote{This would be $g_{ij} = \delta_{ij} + O_2(r^{-\tau})$ and $K_{ij} = O_1(r^{-\tau - 1})$, with $\tau > 0$.} Therefore, we are forced to consider new (weaker) asymptotic behaviour for the fields. With this in mind, we appeal to a few physical considerations. First, we know that the effect of the cosmological constant in Nature drives the expansion of our Universe when matter fields are sufficiently diluted, and thus, for instance, it actually breaks time-symmetry. In fact, idealised cosmological solutions have umbilical Cauchy surfaces, and thus we propose that a \emph{natural} effect of a cosmological constant could be to make isolated solutions asymptotically umbilical. Furthermore, the asymptotic mean curvature should be controlled by the strength of $\Lambda$. All these considerations are basically contained in the following definition, which fixes the asymptotics we shall consider by modifying classical asymptotics as little as possible.

\begin{defn}\label{LamdAE}
We will say that the initial data set $\mathcal{I}$ is admissible if the evolution problem associated to such initial data is well-posed. Also, we will say that $\mathcal{I}$ is $\Lambda$AE of order $\tau>0$ if $\epsilon=\Lambda=cte>0$, $J=0$, and, with respect to some fixed asymptotic chart, $(M,g)$ is AE of order $\tau$, satisfying the asymptotic condition
\begin{align}\label{asymptotic.1}
g_{ij}=\delta_{ij} + O_4(|x|^{-\tau})
\end{align}
and $K$ satisfies the decaying condition
\begin{align}\label{asymptotic.2}
\begin{split}
K_{ij} \pm c g_{ij}=O_4(|x|^{-\tau-1}),
\end{split}
\end{align}
with $c>0$ defined by $c^2\doteq \frac{\Lambda}{3}$.
\end{defn}

The above definition concerns initial data sets which are AE in the metric $g$ and asymptotically umbilical in the extrinsic curvature, and which shall evolve into an Einstein metric in space-time with some particular asymptotic behaviour. Notice that these asymptotics are weaker than the ones with which we dealt in previous sections, since they do not impose that all time derivatives of the metric must increase the rate of decay at space-like infinity. Let us furthermore notice that any time-symmetric vacuum solution of the ECE $(M^3,g,0)$ can be readily mapped into a $\Lambda$AE solution given by $(M^3,g,K=cg)$. This implies that basic examples such as Schwarzschild have their $\Lambda$AE counterpart. Properties of such solutions could be of legitimate interest in classical GR.

\begin{remark}\label{ExpContInterpretation}
If $(M^3,g,K)$ is a $\Lambda$AE initial data set, then asymptotically $\tau=\mp 3c + O(|x|^{-\tau-1})$. Therefore, denoting by $n$ the future-pointing unit normal to $M$ in the evolving space-time, under our conventions (See Appendix \ref{ExtGeomConv})
\begin{align*}
\mathrm{div}_{\bar{g}}n=\pm 3c + O(|x|^{-\tau-1}).
\end{align*}
We therefore interpreter that $K$ approaching $-cg$ asymptotically implies that, near space-like infinity, the associated space-time is expanding, while for $+cg$ it is contracting. Based on observational evidence, it might be more realistic to pick the first case. Thus, such initial data sets might be useful to describe isolated systems in an expanding background, something interesting in realistic situations. The detailed constructions and properties of such initial data sets can be objects of study on their on right.
\end{remark}

\subsection{Einstein solutions: proof of theorem \ref{MainThm1}}
In what follows, we will consider Einstein metrics constructed from admissible $\Lambda$AE initial data sets and establish conditions which guarantee that the fourth order energy is well-defined and actually computable in terms of the ADM mass of $g$ and the cosmological constant $\Lambda$. Before doing this, we need to recall the relevant notations. 

When analysing the initial value problem in GR it is convenient to consider the adapted frames of the form 
\begin{align}
\label{framebuild}
E_{i}  &  = \partial_{i}\; ,\; \; i=1,2,3\\
E_{0}  &  = \partial_{t} - X,
\end{align}
where $X=X^i\partial_i$ is the shift vector field (see definition \ref{AMmanifoldsb}) tangent to $M$. Then $E_0$ is orthogonal to $TM$.  Its dual coframe is then
\begin{align*}
\theta^{i}  &  = dx^{i} +X^{i}dt \; , \; \; i=1,2,3,\\
\theta^{0}  &  = dt,
\end{align*}
where $\{x^i\}_{i=1}^{3}$ are local coordinates on $M$ and $t$ is a coordinate on $\mathbb{R}$. Recalling that $N$ denotes the lapse function, the ambient metric $\bar g$ is written in terms of the coframe $\{\theta^a\}_{a=0}^3$ in the following way
\begin{align*}
\bar g = -N^{2}\theta^0\otimes\theta^0+ g_{ij}\theta^{i}\otimes\theta^{j}.
\end{align*}
Furthermore, using these adapted frames, it follows that 
\begin{align*}
K_{ij}=-\frac{1}{2N}\left(\partial_tg_{ij} - \pounds_{X}g_{ij} \right).
\end{align*}
We are now in position to prove theorem \ref{MainThm1} which we recall for convenience:
\begin{thm}\label{PEthm}
Let $(V^4,\bar{g})$ be an Einstein space-time generated by $\Lambda$AE initial data $\mathcal{I}$ of order $\tau$ with $R_g\in L^{1}(M^3,dV_g)$. Then, the following statements follow:
\begin{enumerate}
\item If $g$ is asymptotically Schwarzschild, then the fourth-order energy (\ref{energy.2}) is well defined for general values of $\alpha$ and $\beta$. Furthermore if $\Lambda>0$; $\alpha<0$ and $\beta\ge -\frac{3}{2}\alpha$, then $\mathcal{E}_{\alpha,\beta}(\bar{g})\geq 0$. Additionally, if $R_g\geq 0$ and $\beta > -\frac{3}{2} \alpha$, then $\mathcal{E}_{\alpha,\beta}(\bar{g})=0$ iff $(M,g)\cong (\mathbb{R}^3,\cdot)$.
\item In the special case $2\alpha+\beta=0$, the fourth order energy is well-defined for $\tau>\frac{1}{2}$. If, additionally, $R_{g}\geq 0$; $\Lambda>0$ and $\alpha<0$, then $\mathcal{E}_{\alpha,-2\alpha}(\bar{g})\geq 0$ with equality holding iff $(M,g)\cong (\mathbb{R}^3,\cdot)$.
\end{enumerate}
\end{thm}
\begin{proof}
In order to analyse the behaviour of the fourth-order energy $\mathcal{E}_{\alpha,\beta}(\bar{g})$ along these kind of solutions, we need to know the behaviour at infinity of the full initial data for $\bar{g}$, which consists of $(g,K)$ and also $(N,X,\partial_tN,\partial_tX)|_{t=0}$. From our hypotheses, we already know the asymptotic behaviour of $(g,K)$ and in what follows we will analyse the corresponding behaviour for the rest of these quantities. 

Let us recall that given an admissible initial data for the $\Lambda$-vacuum Einstein equations, in order to evolve them away of $t=0$, we need to solve the so called reduced Einstein equations, which correspond to the Einstein equations in a suitable gauge. In particular, we can pick a global gauge condition by first fixing a Riemannian metric $e$ on $M$; then defining $\hat{e}=dt^2+e$ as a metric on $V$ and finally imposing initial data so as to satisfy
\begin{align}\label{wavegauge}
\bar{g}^{\mu\nu}\left(\bar{\Gamma}^{\lambda}_{\mu\nu} - \hat{\Gamma}^{\lambda}_{\mu\nu} \right)|_{t=0}=0.
\end{align}
It is standard to see that, in this setting, initial data which satisfy (\ref{wavegauge}) evolve into a solution of the space-time Einstein equations. In particular, it follows that the initial data $N|_{t=0}$ and $X|_{t=0}$ are freely specifiable, while $\dot{N}\doteq \partial_tN|_{t=0}$ and $\dot{X}\doteq \partial_tX|_{t=0}$ are fixed by solving (\ref{wavegauge}). In this setting, taking advantage of our asymptotic structure at infinity, let us fix an asymptotic chart $\{x^i\}_{i=1}^3$ where (\ref{asymptotic.1})-(\ref{asymptotic.2}) hold and construct $e$ so as to be a complete metric which is exactly Euclidean outside a compact set. Thus, (\ref{wavegauge}) implies that near infinity, in these coordinates, it follows that
\begin{align*}
\bar{g}^{\alpha\beta}\Gamma^{\mu}_{\alpha\beta}(\bar{g})|_{t=0}=0,
\end{align*}
which (using $N|_{t=0}=1$ and $X|_{t=0}=0$) translates to
\begin{align*}
&\bar{g}^{\alpha\beta}\Gamma^{0}_{\alpha\beta}(\bar{g})|_{t=0}=-(\partial_tN|_{t=0} + \mathrm{tr}_gK)=0,\\
&\bar{g}^{\alpha\beta}\Gamma^{i}_{\alpha\beta}(\bar{g})|_{t=0}=-g^{ij}\partial_tX_i|_{t=0} + \bar{g}^{ab}\Gamma^{i}_{ab}(g)=0
\end{align*}
Thus,\footnote{The $\pm$ sign for $\dot{N}$ depends on whether we are in the \emph{asymptotically expanding} or \emph{contracting} case (see Remark \ref{ExpContInterpretation}).} 
\begin{align}\label{asymptotic.3}
\begin{split}
\dot{N} &= \pm 3c + O_{3}(|x|^{-\tau-1}),\\
\dot{X}_j &= g_{ji}g^{ab}\Gamma^{i}_{ab}(g).
\end{split}
\end{align}
Associated to the initial $(g,N,X,K,\dot{N},\dot{X})$ we have a space-time $(M\times I,\bar{g})$, with $I=[0,T)$ for some $T>0$, where $\bar{g}$ is Einstein and thus $A_{\bar{g}}=0$. Let us evaluate the fourth order energy on these solutions. From our lapse-shift conditions, we get
\begin{align*}
\mathcal{E}_{\alpha,\beta}(\bar{g})&=\lim_{r\rightarrow\infty}\Big\{\left( \frac{3}{2}\beta  + 2\alpha\right)\int_{S^{n-1}_r} \left( \partial_{j}\partial_{i}\partial_{i}g_{aa} - \partial_{j}\partial_{u}\partial_{i}g_{u i}\right)\hat{\nu}^{j}d\omega_{r}   \\
&+ \frac{\beta}{2}\int_{S^{n-1}_r}\left( \partial_{i}\ddot{g}_{ji} - \partial_{j}\ddot{g}_{ii} \right)\hat{\nu}^{j}d\omega_r + \frac{\beta}{2}\int_{S^{n-1}_r}\left( \partial_{i}\partial_{j}\dot{X}_{i} - \partial_{i}\partial_{i}\dot{X}_{j} \right)\hat{\nu}^{j}d\omega_r \\
&+ ( \beta  + 2\alpha)\left( 2\int_{S^{n-1}_{r}}\partial_{j}\partial_{i}\dot{X}_{i}\hat{\nu}^{j}d\omega_r - \int_{S^{n-1}_r}\partial_{j}\ddot{g}_{ii}\hat{\nu}^{j}d\omega_r \right)  \Big\}
\end{align*} 
Now, from our hypotheses and (\ref{asymptotic.3}) we see that
\begin{align*}
\partial_{ijk}g=O(|x|^{-(\tau +3 )}) \; , \; \dot{X}=O_2(|x|^{-(\tau + 1)}) \; , \; \partial_{ij}\dot{X}=O(|x|^{-(\tau+3)}).
\end{align*}
Also, in the asymptotic region, we have that 
\begin{align}\label{gauge0}
\begin{split}
R_{\alpha\beta}(\bar{g})&=-\frac{1}{2}\bar{g}^{\lambda\mu}\partial_{\lambda}\partial_{\mu}\bar{g}_{\alpha\beta} - \frac{1}{2}\big\{  \partial_{\beta}\bar{g}^{\nu\lambda}\partial_{\nu}\bar{g}_{\lambda\alpha}  + \partial_{\alpha}\bar{g}^{\nu\lambda}\partial_{\nu}\bar{g}_{\lambda\beta} \}  - \bar{\Gamma}^{\nu}_{\sigma\beta}\bar{\Gamma}^{\sigma}_{\nu\alpha},
\end{split}
\end{align}
and since $\bar{g}$ is Einstein
\begin{align*}
\mathrm{Ric}_{\bar{g}}=\Lambda\bar{g},
\end{align*}
therefore, due to $N|_{t=0}=1$ and $X|_{t=0}=0$, we get 
\begin{align*}
\begin{split}
\frac{1}{2}\partial^2_{t}\bar{g}_{ij}|_{t=0} &=\Big( \Lambda\bar{g}_{ij} + \frac{1}{2}\bar{g}^{ab}\partial_{a}\partial_{b}\bar{g}_{ij} + \frac{1}{2}\big(  \partial_{j}\bar{g}^{\nu\lambda}\partial_{\nu}\bar{g}_{i \lambda}  + \partial_{i}\bar{g}^{\nu\lambda}\partial_{\nu}\bar{g}_{j\lambda} \big) + \bar{\Gamma}^{\nu}_{j\sigma}\bar{\Gamma}^{\sigma}_{i\nu}\Big)\Big|_{t=0},
\end{split}
\end{align*}
We can also compute
\begin{align*}
\bar{\Gamma}^{\nu}_{j\sigma}\bar{\Gamma}^{\sigma}_{i\nu}\Big|_{t=0}&=\frac{1}{4}\bar{g}^{\nu\mu}\bar{g}^{\sigma\gamma}
\left(\partial_{j}\bar{g}_{\mu\sigma} + \partial_{\sigma}\bar{g}_{\mu j} - \partial_{\mu}\bar{g}_{j\sigma} \right)\left(\partial_{i}\bar{g}_{\gamma\nu} + \partial_{\nu}\bar{g}_{\gamma i} - \partial_{\gamma}\bar{g}_{i\nu} \right)\Big|_{t=0},\\
&=\frac{1}{4}\Big\{2\bar{g}^{ab}\partial_{t}\bar{g}_{ja}  \partial_{t}\bar{g}_{b i}+\bar{g}^{ab}\bar{g}^{cd} \left(\partial_{j}\bar{g}_{ac} + \partial_{c}\bar{g}_{a j} - \partial_{a}\bar{g}_{jc} \right)\left(\partial_{i}\bar{g}_{d b} + \partial_{b}\bar{g}_{d i} - \partial_{d}\bar{g}_{ib} \right)\Big\}\Big|_{t=0}\\
&=\frac{1}{2} g^{ab} \dot{g}_{a j}  \dot{g}_{ib} + \frac{1}{4}g^{ab}g^{cd} \left(\partial_{j}g_{ac} + \partial_{c}g_{a j} - \partial_{a}g_{jc} \right)\left(\partial_{i}g_{d b} + \partial_{b}g_{d i} - \partial_{d}g_{ib} \right),
\end{align*}
where we used $X|_{t=0}=0,\partial_aN|_{t=0}=0$. In the above identity, $\partial_ag=O(|x|^{-(\tau+1)})$ while $\dot{g}$ has the same asymptotics as $g$. Explicitly, 
\begin{align*}
\dot{g}=\left(-2N K + \pounds_Xg \right)|_{t=0}=-2K
\end{align*}
which, for our initial data sets, means $\dot{g}=\pm 2cg+O_2(|x|^{-(\tau+1)})$. Therefore, 
\begin{align*}
\bar{\Gamma}^{\nu}_{j\sigma}\bar{\Gamma}^{\sigma}_{i\nu}\Big|_{t=0}&= 2c^2 g^{ab} g_{a j}  g_{ib} + O_2(|x|^{-2(\tau+1)})=\frac{2\Lambda}{3}g_{ij} + O_2(|x|^{-2(\tau+1)}).
\end{align*}
With similar arguments, we can also compute that
\begin{align*}
\big( \partial_{j}\bar{g}^{\nu\lambda}\partial_{\nu}\bar{g}_{i \lambda}  + \partial_{i}\bar{g}^{\nu\lambda}\partial_{\nu}\bar{g}_{j\lambda} \big)|_{t=0}&=\big( \partial_{j}\bar{g}^{0\lambda}\partial_{0}\bar{g}_{i \lambda} + \partial_{j}\bar{g}^{a\lambda}\partial_{a}\bar{g}_{i \lambda}  + \partial_{i}\bar{g}^{0\lambda}\partial_{0}\bar{g}_{j\lambda} + \partial_{i}\bar{g}^{a\lambda}\partial_{a}\bar{g}_{j\lambda} \big)|_{t=0},\\
&=\big( \partial_{j}\bar{g}^{a0}\partial_{a}\bar{g}_{i 0} + \partial_{j}\bar{g}^{ab}\partial_{a}\bar{g}_{ib}  + \partial_{i}\bar{g}^{a0}\partial_{a}\bar{g}_{j0} + \partial_{i}\bar{g}^{ab}\partial_{a}\bar{g}_{jb} \big)|_{t=0},\\
&=\big( \partial_{j}\bar{g}^{ab}\partial_{a}\bar{g}_{ib}  + \partial_{i}\bar{g}^{ab}\partial_{a}\bar{g}_{jb} \big)|_{t=0},\\
&=O_2(|x|^{-2(\tau+1)}).
\end{align*}
Putting all the above together, we find that
\begin{align*}
\begin{split}
\frac{1}{2}\partial^2_{t}\bar{g}_{ij}|_{t=0} &= \Lambda g_{ij} + \frac{2\Lambda}{3}g_{ij} + O_2(|x|^{-(\tau+2)}) ,\\
&=\frac{5}{3}\Lambda g_{ij} + O_2(|x|^{-(\tau+2)}) .
\end{split}
\end{align*}
Thus, the following holds
\begin{align*}
\frac{1}{2}\partial_k\ddot{g}_{ij}=\frac{5}{3}\Lambda\partial_kg_{ij} + O_1(|x|^{-(\tau+3)}).
\end{align*}
Thus, the leading order is given by the first term in the right-hand side, which behaves like $O_3(|x|^{-(\tau+1)})$. Using $\partial$ to denote arbitrary spacial derivatives, notice that
\begin{align*}
\partial^3g_{ij}r^{2}&=O(r^{-(\tau + 1)}),\\
\partial^2\dot{X}_ir^{2}&=O(r^{-(\tau + 1)}).
\end{align*}
The above implies that, in these cases, the only contributions to the energy can come from
\begin{align*}
\begin{split}
\mathcal{E}_{\alpha,\beta}(\bar{g})&=\lim_{r\rightarrow\infty}\Big\{\frac{\beta}{2}\int_{S^{n-1}_r}\left( \partial_{i}\ddot{g}_{ji} - \partial_{j}\ddot{g}_{ii} \right)\hat{\nu}^{j}d\omega_r - ( \beta  + 2\alpha)\int_{S^{n-1}_r}\partial_{j}\ddot{g}_{ii}\hat{\nu}^{j}d\omega_r  \Big\}
,\\
&=\frac{5}{3}\Lambda\lim_{r\rightarrow\infty}\Big\{\beta\int_{S^{n-1}_r}\left( \partial_{i}g_{ji} - \partial_{j}g_{ii} \right)\hat{\nu}^{j}d\omega_r - 2( \beta  + 2\alpha)\int_{S^{n-1}_r}\partial_{j}g
_{ii}\hat{\nu}^{j}d\omega_r  \Big\}
,
\end{split}
\end{align*}
From the above we see that, for general $\alpha$ and $\beta$, if $g$ is asymptotically Schwarzschild, then
\begin{align*}
\left(\partial_{j}g_{ii}-\partial_ig_{ji}\right)\nu_j&=-4m\left(1+\frac{m}{2r}\right)^{3}r^{-2} + o(r^{-2}),
\end{align*}
and this implies
\begin{align*}
\lim_{r\rightarrow\infty}\int_{S_r}(\partial_ig_{ji}-\partial_{j}g_{ii})\nu_jr^{2}d\omega_{2}&=4\omega_{2}m,\\
\lim_{r\rightarrow\infty}\int_{S_r}\partial_{j}g_{ii}\nu_jr^{2}d\omega_{2}&= -6\omega_{2}m,
\end{align*}
establishing


\begin{align*}
\mathcal{E}_{\alpha,\beta}(\bar{g})&=40\omega_{2}m\Lambda\left( \frac{2}{3}\beta + \alpha \right)\geq 0 \text{ iff } \beta\geq -\frac{3}{2}\alpha,
\end{align*}
which gives us a positive energy statement for Einstein solutions which are asymptotically Schwarzschild. On the other hand, we can weaken the asymptotics if we consider the special choice $\beta=-2\alpha$. In this case,
\begin{align*}
\mathcal{E}_{\alpha,-2\alpha}(\bar{g})&=-\frac{10}{3}\alpha\Lambda \lim_{r\rightarrow\infty}\int_{S_r}(\partial_ig_{ji}-\partial_{j}g_{ii})\nu_jr^{2}d\omega_{2}.
\end{align*}
Recognizing the last factor as the ADM energy from GR, we see that under the geometric condition $R_g\geq 0$ and $g_{ij}-\delta_{ij}=O_2(|x|^{-\tau})$ for $\tau > \frac{1}{2}$, it follows from the positive energy theorem in GR that if $\alpha<0$, then $\mathcal{E}_{\alpha,-2\alpha}(\bar{g})\geq 0$ for any umbilical Einstein solution with $\Lambda>0$ \cite{PM1}. The rigidity follows,  if $\beta > -\frac{3}{2} \alpha$, from the same results associated to the PMT in GR, which for the kind of asymptotics on $g$ considered here can be consulted, for instance, in \cite{Bartnik,Eichmair1}.

\qed
\end{proof}

\begin{remark}
Let us highlight that there is an explicit dependence of $\mathcal{E}_{\alpha,\beta}$ on the chosen slicing for space-time, which comes about trough the explicit dependence on the initial values of $(N,X,\dot{N},\dot{X})$. Although these values do not affect the evolving space-time due to geometric uniqueness of the Cauchy problem (see, for instance, \cite[Theorem 8.9, Chapter 6]{CB-book}), these values determine the \emph{space-time observers} along whose flow we evolve the initial data near the initial Cauchy surface. What we explicitly see is that the energy $\mathcal{E}_{\alpha,\beta}$ is sensitive to such a choice. This could be expected on physical grounds as the energy measured by different asymptotic observers should be dependent on the observers. Furthermore, we could even draw some experience with the ADM energy, which is known to be sensitive to such choices (see, for instance, \cite[Chapter 1, Section 1.1.3]{ChruscielEnergyNotes}). Let us however notice that, by examination of the above proof, it is only the behaviour near infinity of these observers that matter, and therefore there will be a class of observers, asymptotic to the one defined by $N=1$ and $X=0$, who perceive the same value to the energy. This will be the subject of the next subsection.
\end{remark}

\subsection{Einstein solutions: proof of theorem \ref{MainThm2}}
We introduce the following definition.

\begin{defn}\label{AsymptoticObserversDefns}
Given a ($\Lambda$)AE initial data set $\mathcal{J}$ we define the \emph{canonical space-time observers} $\mathcal{O}_0$ as those whose flow lines in the evolving space-time satisfy $N_0=1$ and $X_0=0$ at $t=0$. If $\{x^i\}_{i=1}^n$ denote asymptotic coordinates for $M^n$, we will say that a different set of observers $\mathcal{O}_\rho$ are asymptotic of order $\rho$ to $\mathcal{O}_0$ on $M$ if $N-1=O_4(|x|^{-\rho})$ and $X^i=O_4(|x|^{-\rho})$.
\end{defn}

We now intend to show that there exist a class of observers which are asymptotic to the canonical observers $\mathcal{O}_0$ for which the energy $\mathcal{E}_{\alpha,\beta}(\bar{g})$ represents a asymptotic invariant, in the sense of having the same value for every observer in the class, and furthermore, in those cases, the energy is independent of the asymptotic chart used to compute as long as the chart satisfies a minimal order of decay for the metric.\footnote{See Theorem \ref{InvarianceUmbilicalEnergyTHM} for a detailed statement.} The proof of such a statement is computationally heavy and therefore we first present a computational lemma, which can be of interest on its own, since in it we shall explicitly compute the $\hat{e}$-wave gauge condition which determined $(\dot{N},\dot{X})$ in the initial value problem for a general observer.

\begin{lemma}\label{GeneralGaugeConditions}
Consider a space-time $(V\doteq \mathbb{R}\times M^n,\bar{g})$ and let $\{E_{\alpha}\}^{n+1}_{\alpha=0}$ be an adapted orthogonal frame constructed as in \eqref{framebuild}. Let also $e$ be a Riemannian metric on $M$ and $\hat{e}\doteq dt^2+e$ a Riemannian metric on $V$, where $t:V\mapsto \mathbb{R}$ stands for a coordinate choice for the $\mathbb{R}$-factor and we assume that $\partial_t$ is time-like. Denote the associated Riemmanian covariant derivatives to $e$ and $\hat{e}$ by $D$ and $\hat{D}$ respectively. Let us denote by $(N,X)$ the lapse function and shift vector fields associated to the time-like vector field $\partial_t$. Then, denoting by $\mathcal{S}\in \Gamma(T^1_{2}V)$ the tensor field defined by $S^{\mu}_{\alpha\beta}\doteq \Gamma^{\mu}_{\alpha\beta}(\bar{g}) - \Gamma^{\mu}_{\alpha\beta}(\hat{e})$, the following decompositions hold:
\begin{align*}
\begin{split}
\mathcal{S}^{0}_{00}&=N^{-1}(\partial_tN - X(N)),\\
\mathcal{S}^{0}_{ab}&=\frac{N^{-2}}{2}( - 2NK + (X\otimes d|X|^2_e)_{\mathrm{Sym}} - 2(\bar{g}_{\cdot}DX)_{\mathrm{Sym}} + \pounds_{X}\bar{g} - D_X\bar{g})_{ab},\\
\mathcal{S}^{j}_{00}&={\partial_tX^j - \frac{X^j}{2}\partial_t|X|^2_e + \frac{X^{j}}{2}X(|X|^2_e) - D_XX^{j} + N\bar{\nabla}^jN},\\
\mathcal{S}^{j}_{ab}&=\frac{\bar{g}^{jk}}{2}\left(D_{a}\bar{g}_{kb} + D_{b}\bar{g}_{ka} - D_{k}\bar{g}_{ab}  \right)
\end{split}
\end{align*}
where above $(\bar{g}_{\cdot}DX)_{ab}\doteq g_{bl}D_aX^l$ and we have used the subscript $\mathrm{Sym}$ to denote symmetrization of $2$-tensors.
\end{lemma}
\begin{proof}
Let us start by recalling that, in any basis (including the one given by our adapted frame) the tensor $S$ is written
\begin{align*}
\mathcal{S}^{\mu}_{\alpha\beta}&=\frac{\bar{g}^{\mu\sigma}}{2}\left(\hat{D}_{\alpha}\bar{g}_{\sigma\beta} + \hat{D}_{\beta}\bar{g}_{\sigma\alpha} - \hat{D}_{\sigma}\bar{g}_{\alpha\beta}  \right)
\end{align*}
and also that in coordinates related to an arbitrary frame the Levi-Civita connection coefficients read as
\begin{align*}
\hat{\Gamma}^{\mu}_{\alpha\beta}=\frac{\hat{e}^{\mu\nu}}{2}\left(E_{\alpha}(\hat{e}_{\nu\beta}) + E_{\beta}(\hat{e}_{\nu\alpha}) - E_{\nu}(\hat{e}_{\alpha\beta}) \right) + \frac{\hat{e}^{\mu\nu}}{2}\left(C^{\sigma}_{\nu\alpha}\hat{e}_{\sigma\beta} + C^{\sigma}_{\nu\beta}\hat{e}_{\sigma\alpha} - C^{\sigma}_{\beta\alpha}\hat{e}_{\nu\sigma} \right)
\end{align*}
where $C^{\sigma}_{\gamma\rho}\doteq [E_{\gamma},E_{\rho}]^{\sigma}$. Furthermore, the following identities also hold: 
\begin{align*}
\hat{e}(E_0,E_j)&=\hat{e}(\partial_t,E_j) - \hat{e}(X,E_j)=- e(X,E_j)\doteq-X^{\flat}_j\\
\hat{e}(E_0,E_0)&=\hat{e}(\partial_t,\partial_t) + \hat{e}(X,X)=1+|X|^2_e,\\
\hat{e}^{00}&=\hat{e}^{\sharp}(\theta^0,\theta^0)=\hat{e}^{tt}=1,\\
\hat{e}^{0j}&=\hat{e}^{\sharp}(\theta^0,\theta^j)=\hat{e}^{\sharp}(dt,dx^j) + \hat{e}^{\sharp}(dt,X^jdt)=X^j
\end{align*}
For brevity's sake  we will only detail the computations for the first term. Then,
\begin{align*}
\mathcal{S}^{0}_{00}(\bar{g})&=\frac{\bar{g}^{00}}{2}\hat{D}_{0}\bar{g}_{00}=-\frac{N^{-2}}{2}(\hat{D}_{\partial_t}\bar{g}_{00} - \hat{D}_{X}\bar{g}_{00}).
\end{align*}
where
\begin{align*}
\hat{D}_{X}\bar{g}_{00}&=X^j(E_j(g_{00}) - \hat{\Gamma}^{\alpha}_{j0}\bar{g}_{\alpha 0 } - \hat{\Gamma}^{\alpha}_{j0}\bar{g}_{0\alpha})=X^j(E_j(g_{00}) - \hat{\Gamma}^{0}_{j0}\bar{g}_{00} - \hat{\Gamma}^{0}_{j0}\bar{g}_{00}),\\
&=X^j(E_j(g_{00}) + 2N^{2}\hat{\Gamma}^{0}_{j0})=-X(N^2),
\end{align*}
and we used that $\hat{\Gamma}^{0}_{j0}X^j=0$, which follows from 
\begin{align*}
\hat{\Gamma}^{0}_{j0}&=\frac{\hat{e}^{0\nu}}{2}\left(E_{j}(\hat{e}_{\nu 0}) + E_{0}(\hat{e}_{\nu j}) - E_{\nu}(\hat{e}_{j0}) \right) + \frac{\hat{e}^{0\nu}}{2}\left(C^{\sigma}_{\nu j}\hat{e}_{\sigma 0} + C^{\sigma}_{\nu 0}\hat{e}_{\sigma j} - C^{\sigma}_{0j}\hat{e}_{\nu\sigma} \right),\\
&=\frac{\hat{e}^{00}}{2}(E_{j}(\hat{e}_{0 0}) + \underbrace{E_{0}(\hat{e}_{0 j}) - E_{0}(\hat{e}_{j0})}_{=0} ) + \frac{\hat{e}^{00}}{2}(C^{\sigma}_{0j}\hat{e}_{\sigma 0} + \underbrace{C^{\sigma}_{0 0}}_{=0}\hat{e}_{\sigma j} - C^{\sigma}_{0j}\hat{e}_{0\sigma} )\\
&+\frac{\hat{e}^{0k}}{2}\left(E_{j}(\hat{e}_{k 0}) + E_{0}(\hat{e}_{k j}) - E_{k}(\hat{e}_{j0}) \right) + \frac{\hat{e}^{0k}}{2}(\underbrace{C^{\sigma}_{kj}}_{=0}\hat{e}_{\sigma 0} + C^{\sigma}_{k 0}\hat{e}_{\sigma j} - C^{\sigma}_{0j}\hat{e}_{k\sigma} ),\\
&=\frac{1}{2}E_{j}(|X|^2_e)  +\frac{X^{k}}{2}( \partial_t(\hat{e}_{k j}) - X(\hat{e}_{k j}) + \underbrace{E_{k}(X^{\flat}_{j}) - E_{j}(X^{\flat}_{k})}_{={dX^{\flat}_{kj}}} ) + \frac{X^{k}}{2}( -E_k(X^{l})\hat{e}_{l j} - E_j(X^{l})\hat{e}_{kl} ),\\
&=\frac{1}{2}E_{j}(|X|^2_e)  +\frac{X^{k}}{2}( \underbrace{\partial_t(\hat{e}_{k j})}_{-2\hat{K}_{kj}=0} + {dX^{\flat}_{kj}} ) - \frac{X^{l}}{2}X^kE_k(\hat{e}_{l j}) + \frac{X^{k}}{2}( -E_k(X^{l})\hat{e}_{l j} - E_j(X^{l})\hat{e}_{kl} ),\\
&=\frac{1}{2}E_{j}(|X|^2_e)  +\frac{1}{2} X^{k} {dX^{\flat}_{kj}}  - \frac{X^{k}}{2}( \underbrace{X^lE_k(\hat{e}_{l j}) + E_k(X^{l})\hat{e}_{l j}}_{=E_k(X^{\flat}_j)} + E_j(X^{l})\hat{e}_{kl} ),\\
&=\frac{1}{2}E_{j}(|X|^2_e) + \frac{1}{2} {dX^{\flat}_{kj}}X^{k}  - \frac{1}{2}( X^{k}\underbrace{E_k(X^{\flat}_j)}_{=dX^{\flat}_{kj} + E_j(X^{\flat}_k)} + E_j(X^{l})X^{\flat}_l ),\\
&=\frac{1}{2}E_{j}(|X|^2_e) + \frac{1}{2}{dX^{\flat}_{kj}}X^{k}  - \frac{1}{2}\left( dX^{\flat}_{kj}X^{k} + E_j(|X|^2_e)  \right),\\
&={0.}
\end{align*}
Above we used that $X^{\flat} = \hat{e}(X,\cdot)$ and
\begin{align*}
C^{\sigma}_{0j}=[E_0,E_j]^{\sigma}=-[X,E_j]^{\sigma}=E_jX^{\sigma}.
\end{align*}

Let us highlight that here  we denoted $dX^{\flat}_{jk} = E_{j}(X^{\flat}_{k}) - E_{k}(X^{\flat}_{j})$, to be coherent with our differential forms notations (see \eqref{differentialformnotation} in the appendix).

Similarly, 
\begin{align*}
\hat{D}_{\partial_t}\bar{g}_{00}&=\partial_tg_{00} - 2\hat{\Gamma}^{\alpha}_{t0}\bar{g}_{\alpha 0} = - \partial_tN^2 -2\hat{\Gamma}^{0}_{t0}\bar{g}_{00},\\
&=- \partial_tN^2 + 2N^2\hat{\Gamma}^{0}_{t0}=- \partial_tN^2.
\end{align*}
where we used
\begin{align*}
\hat{\Gamma}^{0}_{t 0}&=\frac{\hat{e}^{0\nu}}{2}\left(E_{t}(\hat{e}_{\nu 0}) + E_{0}(\hat{e}_{\nu t}) - E_{\nu}(\hat{e}_{t 0}) \right) + \frac{\hat{e}^{0\nu}}{2}\left(C^{\sigma}_{\nu t}\hat{e}_{\sigma 0} + C^{\sigma}_{\nu 0}\hat{e}_{\sigma t} - C^{\sigma}_{0 t}\hat{e}_{\nu\sigma} \right),\\
&=\frac{1}{2}\partial_{t}|X|^2_e  - \frac{1}{2}( \partial_tX^{l}X^{\flat}_l + X^{k}\partial_{t}X^{\flat}_{k})=0.
\end{align*}
Therefore, 
\begin{align*}
\mathcal{S}^{0}_{00}(\bar{g})&=\frac{N^{-2}}{2}(\partial_tN^2 - X(N^2))=N^{-1}(\partial_tN - X(N)).
\end{align*}

Also
\begin{align*}
\begin{split}
\mathcal{S}^{j}_{00}(\bar{g})&=\frac{\bar{g}^{jk}}{2}\left( 2\hat{D}_{0}\bar{g}_{k 0} - \hat{D}_{k}\bar{g}_{00}  \right)=\frac{\bar{g}^{jk}}{2}\left( 2\hat{D}_{\partial_t}\bar{g}_{k 0} - 2\hat{D}_{X}\bar{g}_{k 0}  - \hat{D}_{k}\bar{g}_{00}  \right).
\end{split}
\end{align*}
Now, one can compute that
\begin{align*}
\hat{D}_{\partial_t}\bar{g}_{k 0}&=\partial_t\bar{g}_{k 0} - \hat{\Gamma}^{\alpha}_{tk}\bar{g}_{\alpha 0} - \hat{\Gamma}^{\alpha}_{t0}\bar{g}_{k\alpha}= - \hat{\Gamma}^{\alpha}_{tk}\bar{g}_{00} - \hat{\Gamma}^{j}_{t0}\bar{g}_{kj} =N^2 \hat{\Gamma}^{0}_{tk} - \hat{\Gamma}^{j}_{t0}\bar{g}_{kj}
\end{align*}
and since
\begin{align*}
\hat{\Gamma}^{0}_{tk}&=\frac{\hat{e}^{0\nu}}{2}(E_{t}(\hat{e}_{\nu k}) +E_{k}(\hat{e}_{\nu t}) - E_{\nu}({\hat{e}_{tk}}) ) + \frac{\hat{e}^{0\nu}}{2}(C^{\sigma}_{\nu t}\hat{e}_{\sigma k} + C^{\sigma}_{\nu k}\hat{e}_{\sigma t} - {C^{\sigma}_{k t}}\hat{e}_{\nu\sigma} )
=-\frac{1}{2}\partial_{t}X^{\flat}_{k} + \frac{1}{2}\partial_tX^{l}\hat{e}_{lk},
\end{align*}
and
\begin{align*}
\hat{\Gamma}^{j}_{t0}&=\frac{\hat{e}^{j\nu}}{2}\left(E_{t}(\hat{e}_{\nu 0}) + E_{0}(\hat{e}_{\nu t}) - E_{\nu}(\hat{e}_{t0}) \right) + \frac{\hat{e}^{j\nu}}{2}\left(C^{\sigma}_{\nu t}\hat{e}_{\sigma 0} + C^{\sigma}_{\nu 0}\hat{e}_{\sigma t} - C^{\sigma}_{0t}\hat{e}_{\nu\sigma} \right)
=\frac{X^{j}}{2}\partial_{t}(|X|^2_e) - \partial_{t}X^{j},
\end{align*}
 we find
\begin{align*}
\hat{D}_{\partial_t}\bar{g}_{k 0}
&=-\bar{g}_{kj}(\frac{X^{j}}{2}\partial_{t}(|X|^2_e) - \partial_{t}X^{j} )=\bar{g}_{kj}(\partial_{t}X^{j} - \frac{X^{j}}{2}\partial_{t}(|X|^2_e)).
\end{align*}

Now, since
\begin{align*}
\hat{D}_{X}\bar{g}_{k 0}=X^j(E_j(\bar{g}_{k0} )- \hat{\Gamma}^{\alpha}_{jk}\bar{g}_{\alpha 0} - \hat{\Gamma}^{\alpha}_{j0}\bar{g}_{k\alpha})=-X^j(\hat{\Gamma}^{0}_{jk}\bar{g}_{00} + \hat{\Gamma}^{l}_{j0}\bar{g}_{kl})=-X^j(-N^2\hat{\Gamma}^{0}_{jk} + \hat{\Gamma}^{l}_{j0}\bar{g}_{kl}),
\end{align*}
we need to compute the following quantities:
\begin{align*}
\hat{\Gamma}^{0}_{jk}&=\frac{\hat{e}^{0\nu}}{2}(E_{j}(\hat{e}_{\nu k}) + E_{k}(\hat{e}_{\nu j}) - E_{\nu}(\hat{e}_{jk}) ) + \frac{\hat{e}^{0\nu}}{2}(C^{\sigma}_{\nu j}\hat{e}_{\sigma k} + C^{\sigma}_{\nu k}\hat{e}_{\sigma j} - \underbrace{C^{\sigma}_{kj}}_{=0}\hat{e}_{\nu\sigma} )
=0
\end{align*}
and 
\begin{align*}
\hat{\Gamma}^{l}_{j0}&=\frac{\hat{e}^{l\nu}}{2}(E_{j}(\hat{e}_{\nu 0}) + E_{0}(\hat{e}_{\nu j}) - E_{\nu}(\hat{e}_{j0}) ) + \frac{\hat{e}^{l\nu}}{2}(C^{\sigma}_{\nu j}\hat{e}_{\sigma 0} + C^{\sigma}_{\nu 0}\hat{e}_{\sigma j} - C^{\sigma}_{0 j}\hat{e}_{\nu\sigma} )\\
&=\frac{X^{l}}{2}(E_{j}(|X|^2_e) )- D_jX^{l}.
\end{align*}
Thus
\begin{align*}
\begin{split}
\hat{\Gamma}^{l}_{j0}\bar{g}_{kl}&=\frac{X_{k}}{2}(E_{j}(|X|^2_e) ) - \bar{g}_{kl}D_jX^{l}
\end{split}
\end{align*}
implying
\begin{align*}
\hat{D}_{X}\bar{g}_{k 0}&=-X^j(-N^2\hat{\Gamma}^{0}_{jk} + \hat{\Gamma}^{l}_{j0}\bar{g}_{kl})=-\frac{X_{k}}{2}X(|X|^2_e) + \bar{g}_{kl}D_XX^{l}
\end{align*}
Thus
\begin{align*}
\mathcal{S}^{j}_{00}(\bar{g})
&=\partial_tX^j - \frac{X^j}{2}\partial_t|X|^2_e + \frac{X^{k}}{2}X(|X|^2_e) - D_XX^{j} + N\bar{\nabla}^jN.
\end{align*}

Also
\begin{align*}
\mathcal{S}^{\mu}_{ab}&=\frac{\bar{g}^{\mu\sigma}}{2}\left(\hat{D}_{a}\bar{g}_{\sigma b} + \hat{D}_{b}\bar{g}_{\sigma a} - \hat{D}_{\sigma}\bar{g}_{ab}  \right),\\
&= \frac{\bar{g}^{\mu 0}}{2}(\hat{D}_{a}\bar{g}_{0 b} + \hat{D}_{b}\bar{g}_{0 a} - \hat{D}_{0}\bar{g}_{ab}  ) + \frac{1}{2}(\bar{g}^{\mu k}\hat{D}_{a}\bar{g}_{k b} + \bar{g}^{\mu k}\hat{D}_{b}\bar{g}_{k a} - \bar{g}^{\mu k}\hat{D}_{k}\bar{g}_{ab}  ),
\end{align*}
and noticing that
\begin{align*}
\hat{D}_{a}\bar{g}_{0 b}&=-\hat{\Gamma}^{\alpha}_{a0}\bar{g}_{\alpha b} -\hat{\Gamma}^{\alpha}_{ab}\bar{g}_{0\alpha}=-\hat{\Gamma}^{l}_{a0}\bar{g}_{lb} -\underbrace{\hat{\Gamma}^{0}_{ab}}_{=0}\bar{g}_{00}= - \frac{X_{b}}{2}(E_{a}(|X|^2_e) )+\bar{g}_{lb}\hat{e}^{lu}DX^{\flat}_{au},\\
\hat{D}_{b}\bar{g}_{0 a}&=- \frac{X_{a}}{2}(e_{b}(|X|^2_e) )+\bar{g}_{la}\hat{e}^{lu}DX^{\flat}_{bu}
\end{align*}
and also that
\begin{align*}
\hat{D}_{0}\bar{g}_{ab}&=\hat{D}_{\partial_t}\bar{g}_{ab} - \hat{D}_{X}\bar{g}_{ab}=\partial_t\bar{g}_{ab}-\hat{\Gamma}^{\alpha}_{ta}\bar{g}_{\alpha b} - \hat{\Gamma}^{\alpha}_{tb}\bar{g}_{a\alpha} - X^{u}\hat{D}_u\bar{g}_{ab},\\
&=-2NK_{ab} + \pounds_{X}\bar{g}_{ab} - \underbrace{\hat{\Gamma}^{l}_{ta}}_{=0}\bar{g}_{lb} - \underbrace{\hat{\Gamma}^{l}_{tb}}_{=0}\bar{g}_{al} - \underbrace{X^{u}\hat{D}_u\bar{g}_{ab}}_{=D_X\bar{g}_{ab}},\\
&=-2NK_{ab} + \pounds_{X}\bar{g}_{ab} - D_X\bar{g}_{ab}
\end{align*}
we see that
\begin{align*}
\mathcal{S}^{\mu}_{ab}&= \frac{\bar{g}^{\mu 0}}{2}(\hat{D}_{a}\bar{g}_{0 b} + \hat{D}_{b}\bar{g}_{0 a} - \hat{D}_{0}\bar{g}_{ab}  )+ \frac{1}{2}(\bar{g}^{\mu k}\hat{D}_{a}\bar{g}_{k b} + \bar{g}^{\mu k}\hat{D}_{b}\bar{g}_{k a} - \bar{g}^{\mu k}\hat{D}_{k}\bar{g}_{ab}  ),\\
&=\frac{\bar{g}^{\mu 0}}{2}(- \frac{X_{b}}{2}(e_{a}(|X|^2_e) )+\bar{g}_{lb}D_aX^{l} - \frac{X_{a}}{2}(E_{b}(|X|^2_e) )+\bar{g}_{la}D_bX^{l} + 2NK_{ab} - \pounds_{X}\bar{g}_{ab} + D_X\bar{g}_{ab})\\
&+ \frac{\bar{g}^{\mu k}}{2}(\hat{D}_{a}\bar{g}_{k b} + \hat{D}_{b}\bar{g}_{k a} - \hat{D}_{k}\bar{g}_{ab}  ).
\end{align*}
Thus, 
\begin{align*}
\mathcal{S}^{0}_{ab}&=\frac{N^{-2}}{2}( - 2NK_{ab} + \frac{X_{b}}{2}D_{a}(|X|^2_e) + \frac{X_{a}}{2}D_{b}(|X|^2_e) - \bar{g}_{lb}D_aX^{l} - \bar{g}_{la}D_bX^{l} + \pounds_{X}\bar{g}_{ab} - D_X\bar{g}_{ab}),
\end{align*}
and similarly, 
\begin{align*}
\mathcal{S}^{j}_{ab}
&=\frac{\bar{g}^{jk}}{2}\left(D_{a}\bar{g}_{kb} + D_{b}\bar{g}_{ka} - D_{k}\bar{g}_{ab}  \right)
\end{align*}
\qed
\end{proof}

We can now appeal to the above Lemma to prove theorem \ref{MainThm2} which we recall: 

\begin{thm}\label{InvarianceUmbilicalEnergyTHM}
Under the same hypotheses as in Theorem \ref{PEthm}, given a $\Lambda$AE initial data set and two asymptotic observers $\mathcal{O}_1$ and $\mathcal{O}_2$ of orders $\rho> \frac{1}{2}$, if we denote the energies associated to them by $\mathcal{E}^{(\mathcal{O}_i)}_{\alpha,\beta}$, $i=1,2$, then
\begin{align*}
\mathcal{E}^{(\mathcal{O}_1)}_{\alpha,\beta}(\bar{g})=\mathcal{E}^{(\mathcal{O}_2)}_{\alpha,\beta}(\bar{g}).
\end{align*}
Furthermore, if $\phi_x,\phi_y:M\backslash K\mapsto \mathbb{R}^3\backslash \overline{B_1(0)}$ are two asymptotic charts where $g$ is of order $\tau_x,\tau_y>\frac{n-2}{2}$ respectively, and where the general hypotheses of Theorem \ref{PEthm} are satisfied, then 
\begin{enumerate}
\item If ${\phi_x^{-1}}^{*}g, {\phi_y^{-1}}^{*}g$ are both AS, the value of  $\mathcal{E}^{(\mathcal{O}_i)}_{\alpha,\beta}$ is the same for both coordinate systems for general values of $\alpha$ and $\beta$.
\item If $2\alpha+\beta=0$, the value of  $\mathcal{E}^{(\mathcal{O}_i)}_{\alpha,\beta}$ is equal in both asymptotic coordinate systems.
\end{enumerate}
%
\end{thm}
\begin{proof}
The results follows from the arguments of Theorem \ref{MainThm1} together  with the results of Lemma \ref{GeneralGaugeConditions}. Notice that in this more general case the gauge condition $F^{\mu}\doteq\bar{g}^{\alpha\beta}\mathcal{S}^{\mu}_{\alpha\beta}=0$ is decomposed
\begin{align}\label{GeneralGaugeCondition}
\begin{split}
0=F^{0}
&=-N^{-3}(\partial_tN-X(N)) - N^{-1}\tau + \frac{N^{-2}}{2}( 2\mathrm{div}_{g}X -2\mathrm{div}_{e}X- \bar{g}^{ab}D_X\bar{g}_{ab} + X(|X|^2_e)),\\
0=F^{j}
&=-N^{-2}(\partial_tX^j - \frac{X^j}{2}\partial_t|X|^2_e + \frac{X^{j}}{2}X(|X|^2_e) - D_XX^{j} + N\bar{\nabla}^jN) + \bar{g}^{ab}\mathcal{S}^{j}_{ab}.
\end{split}
\end{align}
{Notice that from the second of the above equations, one gets
\begin{align*}
0&=N^{-2}(\overbrace{\underbrace{X^{\flat}_j\partial_tX^j}_{\frac{1}{2}\partial_t |X|^2_e} - \frac{|X|^2_e}{2}\partial_t|X|^2_e}^{\frac{1}{2}(1-|X|^2_{e})\partial_t|X|^2_e} + \frac{|X|^2_e}{2}X(|X|^2_e) - X^{\flat}_jD_XX^{j} +X^{\flat}_j N\bar{\nabla}^jN) - \bar{g}^{ab}X^{\flat}_j\mathcal{S}^{j}_{ab},
\end{align*}
which implies that, as long as $|X|_e\neq 1$,\footnote{{Since below we shall only be interested in an explicit expression for $\partial_t X^j$ near infinity, where $|X|_e \rightarrow 0$, this condition will be satisfied. Within the general Cauchy problem, this does not pose a relevant problem, since the metric $e$ is actually auxiliary. Thus, if $X\neq 0$, one can chose an equivalent metric given by $e'= \frac{1}{2 \max |X|_e} e,$ which guarantees that $|X|_{e'} <1$ over all of $M$ over all of $M$.}
}
\begin{align*}
\partial_t|X|^2_e&=\frac{2}{|X|^2_e-1}\big( \frac{|X|^2_e}{2}X(|X|^2_e) - \frac{1}{2}D_X|X|^2_e +X^{\flat}_j N\bar{\nabla}^jN   - N^2\bar{g}^{ab}X^{\flat}_j\mathcal{S}^{j}_{ab}\big).
\end{align*}
}
Under our decaying assumptions, this implies that
\begin{align*}
\partial_t|X|^2_e=O_3(|x|^{-2\tau_*-1})
\end{align*}
where $\tau_{*}\doteq \min\{\tau,\rho\}>0$.Therefore, from (\ref{GeneralGaugeCondition}), one gets
\begin{align*}
\begin{split}
\partial_tN&= - N^{2}\tau +X(N) + \frac{N}{2}(  2\mathrm{div}_{g}X -2\mathrm{div}_{e}X- \bar{g}^{ab}D_X\bar{g}_{ab} + X(|X|^2_e))=\pm 3c + O_{3}(|x|^{-\tau_{*}-1}),\\
\partial_tX^j&= \frac{X^j}{2}\partial_t|X|^2_e - \frac{X^{j}}{2}X(|X|^2_e) + D_XX^{j} - N\bar{\nabla}^jN + N^2\bar{g}^{ab}\mathcal{S}^{j}_{ab}=O_{3}(|x|^{-\tau_{*}-1}).
\end{split}
\end{align*}
 Thus, since $\partial^3N^2$ and $\partial^2\dot{X}$ are both $O_1(|x|^{-\tau_{*}-3})$, one sees that these more general choice of asymptotic observers still gives us the following non-zero contribution for the energy
\begin{align}\label{InvariantEnergyContribution}
\mathcal{E}_{\alpha,\beta}(\bar{g})&=\lim_{r\rightarrow\infty}\Big\{ \frac{\beta}{2}\int_{S^{2}_r}\left( \partial_{i}\ddot{g}_{ji} - \partial_{j}\ddot{g}_{ii} \right)\hat{\nu}^{j}d\omega_r  - ( \beta  + 2\alpha)\int_{S^{2}_r}\partial_{j}\ddot{g}_{ii}\hat{\nu}^{j}d\omega_r  \Big\}
\end{align} 
Also, in the asymptotic region, we have that 
\begin{align*}
\bar{g}^{\lambda\mu}\hat{D}_{\lambda}\hat{D}_{\mu}\bar{g}_{\alpha\beta}&=-N^{-2}\hat{D}_{0}\hat{D}_{0}\bar{g}_{\alpha\beta} + \bar{g}^{ab}\hat{D}_{a}\hat{D}_{b}\bar{g}_{\alpha\beta},\\
&=-N^{-2}\hat{D}_{0}\hat{D}_{\partial_t}\bar{g}_{\alpha\beta} + N^{-2}\hat{D}_{0}\hat{D}_{X}\bar{g}_{\alpha\beta} + \bar{g}^{ab}\hat{D}_{a}\hat{D}_{b}\bar{g}_{\alpha\beta},\\
&=-N^{-2}\hat{D}_{t}\hat{D}_{\partial_t}\bar{g}_{\alpha\beta} + N^{-2}\hat{D}_{X}\hat{D}_{\partial_t}\bar{g}_{\alpha\beta} + N^{-2}\hat{D}_{\partial_t}\hat{D}_{X}\bar{g}_{\alpha\beta} - N^{-2}\hat{D}_{X}\hat{D}_{X}\bar{g}_{\alpha\beta} + \bar{g}^{ab}\hat{D}_{a}\hat{D}_{b}\bar{g}_{\alpha\beta},\\
&=-N^{-2}\partial^2_{t}\bar{g}_{\alpha\beta} + N^{-2}X^j\partial_{j}\partial_t\bar{g}_{\alpha\beta} + N^{-2}\partial_t(X^j\partial_j\bar{g}_{\alpha\beta}) - N^{-2}X^j\partial_{j}(X^k\partial_{k}\bar{g}_{\alpha\beta}) + \bar{g}^{ab}\partial_{a}\partial_{b}\bar{g}_{\alpha\beta}
\end{align*}
thus
\begin{align*}
\bar{g}^{\lambda\mu}\hat{D}_{\lambda}\hat{D}_{\mu}\bar{g}_{ij}\vert_{t=0}&=-N^{-2}\partial^2_{t}\bar{g}_{ij}\vert_{t=0} + \underbrace{\bar{g}^{ab}\partial_{a}\partial_{b}\bar{g}_{ij}}_{=O_2(|x|^{-\tau-2})} + \underbrace{N^{-2}X^j\partial_{j}\partial_t\bar{g}_{ij}}_{=O(|x|^{-(\rho+\tau+1)})} + N^{-2}( \underbrace{\partial_tX^j\partial_j\bar{g}_{ij}}_{=O_2(|x|^{-(\tau+\tau_{*}+2)})} + \underbrace{X^j\partial_j\partial_t\bar{g}_{ij})}_{=O_2(|x|^{-(\tau+\rho+1)})} ) \\
&- N^{-2}\underbrace{X^j(\partial_{j}X^k\partial_{k}\bar{g}_{ij} + X^k\partial_{j}\partial_{k}\bar{g}_{ij})}_{=O_2(|x|^{-(2\rho+\tau+2)})},\\
&=-N^{-2}\partial^2_{t}\bar{g}_{ij}\vert_{t=0} + O_2(|x|^{-\tau-2})+O_2(|x|^{-\rho-\tau-1}) \\
&=-N^{-2}\partial^2_{t}\bar{g}_{ij}\vert_{t=0} + O_2(|x|^{-\tau-2})+O_2(|x|^{-2\tau_*-1}).
\end{align*}


Putting this together with (\ref{gauge0}) and the Einstein hypothesis on $\bar{g}$, we now get
\begin{align}\label{InvarianceUmbilicalEnergy.0}
\begin{split}
\!\!\!\!\frac{1}{2}\partial^2_{t}\bar{g}_{ij}|_{t=0} &=\Big( \Lambda\bar{g}_{ij} + \frac{1}{2}\big(  \partial_{j}\bar{g}^{\nu\lambda}\partial_{\nu}\bar{g}_{i \lambda}  + \partial_{i}\bar{g}^{\nu\lambda}\partial_{\nu}\bar{g}_{j\lambda} \big) + \bar{\Gamma}^{\nu}_{j\sigma}\bar{\Gamma}^{\sigma}_{i\nu}\Big)\Big|_{t=0} + O_2(|x|^{-\tau-2}) +O_2(|x|^{-2\tau_*-1}).
\end{split}
\end{align}
We can also compute the quadratic term (the details will be omitted for conciseness):
\begin{align*}
\bar{\Gamma}^{\nu}_{j\sigma}\bar{\Gamma}^{\sigma}_{i\nu}\Big|_{t=0}&=\frac{1}{4}\bar{g}^{\nu\mu}\bar{g}^{\sigma\gamma}
(\partial_{j}\bar{g}_{\mu\sigma} + \partial_{\sigma}\bar{g}_{\mu j} - \partial_{\mu}\bar{g}_{j\sigma} )(\partial_{i}\bar{g}_{\gamma\nu} + \partial_{\nu}\bar{g}_{\gamma i} - \partial_{\gamma}\bar{g}_{i\nu} )\Big|_{t=0},\\
&=\frac{1}{2}\bar{g}^{ab}\dot{g}_{ja}\dot{g}_{b i} + O_3(|x|^{-\rho-1}) + O_3(|x|^{-2\tau_{*}}).
\end{align*}
where we used since $\bar{g}_{\alpha\beta}|_{t=0}=\eta_{\alpha\beta} + O_4(|x|^{-\tau_{*}})$, $\tau_{*}>0$, then $\bar{g}^{\alpha\beta}=\eta^{\alpha\beta}+O_4(|x|^{-\tau_{*}})$. Recalling that $\dot{g}=\pm 2cg+O_2(|x|^{-(\tau+1)})$, we now find 
\begin{align*}
\bar{\Gamma}^{\nu}_{j\sigma}\bar{\Gamma}^{\sigma}_{i\nu}\Big|_{t=0}&= 2c^2 g^{ab} g_{a j}  g_{ib} + O_3(|x|^{-\rho-1}) + O_3(|x|^{-2\tau_{*}})=\frac{2\Lambda}{3}g_{ij} + O_3(|x|^{-\rho-1}) + O_3(|x|^{-2\tau_{*}}).
\end{align*}
With similar arguments, we can also compute that
\begin{align*}
\big( \partial_{j}\bar{g}^{\nu\lambda}\partial_{\nu}\bar{g}_{i \lambda}  + \partial_{i}\bar{g}^{\nu\lambda}\partial_{\nu}\bar{g}_{j\lambda} \big)|_{t=0}&=\big( \partial_{j}\bar{g}^{t\lambda}\partial_{t}\bar{g}_{i \lambda} + \partial_{j}\bar{g}^{a\lambda}\partial_{a}\bar{g}_{i \lambda}  + \partial_{i}\bar{g}^{t\lambda}\partial_{t}\bar{g}_{j\lambda} + \partial_{i}\bar{g}^{a\lambda}\partial_{a}\bar{g}_{j\lambda} \big)|_{t=0},\\
&=O_3(|x|^{-\tau_{*}-1}).
\end{align*}
Putting all the above together, we find that
\begin{align*}
\begin{split}
\frac{1}{2}\partial^2_{t}\bar{g}_{ij}|_{t=0} &= \Lambda g_{ij} + \frac{2\Lambda}{3}g_{ij} + O_2(|x|^{-\tau_{*}-1}) + O_2(|x|^{-2\tau_{*}}) +O_2(|x|^{-2\tau_*-1}),\\
&=\frac{5}{3}\Lambda g_{ij} + O_2(|x|^{-\tau_{*}-1}) + O_2(|x|^{-2\tau_{*}}).
\end{split}
\end{align*}
Thus, the following holds
\begin{align*}
\frac{1}{2}\partial_k\ddot{g}_{ij}=\frac{5}{3}\Lambda\partial_kg_{ij} + O_1(|x|^{-\tau_{*}-2}) + O_1(|x|^{-2\tau_{*}-1}).
\end{align*}
Plugging this information in (\ref{InvarianceUmbilicalEnergy.0}) one obtain
\begin{align*}
&\int_{S^{2}_r}\left( \partial_{i}\ddot{g}_{ji} - \partial_{j}\ddot{g}_{ii} \right)\hat{\nu}^{j}d\omega_r - ( \beta  + 2\alpha)\int_{S^{2}_r}\partial_{j}\ddot{g}_{ii}\hat{\nu}^{j}d\omega_r=\\
&\frac{5}{3}\Lambda\Big\{\beta\int_{S^{2}_r}\left( \partial_{i}g_{ji} - \partial_{j}g_{ii} \right)\hat{\nu}^{j}d\omega_r - 2( \beta  + 2\alpha)\int_{S^{2}_r}\partial_{j}g_{ii}\hat{\nu}^{j}d\omega_r  \Big\} + O_2(r^{-\tau_{*}}) + O_2(r^{-2\tau_{*}+1}).
\end{align*}
Above, the $O_2(r^{-\tau_{*}})$ term always vanishes in the limit $r\rightarrow\infty$, while the term of order $O_2(r^{-2\tau_{*}+1})$ also vanishes under the condition $\tau_{*}=\min\{\tau,\rho\}>\frac{1}{2}$. Therefore, we find
\begin{align*}
\begin{split}
\mathcal{E}^{\mathcal{O}}_{\alpha,\beta}(\bar{g})
&=\frac{5}{3}\Lambda\lim_{r\rightarrow\infty}\Big\{\beta\int_{S^{n-1}_r}\left( \partial_{i}g_{ji} - \partial_{j}g_{ii} \right)\hat{\nu}^{j}d\omega_r - 2( \beta  + 2\alpha)\int_{S^{n-1}_r}\partial_{j}g
_{ii}\hat{\nu}^{j}d\omega_r  \Big\}=\mathcal{E}^{\mathcal{O}_0}_{\alpha,\beta}(\bar{g})
,
\end{split}
\end{align*}
which shows that under our hypotheses the value of the energy is independent of the chosen observers $\mathcal{O}$.

Since in all the cases we treat $\mathcal{E}_{\alpha,\beta}(\bar{g})$ is proportional to the ADM energy of $g$, if $\phi_x,\phi_y$ are two asymptotic charts with coordinates $\{x^i\}_{i=1}^{3}$ and $\{y^i\}_{i=1}^{3}$ respectively of decay rates $\tau_x,\tau_y>\frac{1}{2}$, under the condition $R_g\in L^1(M,dV_g)$ the equality $\mathcal{E}^{\phi_x}_{\alpha,\beta}(\bar{g})=\mathcal{E}^{\phi_y}_{\alpha,\beta}(\bar{g})$ follows for the equality $E^{\phi_x}_{ADM}(g)=E^{\phi_y}_{ADM}(g)$, which itself is a consequence of \cite[Theorem 4.2]{Bartnik}.

\end{proof}

\subsection{{Conformally Einstein solutions}}

In this section we would like to consider space-time fourth order solutions, that is $A_{\bar{g}}=0$, which are conformally Einstein. {In particular, let us notice that conformally Einstein metrics are Bach-flat, and therefore they correspond to solutions to (vacuum) conformal gravity, as described in Section \ref{SectionConfGrav}. Such solutions represent a borderline, being tightly related to second order solutions.} 
{Nevertheless, the conformal symmetry of the equations imposes too much asymptotic freedom. That is, if we want to impose a particular kind of asymptotic behaviour for the metric, such behaviour cannot be in general conformally invariant and therefore one needs to concentrate on a subclass of conformal transformations. With this in mind, let us consider the following kinds of space-times.}

{\begin{coro}\label{BachFlatCoro}
Let $(V^4=M^3\times \mathbb{R},\bar{g})$ be an AM space-time which is conformal to an Einstein space-time $(V^4=M^3\times \mathbb{R},\hat{\bar{g}})$, where the latter is generated by $\Lambda$AE initial data $\mathcal{J}\doteq (M^3,g,K)$ of order $\tau>0$ with respect to an asymptotic chart $\{x^i\}_{i=1}^3$. If $\mathcal{J}$ satisfies the hypotheses of Theorem \ref{InvarianceUmbilicalEnergyTHM} and $\Omega$ stands for the conformal factor, $\Omega:M\times \mathbb{R}\mapsto \mathbb{R}_{+}$, $\bar{g}=\Omega^2\hat{\bar{g}}$, and $\Omega=1+o_4(|x|^{-\sigma})$,\footnote{{In this case, in the $o_4(|x|^{-\sigma})$ hypothesis we impose on $\Omega$ implies that time-derivatives also increase the decay by one order.}} with $\sigma\geq \max\{1,\tau\}$, then $\mathcal{E}_{\alpha,\beta}(\bar{g})=\mathcal{E}_{\alpha,\beta}(\hat{\bar{g}})$, which is well-defined, and independent both of the asymptotic observers and asymptotic charts under the same conditions of Theorem \ref{InvarianceUmbilicalEnergyTHM}.
\end{coro}}
{\begin{remark}
Notice that, although the above result is stated for a general conformally Einstein metric and general $\alpha,\beta$, the physically motivated examples arise for the special case of Bach flat metrics and $3\alpha+\beta=0$.
\end{remark}}
\begin{proof}
{Considering that $\Omega:M\times \mathbb{R}\mapsto \mathbb{R}_{+}$ satisfies our hypotheses, notice that
\begin{align*}
\bar{X}=\hat{\bar{X}},\:\: n&=N^{-1}E_0=\Omega^{-1}\hat{n}=\left(\Omega\hat{N}\right)^{-1}\hat{E}_0,
\end{align*}
where $\{E_{\alpha}\}$ and $\{\hat{E}_{\alpha}\}$ stand for the adapted orthogonal frames associated to $\bar{g}$ and $\hat{\bar{g}}$ respectively. From this we see that
\begin{align*}
&g=\Omega^2\hat{g},\; \; N=\Omega\hat{N},\;\; \partial_t\bar{X}|_{t=0}=\partial_t\hat{\bar{X}}|_{t=0},\\
&\ddot{g}
=\Omega^2\ddot{\hat{g}} + 4\Omega\dot{\Omega}\partial_t\hat{\bar{g}} + 2\dot{\Omega}^2\:\hat{g} + 2\Omega\ddot{\Omega}\:\hat{g} .
\end{align*}
Let us now appeal to the computations obtained in the proof of Theorem \ref{InvarianceUmbilicalEnergyTHM} to get
\begin{align*}
g_{ij}=\delta_{ij} + O_4(r^{-\tau}),\;\; N|_{t=0}=1 + O_4(r^{-\tau_{*}}), \;\; \dot{X}_j=O_{3}(r^{-(\tau_{*}+1)}),
\end{align*}
where $\tau_{*}\doteq\min\{\tau,\rho,\sigma\}$. Notice that this already implies that the energy will be given by the same expression as in (\ref{InvariantEnergyContribution}). Also,
\begin{align*}
\ddot{g}&=\Omega^2\ddot{\hat{g}} + 4\Omega\dot{\Omega}\dot{\hat{g}} + 2\dot{\Omega}^2\:\hat{g} + 2\Omega\ddot{\Omega}\:\hat{g}=\left(1 + o_4(r^{-\sigma}) \right)\ddot{\hat{g}} + o_4(r^{-(\sigma+1)})\dot{\hat{g}} + o_4(r^{-2(\sigma+1)})\:\hat{g} + o_4(r^{-(\sigma+2)}) \hat{g}.
\end{align*}
Now, appealing to the asymptotic analysis done for $\Lambda$-AE Einstein solutions, we know that
\begin{align*}
\frac{1}{2}\ddot{\hat{g}}_{ij}=\frac{5}{3}\Lambda \hat{g}_{ij} + O_2(r^{-(\tau_{*}+2)}),\;\;\dot{\hat{g}}_{ij}=\pm 2\left(\frac{\Lambda}{3}\right)^{\frac{1}{2}} \hat{g}_{ij} + O_3(r^{-\tau_{*}}),
\end{align*}
then, since $\hat{g}_{ij}=\delta_{ij}+O_4(r^{-\tau_{*}})$, we find
\begin{align*}
\begin{split}
\ddot{g}&=\frac{10}{3}\Lambda (1 + o_4(r^{-\sigma})) \hat{g}_{ij} + o_1(r^{-(\sigma+1)}) + O_2(r^{-(\tau_{*}+2)}).
\end{split}
\end{align*}
From which it also follows that
\begin{align*}
\partial_k\ddot{g}_{ij}&=\frac{10}{3}\Lambda(1 + o(r^{-\sigma}))\partial_k\hat{g}_{ij} + \frac{10}{3}\Lambda o(r^{-(\sigma+1)})\hat{g}_{ij} + o(r^{-(\sigma+2)})  + O(r^{-(\tau_{*}+3)}),\\
&=\frac{10}{3}\Lambda\partial_k\hat{g}_{ij} + o(r^{-(\sigma+1)})  + O(r^{-(\tau_{*}+3)}),
\end{align*}
which implies
\begin{align*}
\partial_{j}\ddot{g}_{ii}-\partial_i\ddot{g}_{ij}&=\frac{10}{3}\Lambda\left(\partial_j\hat{g}_{ii} - \partial_i\hat{g}_{ij} \right) + o(r^{-(\sigma+1)}) + O_1(r^{-(\tau_{*}+3)}).
\end{align*}
Therefore, we find
\begin{align*}
\int_{S^2_r}\left(\partial_{j}\ddot{g}_{ii}-\partial_i\ddot{g}_{ij}\right)\nu_jd\omega_r&=\frac{10}{3}\Lambda\int_{S^2_r}\left(\partial_{j}\hat{g}_{ii}-\partial_i\hat{g}_{ij}\right)\nu_jd\omega_r + \underbrace{o(r^{-(\sigma+1)})r^{2} + O(r^{-(\tau_{*}+3)})r^{2}}_{o(1)}.
\end{align*}
Thus, passing to the limit the last terms do not contribute, and 
we find 
\begin{align}\label{BachSolutions.1}
\begin{split}
\lim_{r\rightarrow\infty}\int_{S^2_r}\left(\partial_{j}\ddot{g}_{ii}-\partial_i\ddot{g}_{ij}\right)\nu_jd\omega_r&=\frac{10}{3}\Lambda\lim_{r\rightarrow\infty}\int_{S_r}\left(\partial_{j}\hat{g}_{ii}-\partial_i\hat{g}_{ij}\right)\nu_jd\omega_r \\
&=\lim_{r\rightarrow\infty}\int_{S^2_r}\left(\partial_{j}\ddot{\hat{g}}_{ii}-\partial_i\ddot{\hat{g}}_{ij}\right)\nu_jd\omega_r
\end{split}
\end{align}
Similarly, one finds that 
\begin{align}\label{BachSolutions.2}
\lim_{r\rightarrow\infty}\int_{S^2_r}\partial_{j}\ddot{g}_{ii}\nu^jd\omega_r=\lim_{r\rightarrow\infty}\int_{S^2_r}\partial_{j}\ddot{\hat{g}}_{ii}\hat{\nu}^jd\omega_r.
\end{align}
Putting together (\ref{BachSolutions.1})-(\ref{BachSolutions.2}) in (\ref{InvariantEnergyContribution}), we establish our claim.\qed
}
\end{proof}

{
\begin{remark}
Let us highlight that, although the Bach-flat solutions associated to the above corollary are related to the same kind of equations as those treated in Section \ref{sectionconfflat}, the kind of asymptotic behaviour imposed in the above corollary is notably different than that of the FMSK solutions associated to Theorem \ref{FSMKlemma}. Interestingly, in both cases we get that $\mathcal{E}_{\alpha,\beta}$ is related to a physically meaningful quantity.
\end{remark}}

\section{The intrinsic case}

Finally, we would like to merely introduce another limit in which the energy $\mathcal{E}_{\alpha,\beta}$ is particularly interesting from the geometric viewpoint. This limit is realized in a \textit{natural} situation, which is given by \textit{fourth order stationary solutions}. Borrowing partial terminology from relativity, we could define a (globally hyperbolic) stationary Lorentzian manifold to be a manifold $(M^n\times I,\bar{g})$, with $I\subset \mathbb{R}$, such that
\begin{align*}
\bar{g}=-N^2dt^2 + \tilde{g},
\end{align*}
where $\tilde{g}$ stands for a time-independent tensor field which restricts to the same Riemannian metric $g$ when applied to tangent vectors to $M$, so that $(M,g,N)$ is a Riemannian manifold, with metric $g$ and $N:M\mapsto \mathbb{R}$ is a positive function. In this case, the energy would look like
\begin{align*}
\begin{split}
\mathcal{E}_{\alpha,\beta}(\bar{g})&=\lim_{r\rightarrow\infty}\Big\{\left( \frac{3}{2}\beta  + 2\alpha\right)\int_{S^{n-1}_r} \left( \partial_{j}\partial_{i}\partial_{i}g_{aa} - \partial_{j}\partial_{u}\partial_{i}g_{u i}\right)\hat{\nu}^{j}d\omega_{r}   \\
&+ ( \beta  + 2\alpha)\int_{S^{n-1}_r}\partial_{j}\partial_{i}\partial_{i}N^2 \hat{\nu}^{j}d\omega_r  \Big\}
\end{split}.
\end{align*}
In particular, for those theories parametrized by $2\alpha+\beta=0$ we get a purely Riemannian situation, given by
\begin{align}\label{4thenergy-static.2}
\begin{split}
\mathcal{E}_{\alpha}(g)\doteq \mathcal{E}_{\alpha,-2\alpha}(\bar{g})&=- \alpha\lim_{r\rightarrow\infty}   \int_{S^{n-1}_r} \left( \partial_{j}\partial_{i}\partial_{i}g_{aa} - \partial_{j}\partial_{u}\partial_{i}g_{u i}\right)\hat{\nu}^{j}d\omega_{r}   \end{split}.
\end{align}
The above relation is not only a suggestive higher-order generalization of the ADM energy from GR, but it strongly related to $Q$-curvature analysis and in particular to positive mass theorems for the Paneitz operator. These topics have received plenty of attention in geometric analysis (see, for instance, \cite{Malchiodi1}, or \cite{HangYang1}), and due to delicate analysis related with these problems, the analysis of (\ref{4thenergy-static.2}) will be done separately. Nevertheless, we would like to highlight that (\ref{4thenergy-static.2}) carries a powerful positive energy theorem, which implies the know positive mass theorems for the Paneitz operator. Furthermore, it creates a bridge between $Q$-curvature analysis and positive energy theorem for these fourth-order gravitational theories, which seems to parallel the relation between scalar curvature analysis and positive energy theorems in GR (see \cite{avalos2021positive}). We expect this relation to bring about strong results in geometry and analysis.

\section{Appendix}
\subsection{Geometric Conventions}
\label{subsectionconvention}
\subsubsection{Curvature conventions}
To avoid any ambiguity let us pinpoint the curvature conventions we follow in this text: 
 the curvature tensor is defined as:
\begin{align*}
R(X,Y)Z=\nabla_X\nabla_YZ - \nabla_Y\nabla_XZ - \nabla_{[X,Y]}Z,
\end{align*}
with its components ordered as follows:
\begin{align*}
\begin{split}
R^{i}_{jkl}&=dx^{i}(R(\partial_k,\partial_l)\partial_j),\\
&=\partial_k\Gamma^{i}_{lj}-\partial_l\Gamma^{i}_{kj} +  \Gamma^{i}_{ku}\Gamma^u_{jl}- \Gamma^{i}_{lu}\Gamma^u_{jk}.\\
\end{split}
\end{align*}
From this we  get the canonical Ricci, Einstein and scalar tensors:
\begin{align*}
\mathrm{Ric}_{ij}\doteq R^l_{ilj}.
\end{align*}

\subsubsection{Differential forms}\label{DiffFormsConventions}

In this section we will establish our conventions concerning differential forms and operations related to them. First of all, given a $k$-form $\omega\in \Omega^k(V)$ on a $d$-dimensional manifold $V$ and a local coordinate system $\{x^i\}_{i=1}^d$, we fix $\omega_{i_1\cdots i_k}\doteq \omega(\partial_{i_1},\cdots,\partial_{i_k})$ as its components relative to the basis $\{dx^{j_1}\wedge\cdots dx^{j_k}\}_{j_1<\cdots<j_k}$, and therefore we may locally write
\begin{align*}
\omega=\sum_{i_1<\cdots<i_{k}}\omega_{i_1\cdots i_k}dx^{i_1}\wedge\cdots dx^{i_k}=\frac{1}{k!}\omega_{i_1\cdots i_k}dx^{i_1}\wedge\cdots dx^{i_k},
\end{align*}
where in the last equality the summation is not restricted to $i_1<\cdots<i_k$. In this setting we have several well-known operations. To start with, given a semi-Riemannian manifold $(V^d,g)$ we have an induced semi-Riemannian metric $g^{(k)}$ on each space $\Omega^k(V)$ which is given by
\begin{align*}
g^{(k)}(\alpha,\beta)\doteq \sum_{i_1<\cdots <i_k}\alpha_{i_1\cdots i_k}\beta^{i_1\cdots i_k}=\frac{1}{k!}\alpha_{i_1\cdots i_k}\beta^{i_1\cdots i_k}, \text{ for any } \alpha,\beta\in \Omega^k(V)
\end{align*}
where above $\beta^{i_1\cdots i_k}\doteq g^{i_1j_1}\cdots g^{i_kj_k}\beta_{j_1\cdots j_k}$. It is not difficult to see that if $\{e_1,\cdots,e_d\}$ is a $g$-orthonormal basis for $T_pV$ at a point $p\in V$, and if $\{e^1,\cdots,e^d\}$ stands for its dual basis, then $\{e^i_{1}\wedge\cdots\wedge e^{i_k}\}_{i_1<\cdots<i_k}$ is a $g^{(k)}$-orthonormal basis for the fibre of $\Omega^k(V)$ over $p$ (see, for instance, \cite[Proposition 7.2.11]{AMR}). In particular, it holds that
\begin{align*}
g^{(k)}(e^i_{1}\wedge\cdots\wedge e^{i_k},e^i_{1}\wedge\cdots\wedge e^{i_k})=c_{i_1}\cdots c_{i_k}
\end{align*}
where above $c_{i_{j}}\doteq g(e_{i_j},e_{i_j})=\pm 1$.

Also, in this setting, we introduce the Hodge-star operator, which is defined point wise as a linear operator on each fiber. For this one need to introduce a volume form. Thus, let us again consider an orientable semi-Riemannian manifold $(V^d,g)$ and denote the associated Riemannian volume form by $dV_g$, which locally reads as $dV_g=\sqrt{|\mathrm{det}(g)|}dx^1\wedge\cdots\wedge dx^d$. Given our vector bundle $\Omega^k(V)\xrightarrow[]{\pi} V$, one defines a linear operator
\begin{align*}
\star_{g}:\Omega^{k}(V)\mapsto \Omega^{d-k}(V)
\end{align*}
by the requirement that for all $\alpha,\beta\in \Omega^{k}(V)$ it holds that
\begin{align*}
\star_{g}\alpha=g^{(k)}(\beta,\star_g \alpha)dV_g.
\end{align*}
Above, we add the subscript $g$ on $\star$ to highlight the dependence on this operator on the choice of metric $g$. This can be seen to be well-defined and, in fact, in local (oriented) coordinates one can see that the action of $\star$ on $\alpha\in \Omega^k(V)$ is given by (see, for instance, \cite[Proposition 7.2.12 and Example 7.2.14(D)]{AMR}):
\begin{align*}
\begin{split}
\star_g\alpha&=\sum_{\underset{j_{k+1}<\cdots<j_d}{j_1<\cdots<j_k}}\alpha^{j_1\cdots j_k}\epsilon_{j_1\cdots j_k j_{k+1}\cdots j_d}\sqrt{|\mathrm{det}(g)|}dx^{j_{k+1}}\wedge\cdots\wedge dx^{j_d},\\
(\star_g\alpha)_{j_{k+1}\cdots j_{d}}&=\sum_{j_1<\cdots<j_k}\alpha^{j_1\cdots j_k}\epsilon_{j_1\cdots j_k j_{k+1}\cdots j_d}\sqrt{|\mathrm{det}(g)|},
\end{split}
\end{align*}
where above $\epsilon_{j_1\cdots j_d}$ stands for the antisymmetric Levi-Civita symbol and the sums which appeal to Einstein summation convention are understood as over all possible values of the indices. Sometimes, we shall write $\mu_{j_1\cdots j_d}=\epsilon_{j_1\cdots j_d}\sqrt{|\mathrm{det}(g)|}$ which is defined via the relation
\begin{align*}
dV_g=\mu_{j_1\cdots j_d}dx^{j_1}\wedge\cdots\wedge dx^{j_d} \text{ (no summation) }.
\end{align*}

Let us now recall that, given $\alpha\in \Omega^k(V)$, the exterior differential $d:\Omega^k(V)\mapsto \Omega^{k+1}(V)$ is characterised by its local action given by
\begin{align}
\label{differentialformnotation}
d\alpha=\sum_{i_1<\cdots<i_{k}}\partial_{i}\alpha_{i\cdots i_{k+1}}dx^{i}\wedge dx^{i_1}\wedge\cdots\wedge dx^{k}.
\end{align}
In particular, given a $1$-form $\alpha\in \Omega^1(V)$, we see that 
\begin{align*}
d\alpha&=\sum_{j=1}^d\partial_{i}\alpha_{j}dx^{i}\wedge dx^{j}=\partial_{i}\alpha_{j}dx^{i}\wedge dx^{j}=\sum_{i<j}(\partial_i\alpha_j-\partial_j\alpha_i)dx^i\wedge dx^j,\\
d\alpha_{ij}&=\partial_i\alpha_j-\partial_j\alpha_i.
\end{align*} 

Now, given a semi-Riemannian manifold $(V^d,g)$ we can consider the formal adjoint $d^{*}:\Omega^{k+1}(V)\mapsto \Omega^{k}(V)$, which we shall denote by $\delta_{g}$ and where we shall highlight its dependence on $g$, which arises through the canonical choice of $dV_g$ as the volume form. One can thus globally write
\begin{align*}
\delta_g\alpha=(-1)^{dk+1+\mathrm{Ind}(g)}\star_gd\star_g\alpha \text{ for all } \alpha\in\Omega^{k+1}(V),
\end{align*}
where $\mathrm{Ind}(g)$ denotes the \emph{index} of the semi-Riemannian metric $g$. For instance, if $g$ is Lorentzian, we have $\mathrm{Ind}(g)=1$. One can then compute that, in local coordinates, the following formula holds:
\begin{align*}
\delta_g\alpha_{i_1\cdots i_k}=-\nabla^i\alpha_{ii_1\cdots i_k},
\end{align*}
where above $\nabla$ stands for the Riemannian connection associated to $g$.

Also, given $X\in \Gamma(TV)$, let us introduce the \emph{interior product} $X\lrcorner:\Omega^k(V)\mapsto \Omega^{k-1}(TV)$ on any manifold $V^d$, which is defined by
\begin{align*}
X\lrcorner\alpha(X_1,\cdots,X_{k-1})\doteq \alpha(X,X_1,\cdots,X_{k-1}) \text{ for any } \alpha\in \Omega(V) \text{ and } \forall \:\: X_1,\cdots,X_{k-1}\in \Gamma(TV).
\end{align*}

An important formula linking the exterior derivative, the interior product and the Lie derivative is given by Cartan's famous \emph{magic} formula:
\begin{align*}
\pounds_X\alpha=d(X\lrcorner\alpha) + X\lrcorner(d\alpha) \:\: \forall \:\: \alpha\in \Omega^k(V) \text{ and } X\in \Gamma(TV).
\end{align*}

The above formula plays an important role in the application of Stokes' theorem to operators in divergence form. Since in the core of this paper we will be interested in certain flux formalae which are derived in this manner within the Lorentzian setting, let us below briefly highlight a few differences with the more usual Riemannian setting. 

Let $(V^{n+1}=M^n\times \mathbb{R}, \bar{g})$ be a Lorentzian manifold, parametrise the $\mathbb{R}$ factor with a coordinate $t$. Assume $\partial_t$ is time-like and denote by $g_t$ the induced Riemannian metric on each $M_t\doteq M\times \{t\}$. Assume furthermore that $(M^n,g_t)$ are orientable Riemmanian manifolds and denote by $dV_{g_t}$ their corresponding volume forms. With all this, we have a natural orientation for $V$: if at $p\in M$ $\{e_1,\cdots,e_n\}$ denotes a a positive basis for $T_pM$, then $\{\partial_t,e_1,\cdots,e_n\}$ denotes a positive basis for $T_{(p,t)}V$. Thus, if $\{x^i\}_{i=1}^n$ is a positively oriented coordinate system for $M$, then $dV_{\bar{g}}=\sqrt{|\mathrm{det}(\bar{g})|}dt\wedge dx^1\wedge\cdots\wedge dx^n$ denotes our volume form.

Let $\Omega\subset M$ be a compact subset with smooth boundary and define the subset $\mathcal{C}_T=\Omega\times [0,T]$. This subset is \emph{Stokes regular}, in the sense that it regular enough so as to apply Stokes' theorem over it. Now, let $X\in \Gamma(TV)$ and from the formula 
\begin{align*}
L_XdV_{\bar{g}}=\mathrm{div}_{\bar{g}}XdV_{\bar{g}}=d(X\lrcorner dV_{\bar{g}})
\end{align*}
one gets that
\begin{align}\label{Flux.0}
\int_{\mathcal{C}_T}\mathrm{div}_{\bar{g}}XdV_{\bar{g}}&=\int_{\partial \mathcal{C}_T}\mathcal{J}^{*}(X\lrcorner dV_{\bar{g}})
\end{align}
where above $\mathcal{J}^{*}:\partial\mathcal{C}_T\mapsto \mathcal{C}_T$ denotes the inclusion. We can split $\partial\mathcal{C}_{T}=\Omega_0\cup\Omega_T\cup L$, where $\Omega_0=\Omega\times \{0\}$, $\Omega_T=\Omega\times \{T\}$ and $L=\partial\Omega\times [0,T]$. On each of these hypersurfaces we denote the inclusion into $V$ by $\mathcal{J}^{*}_0,\mathcal{J}^{*}_T$ and $\mathcal{J}^{*}_L$ respectively. Now, let $n$ denote the future-pointing unit normal to each $t$-constant hypersurface $M_t=M\times\{t\}$. Then, writing $X=-\bar{g}(X,n)n + X^{\top}$, we find 
\begin{align*}
\mathcal{J}^{*}_0(X\lrcorner dV_{\bar{g}})=\mathcal{J}^{*}_0(dV_{\bar{g}}(-\bar{g}(X,n)n + X^{\top},\cdot))=-\bar{g}(X,n)\mathcal{J}^{*}_0(dV_{\bar{g}}(n,\cdot)).
\end{align*}
Notice that in (\ref{Flux.0}) $\Omega_0$ is oriented with its Stokes induced orientation, where the \emph{outward-pointing} unit normal corresponds to $-n$ and thus $\{e_1,\cdots,e_n\}$ is a positive basis for $\Omega_0$ at $p$ iff $\{-n,e_1,\cdots,e_n\}$ is positive for $\mathcal{C}_T$. This implies that the induced Stokes orientation for $\Omega_0$ is actually opposite to its intrinsic orientation. On the other hand, we see that in the case of $\Omega_T$ these two orientations agree. All this implies that
\begin{align*}
\mathcal{J}^{*}_0(X\lrcorner dV_{\bar{g}})=-\bar{g}(X,n)dV_{\bar{g}_0} \: ,\: \mathcal{J}^{*}_T(X\lrcorner dV_{\bar{g}})=-\bar{g}(X,n)dV_{\bar{g}_T}
\end{align*}
and, using intrinsic orientations for $\Omega$,
\begin{align}\label{Flux.1}
\begin{split}
\int_{\mathcal{C}_T}\mathrm{div}_{\bar{g}}XdV_{\bar{g}}&=-\int_{\Omega_0}\mathcal{J}^{*}_0(X\lrcorner dV_{\bar{g}}) + \int_{\Omega_T}\mathcal{J}^{*}_T(X\lrcorner dV_{\bar{g}}) + \int_{L}\mathcal{J}^{*}_L(X\lrcorner dV_{\bar{g}}),\\
&=\int_{\Omega_0}\bar{g}(X,n)dV_{\bar{g}_0} - \int_{\Omega_T}\bar{g}(X,n)dV_{\bar{g}_T} + \int_{L}\mathcal{J}^{*}_L(X\lrcorner dV_{\bar{g}}),
\end{split}
\end{align}
where ${J}^{*}_L(X\lrcorner dV_{\bar{g}})=g_t(X,\nu)\nu\lrcorner dV_{\bar{g}}=g_t(X,\nu)dL$, with $\nu$ the  outward pointing unit normal  vector field to $L$, which is to be understood with its induced Stokes orientation.

\subsubsection*{Extrinsic geometry}\label{ExtGeomConv}

In this section we shall quickly fix our conventions for the extrinsic curvature. Thus, let $M^n\hookrightarrow (V^{n+1},\bar{g})$ be an immersed hypersurface in a time-oriented Lorentzian manifold. We define the second fundamental form of $M$ as
\begin{align*}
\begin{split}
\mathbb{II}:\Gamma(TM)\times \Gamma(TM)&\mapsto \Gamma(TM^{\perp})\\
(X,Y)&\mapsto (\bar{\nabla}_{\bar{X}}\bar{Y})^{\perp}
\end{split}
\end{align*}
where $\bar{X}$ and $\bar{Y}$ denote arbitrary extensions (respectively) of $X$ and $Y$ to $V$, $\bar{\nabla}$ denotes the Riemannian connection associated to $\bar{g}$ and $TM^{\perp}$ denotes the normal bundle of $M$. Associated to the second fundamental form, we have the extrinsic curvature, here denoted by $K\in \Gamma(T^0_2M)$, which we define with respect to the \emph{future-pointing} unit normal to $M$. Thus, $K$ is defined by
\begin{align*}
K(X,Y)\doteq \bar{g}(\mathbb{II}(\bar{X},\bar{Y}),n)=\bar{g}(\bar{\nabla}_{\bar{X}}\bar{Y},n), \:\: \forall \: X,Y\in\Gamma(TM).
\end{align*}

Also, we define $\tau\doteq \mathrm{tr}_gK$ as the (not normalised) mean curvature of the immersion, and therefore we find that
\begin{align*}
\tau=-\mathrm{div}_{\bar{g}}n.
\end{align*}

Finally, let us notice that if $V^{n+1}=M^n\times\mathbb{R}$, with $\mathbb{R}$ parametrised by a coordinate $t$ and the time orientation given by $\partial_t$, then defining the associated lapse $N$ and shift $X$ by
\begin{align*}
N=-\bar{g}(\partial_t,n),\:\: X=\partial_t-Nn.
\end{align*}
we may write $n=\frac{1}{N}(\partial_t-X)$. This implies that
\begin{align*}
-K(U,W)&=\bar{g}(\bar{U},\bar{\nabla}_{\bar{W}}n)
=\frac{1}{N}\left(\bar{g}(\bar{U},\bar{\nabla}_{\bar{W}}\partial_t) - \bar{g}(\bar{U},\bar{\nabla}_{\bar{W}}X)\right),\\
&=\frac{1}{N}\left(\partial_t(\bar{g}(\bar{U},\bar{W})) - \bar{g}(\bar{\nabla}_{\partial_t}\bar{U},\bar{W}) + \bar{g}(\bar{U},[\bar{W},\partial_t]) - g(U,\nabla_{W}X)\right)
\end{align*}
Therefore,
\begin{align*}
-2NK(U,W)
&=\partial_t(\bar{g}(\bar{U},\bar{W})) - \bar{g}(\bar{U},[\partial_t,\bar{W}]) - \bar{g}(\bar{W},[\partial_t,\bar{U}]) - \pounds_{X}g(U,W).
\end{align*}
That is, 
\begin{align*}
K(U,W)=-\frac{1}{2N}\left( \partial_t(\bar{g}(\bar{U},\bar{W})) - \bar{g}(\bar{U},[\partial_t,\bar{W}]) - \bar{g}(\bar{W},[\partial_t,\bar{U}]) - \pounds_{X}g(U,W)\right)
\end{align*}
Notice that locally this reduces to a simple expression given by
\begin{align*}
K_{ij}=-\frac{1}{2N}\left( \partial_t\bar{g}_{ij} - \pounds_{X}g_{ij}\right),
\end{align*}
from which we sometimes write $K=-\frac{1}{2N}\left( \partial_tg - \pounds_{X}g_{ij}  \right)$.

\subsection{Proof of proposition \ref{solscharzschgener}}
Since the sketch presented has already dealt with the $2 \alpha + \beta = 0$ case, we will assume in the following that $\chi:= 2 \alpha + \beta \neq 0$. Employing the same Maple procedure as presented in figure  \ref{FequationM},  we can see (see figure \ref{shapeofM}) that, as announced in the sketch of the proof: $M= -m -\frac{\Lambda}{3}r^3 +C_1 r^{f(\alpha,\beta)} + C_2 r^{g(\alpha,\beta)}$, with
$$
\begin{aligned}
f&= \frac{6 \alpha +3 \beta + \sqrt{100 \alpha^2 + 84 \alpha \beta + 17 \beta^2}}{2\left( 2 \alpha + \beta\right)} \\
g&= \frac{6 \alpha +3 \beta - \sqrt{100 \alpha^2 + 84 \alpha \beta + 17 \beta^2}}{2\left( 2 \alpha + \beta\right)}
\end{aligned}
$$
and 
$$\begin{aligned}
\frac{\left(2 \alpha + \beta \right)^2 r }{24 \left( \beta + 4 \alpha \right) }A_{22}&= C_1\left( h^+_1 r^{t_1^+(\alpha, \beta) }+h^+_2r^{t_2^+(\alpha, \beta) }+h^+_3 r^{t_3^+(\alpha, \beta) } + h^+_4 r^{t_4^+(\alpha, \beta) } \right) \\&+  C_2\left( h^-_1 r^{t_1^-(\alpha, \beta) }+h^-_2 r^{t_2^-(\alpha, \beta) }+h^-_3 r^{t_3^-(\alpha, \beta) }+ h^-_4 r^{t_4^-(\alpha, \beta) } \right),
\end{aligned}
$$

with

\begin{equation}
\label{theti}
\begin{aligned}
t_1^-&=\frac{18 \alpha +9 \beta - \sqrt{100 \alpha^2 + 84 \alpha \beta + 17 \beta^2}}{4 \alpha + 2 \beta}\\
t_2^-&=\frac{10 \alpha +5 \beta - \sqrt{100 \alpha^2 + 84 \alpha \beta + 17 \beta^2}}{4 \alpha + 2 \beta}\\
t_3^-&=\frac{6 \alpha +3 \beta - \sqrt{100 \alpha^2 + 84 \alpha \beta + 17 \beta^2}}{4 \alpha + 2 \beta}\\
t_4^-&=\frac{6 \alpha +3\beta - \sqrt{100 \alpha^2 + 84 \alpha \beta + 17 \beta^2}}{2 \alpha +  \beta}\\
t_1^+&=\frac{18 \alpha +9 \beta +\sqrt{100 \alpha^2 + 84 \alpha \beta + 17 \beta^2}}{4 \alpha + 2 \beta}\\
t_2^+&=\frac{10 \alpha +5 \beta + \sqrt{100 \alpha^2 + 84 \alpha \beta + 17 \beta^2}}{4 \alpha + 2 \beta}\\
t_3^+&=\frac{6 \alpha +3 \beta +\sqrt{100 \alpha^2 + 84 \alpha \beta + 17 \beta^2}}{4 \alpha + 2 \beta}\\
t_4^+&=\frac{6 \alpha +3\beta +\sqrt{100 \alpha^2 + 84 \alpha \beta + 17 \beta^2}}{2 \alpha +  \beta}.
\end{aligned}
\end{equation}

The $4 \alpha + \beta=0$ case will be treated separately as a special case.

\begin{figure}
\includegraphics[scale=0.5]{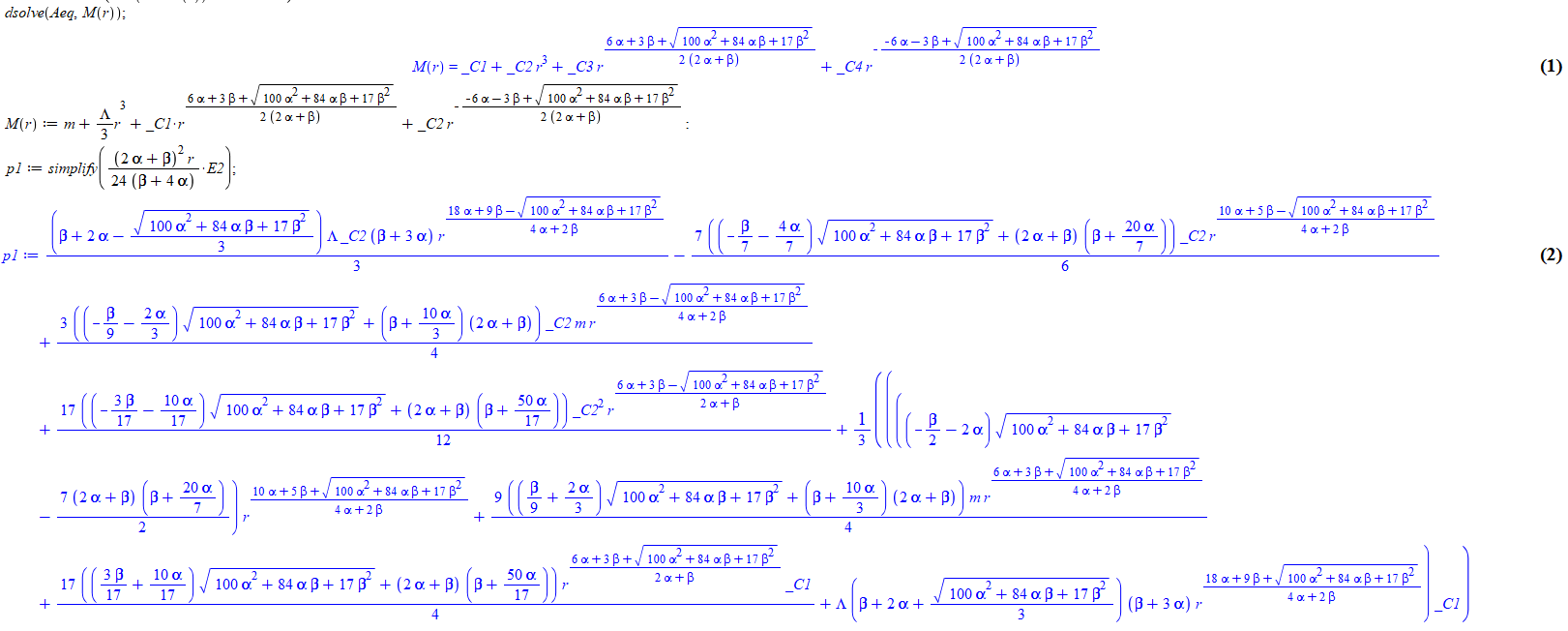}
\caption{Necessary conditions on $M$}
\label{shapeofM}
\end{figure}

\begin{remark}
We must point out that  $100 \alpha^2 + 84 \alpha \beta + 17 \beta^2$ is not necessarily positive.  We will consider first that $100 \alpha^2 + 84 \alpha \beta + 17 \beta^2 \ge 0$, before explaining that the situation is highly similar in the opposite case.
\end{remark}

Since $100 \alpha^2 + 84 \alpha \beta + 17 \beta^2 = 25 \left( 2 \alpha + \beta \right)^2 -8 \beta \left( 2 \alpha + \beta \right) = 25 \chi^2 - 8 \beta \chi$, we will favor working with $(\chi, \beta)$ instead of $(\alpha, \beta)$. We will thus write:

\begin{equation}
\label{thefandthegchi}
\begin{aligned}
f&=\frac{3}{2}+\frac{\sqrt{25 -8 \frac{\beta}{\chi} } }{2} \\
g&= \frac{3}{2}-\frac{\sqrt{25 -8 \frac{\beta}{\chi} } }{2},
\end{aligned}
\end{equation}

and
\begin{equation}
\label{theti}
\begin{aligned}
t_1^-&=\frac{9}{2} -\frac{\sqrt{25 -8 \frac{\beta}{\chi} } }{2} \\
t_2^-&=\frac{5}{2} -\frac{\sqrt{25 -8 \frac{\beta}{\chi} } }{2}\\
t_3^-&=\frac{3}{2} -\frac{\sqrt{25 -8 \frac{\beta}{\chi} } }{2}\\
t_4^-&=3 -\sqrt{25 -8 \frac{\beta}{\chi} }\\
t_1^+&=\frac{9}{2} +\frac{\sqrt{25 -8 \frac{\beta}{\chi} } }{2}\\
t_2^+&=\frac{5}{2} +\frac{\sqrt{25 -8 \frac{\beta}{\chi} } }{2}\\
t_3^+&=\frac{3}{2} +\frac{\sqrt{25 -8 \frac{\beta}{\chi} } }{2}\\
t_4^+&=3 +\sqrt{25 -8 \frac{\beta}{\chi} }.
\end{aligned}
\end{equation}

This of course works under the assumption that $\chi$ is positive. However, if $\chi \le 0$, $ \frac{ \sqrt{ 100 \alpha^2 + 84 \alpha \beta + 17 \beta^2} }{\chi} = \mathrm{sg}(\chi) \sqrt{25 -8 \frac{\beta}{\chi}}$. Up to exchanging $f$ and $g$, and the $t_i^+$ and the $t_i^-$, the coefficients remain the same.

Using these formulas, one can deduce that there exists a finite number of $[\alpha, \beta] \in \R\mathbb{P}^1$  for which  one of the $t_i^\pm$ equals another $t_j^\pm$. We present those values in the table, figure \ref{tabvalues} and detail a few representative cases:
\begin{itemize}
\item
$t_1^- =t_2^-, t_3^-$ clearly has no solution.
\item
$t_1^-=t_4^- $ is equivalent to $\frac{9}{2}-\frac{\sqrt{25 -8 \frac{\beta}{\chi} } }{2}=3 -\sqrt{25 -8 \frac{\beta}{\chi} }$ which is rephrased as $\sqrt{25 -8 \frac{\beta}{\chi} }=-3$. There are then no solutions.
\item
$t_1^-=t_1^+$ if and only if $\sqrt{25 -8 \frac{\beta}{\chi} }=0$, i.e. $\frac{\beta}{\chi}=\frac{25}{8}.$
\item
$t_1^-=t_2^+$ if and only if $\frac{9}{2} -\frac{\sqrt{25 -8 \frac{\beta}{\chi} } }{2}=\frac{5}{2} +\frac{\sqrt{25 -8 \frac{\beta}{\chi} } }{2}$, i.e. $\sqrt{25 -8 \frac{\beta}{\chi} }=2$, which means that: $\frac{\beta}{\chi}=\frac{21}{8}$.
\end{itemize}
All the other combinations fall into one of these configurations (obviously no solution, no solution because of negative squareroot, solution with null squareroot, solution with positive squareroot).
\begin{figure}
\begin{tabular}{|l|l|l|l|l|l|l|l|l|}
\hline
  \quad & $ t_1^-$ &$ t_2^-$ & $t_3^-$& $t_4^-$&$t_1^+$&$ t_2^+$&$t_3^+$&$t_4^+$ \\
\hline
$   t_1^- $& $\R\cup \{ \infty\}$ &$ \emptyset $&$\emptyset$&$\emptyset$&$\frac{25}{8}$&$\frac{21}{8}$&$2$&$3$\\
\hline
$   t_2^- $& $\emptyset$ & $\R\cup \{ \infty\}$ &$\emptyset$&$3$&$\emptyset$&$\frac{25}{8}$&$3$&$\emptyset$\\
\hline
$   t_3^- $& $\emptyset$ & $\emptyset$ &$\R\cup \{ \infty\}$&$2$&$\emptyset$&$\emptyset$&$\frac{25}{8}$&$\emptyset$\\
\hline
$   t_4^- $& $\emptyset$ & $3$ &$2$&$\R\cup \{ \infty\}$&$\emptyset$&$\frac{28}{9}$&$3$&$\frac{25}{8}$\\
\hline
$   t_1^+ $& $\frac{25}{8}$ & $\emptyset$ &$\emptyset$&$\emptyset$&$\R\cup \{ \infty\}$&$\emptyset$&$\emptyset$&$2$\\
\hline
$   t_2^+$& $\frac{21}{8}$ & $\frac{25}{8}$ &$\emptyset$&$\frac{28}{9}$&$\emptyset$&$\R\cup \{ \infty\}$&$\emptyset$&$\emptyset$\\
\hline
$   t_3^+ $& $2$& $3$ &$\frac{25}{8}$&$3$&$\emptyset$&$\emptyset$&$\R\cup \{ \infty\}$&$\emptyset$\\
\hline
$   t_4^+ $& $3$ & $\emptyset$ &$\emptyset$&$\frac{25}{8}$&$2$&$\emptyset$&$\emptyset$&$\R\cup \{ \infty\}$\\
\hline
\end{tabular}
\caption{Values of $\frac{\beta}{\chi}$ for which  $t_i^\pm=t_j^\pm$}
\label{tabvalues}
\end{figure}

Outside of those specific values, the $\left( r^{t_i^\pm} \right)$ form  a free family. Thus for the metric to be $A$ flat one must have:
\begin{equation} \label{5456}\left(-\left( \beta + 4 \alpha \right) \sqrt{ 100 \alpha^2+ 84 \alpha \beta + 17 \beta^2} + \left( 2 \alpha + \beta \right) \left( 7 \beta + 20 \alpha \right) \right) C_2=0.\end{equation}
This equation is obtained by looking at the  $r^{\frac{10 \alpha +5 \beta - \sqrt{100 \alpha^2 + 84 \alpha \beta + 17 \beta^2}}{4 \alpha + 2 \beta}}$ term  (last term of the first line in formula (2) in figure \ref{shapeofM}).
We will rephrase \eqref{5456} in terms of $\beta$ and $\chi$:
$$C_2 \left(- \left( 2-\frac{\beta}{\chi} \right) \sqrt{ 25 - 8 \frac{\beta}{\chi} } +10- 3 \frac{\beta}{\chi} \right)=0,$$
which implies that either $C_2=0$ or $\frac{\beta}{\chi}=0$, i.e. $\beta=0$.

Similarly, looking at the $r^{\frac{10 \alpha +5 \beta + \sqrt{100 \alpha^2 + 84 \alpha \beta + 17 \beta^2}}{4 \alpha + 2 \beta}}$ term  (the term of (2) in figure \ref{shapeofM} between the third and fourth line):
\begin{equation} \label{5457}\left(\left( \beta + 4 \alpha \right) \sqrt{ 100 \alpha^2+ 84 \alpha \beta + 17 \beta^2} + \left( 2 \alpha + \beta \right) \left( 7 \beta + 20 \alpha \right) \right) C_1=0.\end{equation}
We once more rephrase this as: 
$$C_1 \left( \left( 2-\frac{\beta}{\chi} \right) \sqrt{ 25 - 8 \frac{\beta}{\chi} } +10- 3 \frac{\beta}{\chi} \right)=0,$$
which implies that $C_1=0$ or $\frac{\beta}{\chi}=3$. 

Thus outside of  $\frac{\beta}{\chi} = 2, $ $\frac{25}{8}$ $\frac{21}{8}$, $\frac{28}{9}$ $0$, $3$, one must have $C_1=C_2=0$, which implies that the metric is Schwarzschild-de Sitter (or AdS). We only have to test these remaining values to conclude. For convenience, and in order to use the same Maple procedure, we will rephrase those in term of $\alpha$ and $\beta$.
We need to test the cases: $ \beta +4 \alpha=0$, $50 \alpha + 17 \beta=0$, $42 \alpha +13 \beta=0$, $56 \alpha + 19 \beta=0$, $\beta=0$, $3 \alpha + \beta=0$. Actually, this last case corresponds to the conformally invariant one, and will not be considered here (see Proposition \ref{Classificationlemma} for this configuration).

 On the Maple results displayed in figure \ref{caspart}, one can see that for $ \beta +4 \alpha=0$,  $42 \alpha +13 \beta=0$, $56 \alpha + 19 \beta=0$ one must have $C_1=C_2=0$, which is the desired result.
 In the configuration $50 \alpha + 17 \beta=0$ however, one obtains only  $C_1+C_2=0$. Nevertheless, since this corresponds to the case where $f=g$, one concludes that $M(r)= m + \frac{\Lambda}{3}r^3$ (second line of (4) in figure \ref{caspart}), which implies that the metric is indeed Schwarzschild-de Sitter (or AdS).
\begin{figure}
\includegraphics[scale=0.5]{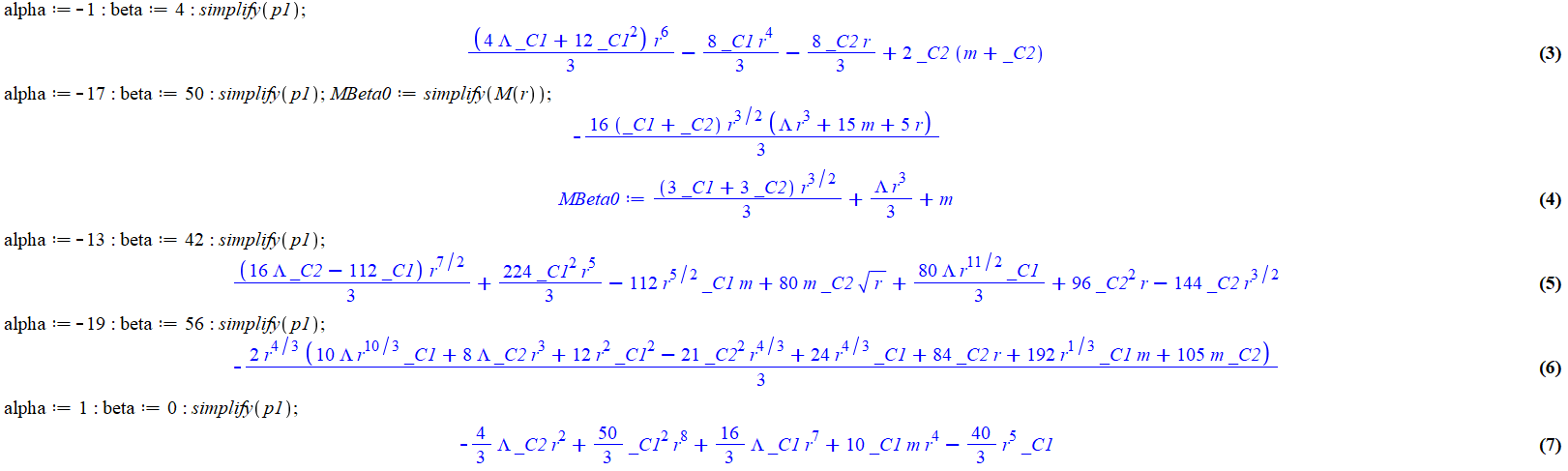}
\caption{Cases  $ \beta +4 \alpha=0$, $50 \alpha + 17 \beta=0$, $42 \alpha +13 \beta=0$, $56 \alpha + 19 \beta=0$ }
\label{caspart}
\end{figure}
 
In the final case: $\beta= 0,$ while $C_1=0$, \emph{a priori} $\Lambda=0$, $C_2 \neq 0$ is an admissible solution, corresponding to the Reissner-Nordstr\"om metric. We can check that it is indeed  a solution (see figure \ref{nordstrom}) and conclude the proof.

\begin{figure}
\includegraphics[scale=0.5]{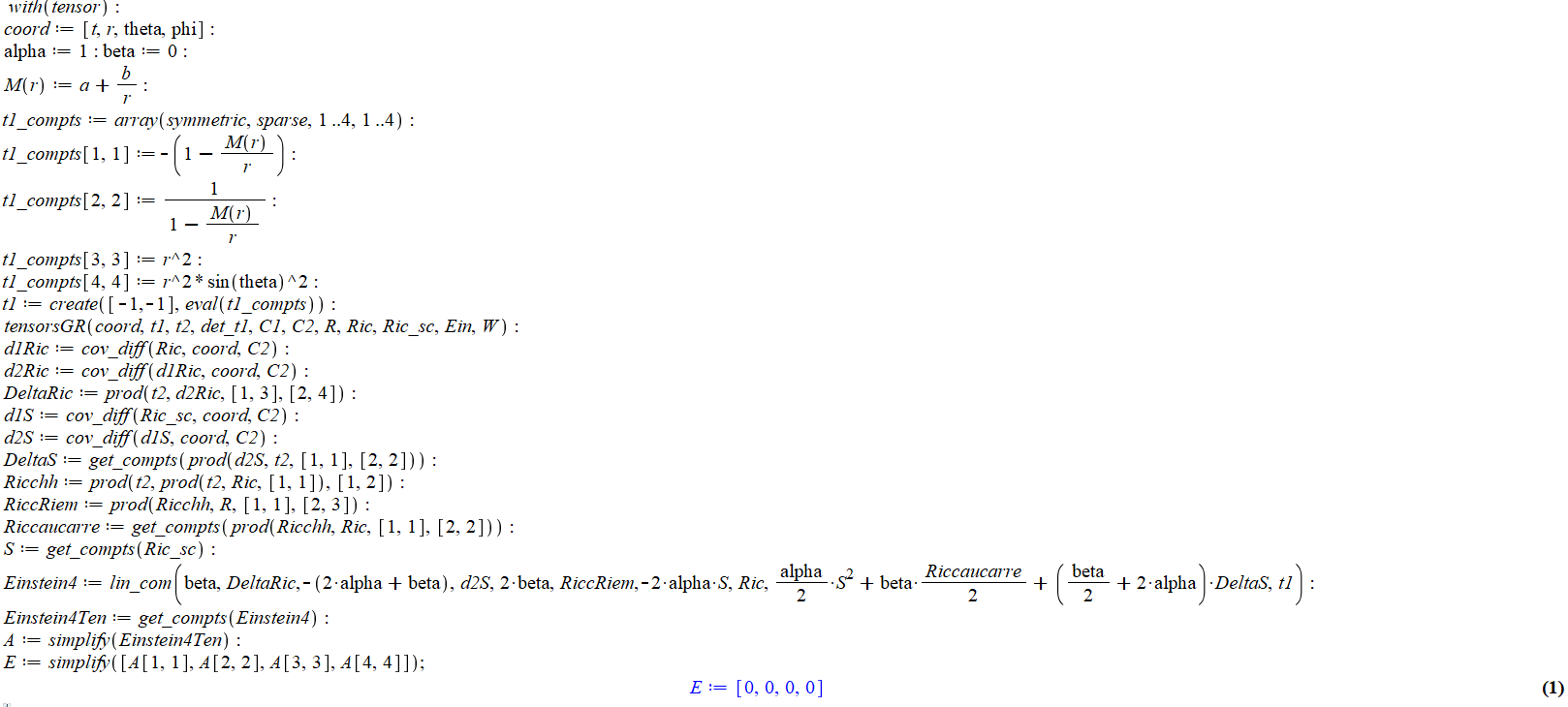}
\caption{The  Reissner-Nordstr\"om metric is $A_{\alpha, 0}$-flat. }
\label{nordstrom}
\end{figure}

Of course the above stands when $100 \alpha^2 + 84 \alpha \beta + 17 \beta^2\ge 0$. The reasoning when $100 \alpha^2 + 84 \alpha \beta + 17 \beta^2<0$ will be very similar. We will thus give a brief overview of the proof in that case: one simply has to replace $f$ and $g$ by 
$$
\begin{aligned}
f&= \frac{6 \alpha +3 \beta + i\sqrt{\left|100 \alpha^2 + 84 \alpha \beta + 17 \beta^2  \right|}}{2\left( 2 \alpha + \beta\right)} \\
g&= \frac{6 \alpha +3 \beta - i\sqrt{\left|100 \alpha^2 + 84 \alpha \beta + 17 \beta^2\right|}}{2\left( 2 \alpha + \beta\right)}.
\end{aligned}
$$
The algebraic operations will remain the same even with complex exponents, and thus $A$ will be written as a sum of (complex) powers of $r$. One simply has to replace the $t_i^{\pm}$ by:
$$
\begin{aligned}
t_1^-&=\frac{18 \alpha +9 \beta - i\sqrt{|100 \alpha^2 + 84 \alpha \beta + 17 \beta^2|}}{4 \alpha + 2 \beta}\\
t_2^-&=\frac{10 \alpha +5 \beta - i\sqrt{|100 \alpha^2 + 84 \alpha \beta + 17 \beta^2|}}{4 \alpha + 2 \beta}\\
t_3^-&=\frac{6 \alpha +3 \beta - i\sqrt{|100 \alpha^2 + 84 \alpha \beta + 17 \beta^2|}}{4 \alpha + 2 \beta}\\
t_4^-&=\frac{6 \alpha +3\beta - i\sqrt{|100 \alpha^2 + 84 \alpha \beta + 17 \beta^2|}}{2 \alpha +  \beta}\\
t_1^+&=\frac{18 \alpha +9 \beta +i\sqrt{|100 \alpha^2 + 84 \alpha \beta + 17 \beta^2|}}{4 \alpha + 2 \beta}\\
t_2^+&=\frac{10 \alpha +5 \beta + i\sqrt{|100 \alpha^2 + 84 \alpha \beta + 17 \beta^2|}}{4 \alpha + 2 \beta}\\
t_3^+&=\frac{6 \alpha +3 \beta +i\sqrt{|100 \alpha^2 + 84 \alpha \beta + 17 \beta^2|}}{4 \alpha + 2 \beta}\\
t_4^+&=\frac{6 \alpha +3\beta +i\sqrt{|100 \alpha^2 + 84 \alpha \beta + 17 \beta^2|}}{2 \alpha +  \beta}.
\end{aligned}
$$

In this case, the $t_i^{\pm}$ cannot interfere and thus the $r^{t_i^{\pm}}$ form a free family. One can then, \emph{mutatis mutandis}, look at \eqref{5456} and \eqref{5457} in the same manner as before, and conclude that $C_1=C_2=0$, and thus that the metric is Schwarzschild-de Sitter (or AdS).

\addcontentsline{toc}{section}{References}
\printbibliography

\end{document}